\documentclass[oneside,american]{amsart}
\usepackage[T1]{fontenc}
\usepackage[latin9]{inputenc}
\usepackage{amstext}
\usepackage{amsthm}
\usepackage{amssymb}

\makeatletter

\newcommand{\noun}[1]{\textsc{#1}}

\numberwithin{equation}{section}
\numberwithin{figure}{section}
\theoremstyle{plain}
\newtheorem{thm}{\protect\theoremname}[section]
\theoremstyle{definition}
\newtheorem{defn}[thm]{\protect\definitionname}
\theoremstyle{remark}
\newtheorem*{rem*}{\protect\remarkname}
\theoremstyle{plain}
\newtheorem{prop}[thm]{\protect\propositionname}
\theoremstyle{plain}
\newtheorem{lem}[thm]{\protect\lemmaname}
\theoremstyle{plain}
\newtheorem*{prop*}{\protect\propositionname}
\theoremstyle{remark}
\newtheorem{rem}[thm]{\protect\remarkname}
\theoremstyle{plain}
\newtheorem{cor}[thm]{\protect\corollaryname}
\theoremstyle{plain}
\newtheorem*{thm*}{\protect\theoremname}

\usepackage[pdftex,pdfpagelabels,bookmarks,hyperindex,hyperfigures]{hyperref}

\usepackage{thmtools}

\makeatother

\usepackage{babel}
\providecommand{\corollaryname}{Corollary}
\providecommand{\definitionname}{Definition}
\providecommand{\lemmaname}{Lemma}
\providecommand{\propositionname}{Proposition}
\providecommand{\remarkname}{Remark}
\providecommand{\theoremname}{Theorem}

\begin{document}
\title{Dimension of Furstenberg measures on $\mathbb{CP}^{1}$}
\author{\noindent Ariel Rapaport and Haojie Ren}
\thanks{This research was supported by the Israel Science Foundation (grant
No. 619/22). AR received support from the Horev Fellowship at the
Technion -- Israel Institute of Technology.}
\begin{abstract}
Let $\theta$ be a finitely supported probability measure on $\mathrm{SL}(2,\mathbb{C})$,
and suppose that the semigroup generated by $\mathcal{G}:=\mathrm{supp}(\theta)$
is strongly irreducible and proximal. Let $\mu$ denote the Furstenberg
measure on $\mathbb{CP}^{1}$ associated to $\theta$. Assume further
that no generalized circle is fixed by all Möbius transformations
corresponding to elements of $\mathcal{G}$, and that $\mathcal{G}$
satisfies a mild Diophantine condition. Under these assumptions, we
prove that $\dim\mu=\min\left\{ 2,h_{\mathrm{RW}}/\left(2\chi\right)\right\} $,
where $h_{\mathrm{RW}}$ and $\chi$ denote the random walk entropy
and Lyapunov exponent associated to $\theta$, respectively.

Since our result expresses $\dim\mu$ in terms of the random walk
entropy rather than the Furstenberg entropy, and relies only on a
mild Diophantine condition as a separation assumption, we are forced
to directly confront difficulties arising from the ambient space $\mathbb{CP}^{1}$
having real dimension $2$ rather than $1$. Moreover, our analysis
takes place in a projective, contracting-on-average setting. This
combination of features introduces significant challenges and requires
genuinely new ideas.
\end{abstract}

\maketitle

\section{\label{sec:Introduction-and-the main result}Introduction and the
main result}

\subsection{\label{subsec:Setup-and-background}Setup and background}

Set $\mathrm{G}:=\mathrm{SL}(2,\mathbb{C})$, and write $\mathbb{C}_{\infty}:=\mathbb{C}\cup\{\infty\}$
for the Riemann sphere. Given $g\in\mathrm{G},$ let $\varphi_{g}:\mathbb{C}_{\infty}\rightarrow\mathbb{C}_{\infty}$
denote the corresponding Möbius transformation. That is,
\[
\varphi_{g}(z)=\frac{az+b}{cz+d}\text{ for }z\in\mathbb{C}_{\infty},\text{ where }g=\left(\begin{array}{cc}
a & b\\
c & d
\end{array}\right).
\]
The action of $\mathrm{G}$ on $\mathbb{C}_{\infty}$ via Möbius transformations
is one of the most classical examples of a Lie group action on a compact
space. In this paper, under mild assumptions, we compute the dimension
of stationary measures on $\mathbb{C}_{\infty}$ associated to finitely
supported probability measures on $\mathrm{G}$.

Write $\mathbb{CP}^{1}:=\left\{ z\mathbb{C}\::\:0\ne z\in\mathbb{C}^{2}\right\} $
for the complex projective line, and define $\psi:\mathbb{CP}^{1}\rightarrow\mathbb{C}_{\infty}$
by
\[
\psi\left(z\mathbb{C}\right)=\begin{cases}
z_{1}/z_{2} & \text{ if }z_{2}\ne0\\
\infty & \text{ if }z_{2}=0
\end{cases}\text{ for all }(z_{1},z_{2})=z\in\mathbb{C}^{2}\setminus\{0\}.
\]
The group $\mathrm{G}$ acts naturally on $\mathbb{CP}^{1}$ by $g\cdot z\mathbb{C}:=gz\mathbb{C}$,
and the map $\psi$ is an isomorphism between this action and the
Möbius action of $\mathrm{G}$ on $\mathbb{C}_{\infty}$.

We equip $\mathbb{CP}^{1}$ with the metric given by 
\[
d_{\mathbb{CP}^{1}}\left(z\mathbb{C},w\mathbb{C}\right):=\frac{1}{\Vert z\Vert\Vert w\Vert}\left|\det\left(\begin{array}{cc}
z_{1} & w_{1}\\
z_{2} & w_{2}
\end{array}\right)\right|
\]
for nonzero vectors $z=(z_{1},z_{2})$ and $w=(w_{1},w_{2})$ in $\mathbb{C}^{2}$.
One readily checks that $d_{\mathbb{CP}^{1}}$ is bi-Lipschitz equivalent
to any Riemannian distance function on $\mathbb{CP}^{1}$.

Throughout the paper, let $\Lambda$ be a finite nonempty index set,
fix a collection $\mathcal{G}=\left\{ g_{i}\right\} _{i\in\Lambda}\subset\mathrm{G}$,
and fix a positive probability vector $p=(p_{i})_{i\in\Lambda}$.
Write $\mathrm{S}_{\mathcal{G}}$ for the subsemigroup of $\mathrm{G}$
generated by $\mathcal{G}$. We shall always assume that $\mathrm{S}_{\mathcal{G}}$
is strongly irreducible and proximal. Strong irreducibility means
that the action of $\mathrm{S}_{\mathcal{G}}$ on $\mathbb{CP}^{1}$
has no finite trajectory, while proximality means that $\mathrm{S}_{\mathcal{G}}$
is unbounded with respect to the operator norm $\Vert\cdot\Vert_{\mathrm{op}}$.

For a metric space $X$, denote by $\mathcal{M}(X)$ the collection
of compactly supported Borel probability measures on $X$. Under the
above assumptions, it is well known that there exists a unique $\mu\in\mathcal{M}\left(\mathbb{CP}^{1}\right)$
satisfying $\mu=\sum_{i\in\Lambda}p_{i}\cdot g_{i}\mu$, where $g_{i}\mu$
denotes the pushforward of $\mu$ via the map $z\mathbb{C}\mapsto g_{i}z\mathbb{C}$.
In other words, $\mu$ is the unique element of $\mathcal{M}\left(\mathbb{CP}^{1}\right)$
that is stationary with respect to $\sum_{i\in\Lambda}p_{i}\delta_{g_{i}}\in\mathcal{M}(\mathrm{G})$,
where $\delta_{g_{i}}$ is the Dirac mass at $g_{i}$. The measure
$\mu$ is called the Furstenberg measure associated to $\mathcal{G}$
and $p$. Furstenberg measures play a central role in the study of
the asymptotic behavior of random matrix products (see \cite{BQ,BL}),
and their dimension theory is an important strand of research in fractal
geometry (see, e.g., \cite{Bourgain-FurMeas,HS,LedLesGAFA}).

It follows from the recent work of Ledrappier and Lessa \cite{MR4845880}
that $\mu$ is exact dimensional. That is, there exists a number $\dim\mu$,
called the dimension of $\mu$, such that
\[
\underset{r\downarrow0}{\lim}\frac{\log\mu\left(B(z\mathbb{C},r)\right)}{\log r}=\dim\mu\text{ for }\mu\text{-a.e. }z\mathbb{C},
\]
where $B(z\mathbb{C},r)$ is the closed ball with center $z\mathbb{C}$
and radius $r$. In Appendix \ref{sec:appendix}, we deduce from \cite{rapaport2020exact}
the exact dimensionality of $\mu$, together with a Ledrappier--Young-type
formula for its dimension.

In our main result, we compute $\dim\mu$ in terms of the random walk
entropy and the Lyapunov exponent, which are fundamental dynamical
quantities. Write $\beta:=p^{\mathbb{N}}$ for the Bernoulli measure
on $\Lambda^{\mathbb{N}}$ corresponding to $p$, and denote by $\chi$
the Lyapunov exponent associated to $\mathcal{G}$ and $p$. That
is,
\begin{equation}
\underset{n\to\infty}{\lim}\frac{1}{n}\log\Vert g_{\omega_{0}}...g_{\omega_{n-1}}\Vert_{\mathrm{op}}=\chi\text{ for }\beta\text{-a.e. }\omega\in\Lambda^{\mathbb{N}},\label{eq:chi =00003D lim a.s.}
\end{equation}
where we always use $2$ as the base of the logarithm. Since $\mathrm{S}_{\mathcal{G}}$
is strongly irreducible and proximal, we have $\chi>0$ (see \cite[Corollary 4.32]{BQ}).

Denote by $h_{\mathrm{RW}}$ the random walk entropy associated to
$\mathcal{G}$ and $p$. That is,
\begin{equation}
h_{\mathrm{RW}}:=\underset{n\to\infty}{\lim}\frac{1}{n}H\left(X_{1}...X_{n}\right)=\underset{n\ge1}{\inf}\frac{1}{n}H\left(X_{1}...X_{n}\right),\label{eq:def of h_RW}
\end{equation}
where $X_{1},X_{2},...$ are i.i.d. $\mathrm{G}$-valued random elements
with $\mathbb{P}\left\{ X_{1}=g_{i}\right\} =p_{i}$ for each $i\in\Lambda$,
and $H\left(X_{1}...X_{n}\right)$ denotes the Shannon entropy of
the discrete random element $X_{1}...X_{n}$. The existence of the
limit and the second equality in (\ref{eq:def of h_RW}) follow from
subadditivity. Writing $H(p)$ for the entropy of $p$, note that
$h_{\mathrm{RW}}=H(p)$ if and only if $\mathcal{G}$ generates a
free semigroup.

By \cite[Proposition 10.2]{MR1449135},
\[
\dim\mu=\inf\left\{ \dim_{H}E\::\:E\subset\mathbb{CP}^{1}\text{ is a Borel set with }\mu(E)>0\right\} ,
\]
where $\dim_{H}E$ denotes the Hausdorff dimension of $E$. Thus,
as $\mathbb{CP}^{1}$ has dimension $2$ as a real manifold, $\dim\mu\le2$.
A second, less obvious upper bound for the dimension of $\mu$, of
a dynamical nature, arises from the aforementioned Ledrappier--Young-type
formula. Namely, using that formula, it is easy to show (see Lemma
\ref{lem:h/2chi is ub for dim mu}) that $\dim\mu\le h_{\mathrm{RW}}/\left(2\chi\right)$.
This bound can also be deduced from \cite[Theorem 1.2]{MR4845880}.

Under the additional assumption that $\mathcal{G}$ is contained in
a discrete subgroup of $\mathrm{G}$, it follows from the work of
Ledrappier \cite{Ledrappier1983} from the early 1980s that $\dim\mu=h_{\mathrm{RW}}/\left(2\chi\right)$.\footnote{In \cite{Ledrappier1983}, a slightly different notion of dimension
was used; however, since $\mu$ is exact dimensional, it coincides
with the usual one.} Motivated by this, and by important developments from the last decade
or so in the dimension theory of stationary fractal measures (see,
e.g., \cite{BHR,Ho1,HS,Var-Bernoulli}), it is expected that, in the
absence of obvious algebraic obstructions, the dimension of $\mu$
should equal its maximal possible value given the above upper bounds.
That is, it is expected that $\dim\mu=\min\left\{ 2,h_{\mathrm{RW}}/\left(2\chi\right)\right\} $.
In our main result, we establish this equality under mild assumptions,
thereby substantially relaxing the aforementioned discreteness assumption.

When $\mathcal{G}\subset\mathrm{SL}(2,\mathbb{R})$, the dimension
of $\mu$ was computed by Hochman and Solomyak \cite{HS}. To state
their result, and ours, we need the following definition. Let $d_{\mathrm{G}}$
denote the Riemannian distance function induced by a left-invariant
Riemannian metric on $\mathrm{G}$. Given a word $i_{1}...i_{n}=u\in\Lambda^{n}$,
write $g_{u}:=g_{i_{1}}...g_{i_{n}}$.
\begin{defn}
\label{def:Diophantine prop}We say that $\mathcal{G}$ is Diophantine
if there exists $c>0$ such that for every $n\ge1$,
\begin{equation}
d_{\mathrm{G}}\left(g_{u_{1}},g_{u_{2}}\right)\ge c^{n}\text{ for all }u_{1},u_{2}\in\Lambda^{n}\text{ with }g_{u_{1}}\ne g_{u_{2}}.\label{eq:prop in Diophantine prop}
\end{equation}
We say that $\mathcal{G}$ is weakly Diophantine if there exists $c>0$
such that (\ref{eq:prop in Diophantine prop}) holds for infinitely
many $n\ge1$.
\end{defn}

\begin{rem*}
As pointed out in \cite[Section 2.3]{HS}, Definition \ref{def:Diophantine prop}
is independent of the specific choice of left-invariant Riemannian
metric from which $d_{\mathrm{G}}$ is induced.
\end{rem*}
\begin{rem*}
We say that $\mathcal{G}$ is defined by algebraic parameters if the
entries of $g_{i}$ are algebraic numbers for each $i\in\Lambda$.
As shown in \cite[Lemma 6.1]{HS}, $\mathcal{G}$ is Diophantine whenever
it is defined by algebraic parameters.
\end{rem*}
\begin{rem*}
The Diophantine condition is substantially milder than the discreteness
assumption from \cite{Ledrappier1983}. Indeed, when $\mathcal{G}$
is contained in a discrete subgroup of $\mathrm{G}$, there exists
$c>0$ such that (\ref{eq:prop in Diophantine prop}) holds for all
$n\ge1$ with $c$ in place of $c^{n}$, which is, of course, much
more restrictive.
\end{rem*}
The main result of \cite{HS} states that $\dim\mu=\min\left\{ 1,h_{\mathrm{RW}}/\left(2\chi\right)\right\} $
whenever $\mathcal{G}\subset\mathrm{SL}(2,\mathbb{R})$, $\mathrm{S}_{\mathcal{G}}$
is strongly irreducible and proximal, and $\mathcal{G}$ is Diophantine.
It appears that the proof in \cite{HS} still applies if $\mathcal{G}$
is assumed to be weakly Diophantine rather than Diophantine. Moreover,
it is straightforward to relax the condition $\mathcal{G}\subset\mathrm{SL}(2,\mathbb{R})$
to the assumption that $\mathcal{G}$ can be conjugated into the subgroup\footnote{Note that $\mathrm{Stab}_{\mathrm{G}}\left(\mathbb{R}_{\infty}\right)$
equals the group generated by $\mathrm{SL}\left(2,\mathbb{R}\right)$
and the matrix $\mathrm{diag}\left(i,-i\right)\in\mathrm{G}$.}
\[
\mathrm{Stab}_{\mathrm{G}}\left(\mathbb{R}_{\infty}\right):=\left\{ g\in\mathrm{G}\::\:\varphi_{g}\left(\mathbb{R}_{\infty}\right)=\mathbb{R}_{\infty}\right\} ,
\]
where $\mathbb{R}_{\infty}:=\mathbb{R}\cup\{\infty\}$. The purpose
of the present paper is to treat the complementary case, namely when
such a conjugation is not possible.

\subsection{The main result}

We continue to use the setup and notation from the previous subsection.
For each $i\in\Lambda$, write $\varphi_{i}:=\varphi_{g_{i}}$. A
subset $C\subset\mathbb{C}_{\infty}$ is called a generalized circle
if either
\[
C=\left\{ z\in\mathbb{C}\::\:|z-z_{0}|=r\right\} \text{ for some }z_{0}\in\mathbb{C},r>0,
\]
or
\[
C=\left\{ z_{0}+tz_{1}\::\:t\in\mathbb{R}\right\} \cup\{\infty\}\text{ for some }z_{0},z_{1}\in\mathbb{C},z_{1}\ne0.
\]
We say that $\mathrm{S}_{\mathcal{G}}$ fixes a generalized circle
if there exists such a $C$ with $\varphi_{i}(C)=C$ for all $i\in\Lambda$.
The following theorem is our main result.
\begin{thm}
\label{thm:main result}Suppose that $\mathrm{S}_{\mathcal{G}}$ is
strongly irreducible, proximal, and does not fix a generalized circle.
Assume moreover that $\mathcal{G}$ is weakly Diophantine. Then,
\begin{equation}
\dim\mu=\min\left\{ 2,\frac{h_{\mathrm{RW}}}{2\chi}\right\} .\label{eq:equality to prove in main thm}
\end{equation}
\end{thm}

Let us make some remarks regarding the assumptions appearing in Theorem
\ref{thm:main result}. First note that $\mathrm{S}_{\mathcal{G}}$
is strongly irreducible, proximal, and does not fix a generalized
circle if and only if $\mathrm{S}_{\mathcal{G}}$ is dense in $\mathrm{G}$
with respect to the Zariski topology generated by the real polynomial
functions (see Section \ref{subsec:Zariski-density-of S_G}). We have
chosen to formulate the theorem in terms of these three conditions
rather than directly in terms of Zariski density, as this makes the
statement more transparent.

We now discuss the individual assumptions in more detail. The strong
irreducibility and proximality assumptions are standard in the theory
of random matrix products. When $\mathrm{S}_{\mathcal{G}}$ is nonproximal,
its closure is a compact Lie group, and the elements of $\mathcal{M}\left(\mathbb{CP}^{1}\right)$
that are stationary and ergodic with respect to $\theta:=\sum_{i\in\Lambda}p_{i}\delta_{g_{i}}$
are $\mathrm{S}_{\mathcal{G}}$-invariant smooth probability measures
supported on trajectories of the closure of $\mathrm{S}_{\mathcal{G}}$.

When $\mathrm{S}_{\mathcal{G}}$ is reducible, i.e. when its action
on $\mathbb{CP}^{1}$ has a common fixed point, one can, after conjugation,
assume that $\varphi_{i}(\infty)=\infty$ for each $i\in\Lambda$.
Hence, this case reduces to the study of self-similar measures on
$\mathbb{R}^{2}$. The strictly contracting case was studied by Hochman
\cite{Ho}, while the general contracting-on-average case was recently
addressed by Kittle and Kogler \cite{kittle2025dimension}.

When $\mathrm{S}_{\mathcal{G}}$ is proximal and irreducible but not
strongly irreducible, it is not difficult to see that there exist
distinct $z\mathbb{C},w\mathbb{C}\in\mathbb{CP}^{1}$ such that $\frac{1}{2}\left(\delta_{z\mathbb{C}}+\delta_{w\mathbb{C}}\right)$
is the unique $\theta$-stationary measure in $\mathcal{M}\left(\mathbb{CP}^{1}\right)$.
In particular, in this case the stationary measure is atomic, and
hence zero-dimensional.

When $\mathrm{S}_{\mathcal{G}}$ is strongly irreducible, proximal,
and fixes a generalized circle $C\subset\mathbb{C}_{\infty}$, the
measure $\mu$ is supported on the closed curve $\psi^{-1}(C)$, where
$\psi$ is the map defined at the beginning of Section \ref{subsec:Setup-and-background}.
Consequently, $\dim\mu\le1$, and (\ref{eq:equality to prove in main thm})
fails whenever $h_{\mathrm{RW}}/\left(2\chi\right)>1$. On the other
hand, in this case $\mathcal{G}$ can be conjugated into $\mathrm{Stab}_{\mathrm{G}}\left(\mathbb{R}_{\infty}\right)$,
and, as noted above, a slight extension of \cite{HS} yields $\dim\mu=\min\left\{ 1,h_{\mathrm{RW}}/\left(2\chi\right)\right\} $.

Finally, it is expected that Theorem \ref{thm:main result} should
remain valid even without the weakly Diophantine assumption. Unfortunately,
this lies well beyond our current reach. Indeed, such a statement
has not been achieved even in the considerably simpler setting of
self-similar measures on the real line, where its validity is regarded
as one of the major open problems in fractal geometry (see \cite{MR3966837,Var_ICM}).

On the other hand, the weak Diophantine condition is quite mild. Firstly,
as pointed out above, $\mathcal{G}$ is always Diophantine whenever
it is defined by algebraic parameters. Moreover, as suggested by the
work of Solomyak and Takahashi \cite{MR4300232} in the real case,
given a well-behaved parametric family of finite subsets of $\mathrm{G}$,
it should be possible to verify the Diophantine property outside a
small exceptional set of parameters. We do not pursue this direction
here, however, leaving it open for further research.

\subsection{Additional related results}

The dimension of Furstenberg measures on the real projective plane
$\mathbb{RP}^{2}$ was recently studied by Li, Pan, and Xu \cite{li2024FurstOnRP2}
and by Jurga \cite{jurga2023FurstOnRP2}. In both works, the results
were applied to settle a folklore conjecture concerning the dimension
of the Rauzy gasket, a well-known fractal arising in dynamical systems.
Let $\theta\in\mathcal{M}\left(\mathrm{SL}\left(3,\mathbb{R}\right)\right)$
be finitely supported, suppose that the semigroup generated by $\mathrm{supp}(\theta)$
is Zariski dense in $\mathrm{SL}\left(3,\mathbb{R}\right)$, and let
$\mu'\in\mathcal{M}\left(\mathbb{RP}^{2}\right)$ denote the Furstenberg
measure associated to $\theta$.

Assuming $\mathrm{supp}(\theta)$ is Diophantine, the dimension of
$\mu'$ was computed in \cite{li2024FurstOnRP2} in terms of the Furstenberg
entropy (see \cite[Eq.  (2.81)]{li2024FurstOnRP2} for the definition)
and the Lyapunov exponents. In the presence of substantial overlaps
between the supports of the measures $\left\{ g\mu':g\in\mathrm{supp}(\theta)\right\} $,
the Furstenberg entropy is usually difficult to compute. Moreover,
the Furstenberg entropy is always bounded above by the random walk
entropy. Hence, it is advantageous to compute $\dim\mu'$ in terms
of the latter rather than the former.\footnote{In the setup studied in the present paper, the dimension of $\mu$
was computed in terms of the Furstenberg entropy in \cite{Ledrappier1983},
while only assuming strong irreducibility and proximality.}

Assuming $\mathrm{supp}(\theta)$ consists of matrices with strictly
positive entries and satisfies the strong open set condition (SOSC),
the dimension of $\mu'$ was computed in \cite{jurga2023FurstOnRP2}
in terms of the Shannon entropy of $\theta$ and the Lyapunov exponents.
Roughly speaking, the SOSC requires that the supports of the measures
$\left\{ g\mu':g\in\mathrm{supp}(\theta)\right\} $ be nearly disjoint.

Both of the above results are obtained by computing the dimension
of projections of $\mu'$ onto (typical) one-dimensional projective
subspaces, and then applying the Ledrappier--Young formula from \cite{LedLesGAFA,rapaport2020exact}.
This approach suffices because Furstenberg entropy is used in place
of random walk entropy in \cite{li2024FurstOnRP2}, and because of
the SOSC assumption in \cite{jurga2023FurstOnRP2}. Consequently,
in both proofs most of the analysis is carried out in a one-dimensional
setting, and in this sense the fact that $\mathbb{RP}^{2}$ is two-dimensional,
which causes significant difficulties, is not confronted directly.

A measure in $\mathcal{M}\left(\mathbb{R}^{d}\right)$ is called self-affine
(resp. self-similar) if it is stationary with respect to a finitely
supported probability measure on the affine (resp. similarity) group
of $\mathbb{R}^{d}$. The dimension of self-affine and self-similar
measures was studied in \cite{Ho,HR,kittle2025dimension,Rap_SA_Rd},
while directly addressing challenges posed by high dimensionality.
However, in this setting the action is affine rather than projective,
which avoids some of the major difficulties present in the projective
case.

In the present work, we compute $\dim\mu$ in terms of the random
walk entropy, while requiring only the weakly Diophantine condition
as a separation assumption. This forces us to confront directly the
difficulties arising from the fact that $\mathbb{CP}^{1}$ has real
dimension $2$ rather than $1$. Furthermore, our analysis takes place
in a projective, contracting-on-average setting. As we explain in
the next subsection, this combination of features introduces significant
new challenges and requires genuinely new ideas.

\subsection{\label{subsec:About-the-proof}About the proof}

In this subsection we provide a general outline of our proof of Theorem
\ref{thm:main result}. Everything discussed here will be repeated
rigorously in later parts of the paper. In what follows we always
assume that $\mathrm{S}_{\mathcal{G}}$ is strongly irreducible, proximal,
and does not fix a generalized circle.

As in many other developments in fractal geometry in recent years,
the key ingredient of our proof is a statement ensuring a substantial
increase of entropy under convolution. This approach was initiated
by Hochman \cite{Ho1} in his seminal work on the dimension of exponentially
separated self-similar measures on $\mathbb{R}$.

In what follows we use standard notation for entropy; see Section
\ref{subsec:Entropy} for the relevant basic definitions. Given $n\ge0$,
write $\mathcal{D}_{n}^{\mathbb{CP}^{1}}$ (resp. $\mathcal{D}_{n}^{\mathrm{G}}$)
for the level-$n$ dyadic-like partition of $\mathbb{CP}^{1}$ (resp.
$\mathrm{G}$), defined later in Section \ref{subsec:Dyadic-partitions}.
We omit the superscript $\mathbb{CP}^{1}$ (resp. $\mathrm{G}$) when
it is clear from the context. Given $\theta\in\mathcal{M}(\mathrm{G})$
and $\xi\in\mathcal{M}\left(\mathbb{CP}^{1}\right)$, write $\theta.\xi\in\mathcal{M}\left(\mathbb{CP}^{1}\right)$
for the pushforward of $\theta\times\xi$ via the action map $(g,z\mathbb{C})\mapsto gz\mathbb{C}$.
For $r>0$, denote by $B(1_{\mathrm{G}},r)$ the closed ball in $\mathrm{G}$
with center $1_{\mathrm{G}}$, the identity element of $\mathrm{G}$,
and radius $r$. The following theorem is our entropy increase result. 
\begin{thm}
\label{thm:ent inc}Suppose that $\dim\mu<2$. Then there exists $0<r<1$
such that for every $\epsilon>0$, there exists $\delta=\delta(\epsilon)>0$
so that $\frac{1}{n}H(\theta.\mu,\mathcal{D}_{n})>\dim\mu+\delta$
for all $n\ge N(\epsilon)\ge1$ and $\theta\in\mathcal{M}\left(B(1_{\mathrm{G}},r)\right)$
with $\frac{1}{n}H(\theta,\mathcal{D}_{n})\ge\epsilon$.
\end{thm}

\begin{rem*}
Since $\mu$ is exact dimensional, $\frac{1}{n}H(\mu,\mathcal{D}_{n})\approx\dim\mu$
for large $n\ge1$. Hence, Theorem \ref{thm:ent inc} guarantees that
the entropy of the convolution $\theta.\mu$ is substantially larger
than the entropy of $\mu$ whenever $\dim\mu<2$ and $\theta\in\mathcal{M}\left(B(1_{\mathrm{G}},r)\right)$
has nonnegligible entropy.
\end{rem*}
\begin{rem*}
It is not difficult to deduce a version of Theorem \ref{thm:ent inc}
that is valid for any $r>0$ (in such a version $\delta$ would also
depend on $r$). However, we do not need this stronger form, and assuming
$r$ is some absolute small constant slightly simplifies the proof.
\end{rem*}
The argument for deducing Theorem \ref{thm:main result} from Theorem
\ref{thm:ent inc}, which we now briefly describe, is based on an
approach developed in \cite{HR} in the self-affine setting. Suppose
that $\mathcal{G}$ is weakly Diophantine, and assume by contradiction
that $\dim\mu<\min\left\{ 2,h_{\mathrm{RW}}/\left(2\chi\right)\right\} $.
Let $L:\Lambda^{\mathbb{N}}\rightarrow\mathbb{CP}^{1}$ denote the
Furstenberg boundary map associated to $\mathcal{G}$ and $p$ (see
Section \ref{subsec:results-from-rand-prod-of-mat}), and let $\left\{ \beta_{\omega}\right\} _{\omega\in\Lambda^{\mathbb{N}}}\subset\mathcal{M}(\Lambda^{\mathbb{N}})$
denote the disintegration of $\beta:=p^{\mathbb{N}}$ with respect
to $L^{-1}\mathcal{B}_{\mathbb{CP}^{1}}$, where $\mathcal{B}_{\mathbb{CP}^{1}}$
is the Borel $\sigma$-algebra of $\mathbb{CP}^{1}$. Given $n\ge1$,
let $\Pi_{n}:\Lambda^{\mathbb{N}}\rightarrow\mathrm{G}$ be defined
by $\Pi_{n}(\omega)=g_{\omega_{0}}...g_{\omega_{n-1}}$ for $\omega\in\Lambda^{\mathbb{N}}$.

Using $\dim\mu<h_{\mathrm{RW}}/\left(2\chi\right)$, the Ledrappier--Young
formula established in \cite{rapaport2020exact}, and the fact that
$\mathcal{G}$ is weakly Diophantine, it is not difficult to show
that there exist $\epsilon>0$ and $M>1$ such that for infinitely
many $n\ge1$,
\begin{equation}
\beta\left\{ \omega\::\:\frac{1}{n}H\left(\Pi_{n}\beta_{\omega},\mathcal{D}_{Mn}\right)>\epsilon\right\} >\epsilon,\label{eq:ent of slices > epsilon in intro}
\end{equation}
where $\Pi_{n}\beta_{\omega}\in\mathcal{M}\left(\mathrm{G}\right)$
denotes the pushforward of $\beta_{\omega}$ via $\Pi_{n}$. Moreover,
by the exact dimensionality of $\mu$, for large $n\ge1$ we have
\begin{equation}
\dim\mu\approx\frac{1}{Mn}H\left(\mu,\mathcal{D}_{\left(M+2\chi\right)n}\mid\mathcal{D}_{2\chi n}\right).\label{eq:dim mu approx cond ent in intro}
\end{equation}
With some additional work, one can now use (\ref{eq:ent of slices > epsilon in intro})
and (\ref{eq:dim mu approx cond ent in intro}), the decomposition
$\mu=\int\left(\Pi_{n}\beta_{\omega}\right).\mu\:d\beta(\omega)$,
the concavity of conditional entropy, the inequality $\dim\mu<2$,
and Theorem \ref{thm:ent inc}, to obtain the desired contradiction.
Note, however, that the measures $\Pi_{n}\beta_{\omega}$ are usually
supported far from the identity of $\mathrm{G}$, and there is no
reason to expect that $\mathrm{diam}\left(\mathrm{supp}\left(\Pi_{n}\beta_{\omega}\right)\right)<r$,
where $r>0$ is the constant appearing in Theorem \ref{thm:ent inc}.
To apply our entropy increase result, we therefore need to `chop'
the measures $\Pi_{n}\beta_{\omega}$ into $o_{\omega}(1)$ pieces
of diameter at most $r$, and translate these pieces into $B(1_{\mathrm{G}},r)$.

For the remainder of this subsection we discuss the proof of Theorem
\ref{thm:ent inc}. First, we need some additional notation. Given
an $\mathbb{R}$-linear subspace $V$ of $\mathbb{C}$, denote by
$\pi_{V}:\mathbb{C}\rightarrow\mathbb{C}$ the orthogonal projection
onto $V$, where $\mathbb{C}$ is identified with $\mathbb{R}^{2}$.
For $n\ge0$, let $\mathcal{D}_{n}^{\mathbb{C}}$ be the level-$n$
dyadic partition of $\mathbb{C}$, again identifying $\mathbb{C}$
with $\mathbb{R}^{2}$. We extend this to a partition of $\mathbb{C}_{\infty}$
by setting $\mathcal{D}_{n}^{\mathbb{C}_{\infty}}:=\mathcal{D}_{n}^{\mathbb{C}}\cup\left\{ \{\infty\}\right\} $.
As before, we omit the superscripts $\mathbb{C}$ and $\mathbb{C}_{\infty}$
when they are clear from the context. For $\xi\in\mathcal{M}\left(\mathbb{C}_{\infty}\right)$
and $z\in\mathbb{C}_{\infty}$ with $\xi\left(\mathcal{D}_{n}(z)\right)>0$,
write $\xi_{z,n}:=\xi_{\mathcal{D}_{n}(z)}$. Here $\mathcal{D}_{n}(z)$
is the unique element of $\mathcal{D}_{n}^{\mathbb{C}_{\infty}}$
containing $z$, and $\xi_{\mathcal{D}_{n}(z)}$ denotes the conditional
measure of $\xi$ on $\mathcal{D}_{n}(z)$. The measure $\xi_{z,n}$
is called a level-$n$ component of $\xi$. As mentioned in Section
\ref{subsec:Component-measures}, we shall use probabilistic notation
introduced in \cite[Section 2.2]{Ho1}. In particular, we often regard
$\xi_{z,n}$ as a random measure in a natural way.

The proof of Theorem \ref{thm:ent inc} relies on Hochman's \cite{Ho}
inverse theorem for entropy growth under convolutions in $\mathbb{R}^{d}$.
An immediate corollary of this result, whose precise statement is
given in Theorem \ref{thm:ent inc in C} below and which we state
here somewhat informally and in less generality, says the following.
Let $\epsilon>0$, $m\ge1$, $n\ge N(\epsilon,m)\ge1$, and $\theta,\xi\in\mathcal{M}(\mathbb{C})$,
be such that $\mathrm{diam}(\mathrm{supp}(\theta)),\mathrm{diam}(\mathrm{supp}(\xi))=O(1)$,
$\frac{1}{n}H\left(\theta,\mathcal{D}_{n}\right)\ge\epsilon$, and
for most scales $1\le i\le n$, and most $z\in\mathbb{C}$ with respect
to $\xi$, there does not exist a nonzero $\mathbb{R}$-linear subspace
$V\subset\mathbb{C}$ so that
\begin{equation}
\frac{1}{m}H\left(\xi_{z,i},\mathcal{D}_{i+m}\right)\ge\frac{1}{m}H\left(\pi_{V^{\perp}}\xi_{z,i},\mathcal{D}_{i+m}\right)+\dim_{\mathbb{R}}V-\epsilon.\label{eq:non sat cond in intro}
\end{equation}
Then, under these assumptions,
\[
\frac{1}{n}H\left(\theta*\xi,\mathcal{D}_{n}\right)\ge\frac{1}{n}H\left(\xi,\mathcal{D}_{n}\right)+\delta,
\]
where $\delta$ is a positive number depending only on $\epsilon$
and $m$.
\begin{rem*}
When $V=\mathbb{C}$, (\ref{eq:non sat cond in intro}) says that
$\frac{1}{m}H\left(\xi_{z,i},\mathcal{D}_{i+m}\right)$ is close to
its maximal possible value, namely $2$. When $\dim_{\mathbb{R}}V=1$,
(\ref{eq:non sat cond in intro}) says that $\xi_{z,i}$ is saturated,
from an entropy standpoint, along lines parallel to $V$. For more
details, see \cite[Section 2]{Ho}.
\end{rem*}
Recall the map $\psi:\mathbb{CP}^{1}\rightarrow\mathbb{C}_{\infty}$
from Section \ref{subsec:Setup-and-background}, and set $\nu:=\psi\mu\in\mathcal{M}\left(\mathbb{C}_{\infty}\right)$.
To apply Theorem \ref{thm:ent inc in C} in the proof of Theorem \ref{thm:ent inc},
we need to verify that (\ref{eq:non sat cond in intro}) fails for
most components $\nu_{z,i}$ and all nonzero real subspaces $V\subset\mathbb{C}$.
When $\dim\mu<2$, this follows from the following statements.
\begin{prop}
\label{prop:uni ent dim}For every $\epsilon>0$, $m\ge M(\epsilon)\ge1$
and $n\ge N(\epsilon,m)\ge1$,
\[
\mathbb{P}_{1\le i\le n}\left\{ \left|\frac{1}{m}H\left(\nu_{z,i},\mathcal{D}_{i+m}\right)-\dim\mu\right|<\epsilon\right\} >1-\epsilon.
\]
\end{prop}

\begin{rem*}
In the terminology of \cite[Section 5]{Ho1}, Proposition \ref{prop:uni ent dim}
says that $\nu$ has uniform entropy dimension $\dim\mu$.
\end{rem*}
Let $\mathbb{RP}^{1}$ denote the set of real lines in $\mathbb{C}$;
that is, $\mathbb{RP}^{1}:=\left\{ z\mathbb{R}\::\:0\ne z\in\mathbb{C}\right\} $.
\begin{prop}
\label{prop:lb on ent of proj of comp of nu}Suppose that $\dim\mu<2$.
Then there exists $\gamma>0$ such that for every $\epsilon>0$, $m\ge M(\epsilon)\ge1$
and $n\ge1$,
\[
\mathbb{P}\left\{ \underset{w\mathbb{R}\in\mathbb{RP}^{1}}{\inf}\frac{1}{m}H\left(\pi_{w\mathbb{R}}\nu_{z,n},\mathcal{D}_{n+m}\right)>\dim\mu-1+\gamma\right\} >1-\epsilon.
\]
\end{prop}

The derivation of Theorem \ref{thm:ent inc} from Propositions \ref{prop:uni ent dim}
and \ref{prop:lb on ent of proj of comp of nu} and Theorem \ref{thm:ent inc in C}
(the corollary of Hochman's inverse theorem) does not require significant
new ideas. It relies on a linearization argument, which is used to
replace the action convolution $\theta.\mu$ with convolutions of
measures on $\mathbb{C}$. Moreover, in the course of the derivation
we establish that, in a suitable sense to be made precise (see Proposition
\ref{prop:from ent on G to ent on C}), if $\theta\in\mathcal{M}(\mathrm{G})$
has nonnegligible entropy, then a nonnegligible portion of the measures
on $\mathbb{C}$ associated to $\theta$ through the linearization
argument also inherit nonnegligible entropy. These ideas have previously
appeared in various forms in the literature (see \cite{BHR,Ho,HS}).

Proposition \ref{prop:uni ent dim} also does not involve major innovations,
and its proof extends existing methods originating in \cite{Ho1}.
On the other hand, Proposition \ref{prop:lb on ent of proj of comp of nu},
whose proof constitutes the main novelty of this paper, does introduce
significant new ideas. For the remainder of this subsection we discuss
Proposition \ref{prop:lb on ent of proj of comp of nu} and its proof.

First, note that by applying Proposition \ref{prop:uni ent dim} and
using basic properties of entropy, one can easily establish a version
of Proposition \ref{prop:lb on ent of proj of comp of nu} in which
$\dim\mu-1+\gamma$ is replaced by $\dim\mu-1-\epsilon$ (where $\epsilon>0$
is arbitrarily small). However, such a version is of no use for the
derivation of Theorem \ref{thm:ent inc}. Proposition \ref{prop:lb on ent of proj of comp of nu}
provides exactly what is needed to rule out (\ref{eq:non sat cond in intro})
for most $\nu_{z,i}$ and all $V\in\mathbb{RP}^{1}$ in the proof
of the entropy increase result.

On the other hand, as we next explain, Proposition \ref{prop:lb on ent of proj of comp of nu}
may be far from being optimal. Indeed, given a self-similar measure
$\mu'$ on $\mathbb{R}^{2}$, corresponding to an IFS containing at
least one similarity with an irrational rotational part, it follows
from \cite{MR3263957,HoSh} that
\begin{equation}
\dim\pi_{V}\mu'=\min\left\{ 1,\dim\mu'\right\} \text{ for all }V\in\mathbb{RP}^{1}.\label{eq:dim of proj of SS}
\end{equation}
Note that $\min\left\{ 1,\dim\mu'\right\} $ is always an upper bound
for $\dim\pi_{V}\mu'$. Combining (\ref{eq:dim of proj of SS}) with
the recursive structure of $\mu'$, one can show that, in a certain
sense that can be made precise, for most components of $\mu'$ all
their projections have normalized entropy close to this upper bound.
In our case, however, we are unable to establish an analogous statement
for $\nu$. That is, we cannot strengthen Proposition \ref{prop:lb on ent of proj of comp of nu}
by replacing $\dim\mu-1+\gamma$ with $\min\left\{ 1,\dim\mu\right\} -\epsilon$.
In fact, it is not even completely clear to us whether such a strengthening
should be expected to hold.
\begin{rem*}
Given a bounded convex open subset $\Omega\subset\mathbb{R}^{2}$,
a measure $\mu'\in\mathcal{M}\left(\Omega\right)$ is said to be self-conformal
if it is stationary with respect to a finitely supported probability
measure on the semigroup of strictly contracting injective conformal
maps from $\Omega$ into itself. Since Möbius transformations are
conformal, the setting of self-conformal measures intersects nontrivially
with the setup studied here. In the paper \cite{MR3903263} by Bruce
and Jin, it is claimed that (\ref{eq:dim of proj of SS}) holds for
all self-conformal measures $\mu'$ satisfying a mild irrationality
assumption. However, as confirmed by X. Jin (private communication),
there appears to be an issue in the proof of this claim that requires
a nontrivial fix.
\end{rem*}
We now turn to the proof of the proposition. Recall that for $i\in\Lambda$
we write $\varphi_{i}:=\varphi_{g_{i}}$, and set $\varphi_{u}:=\varphi_{i_{1}}\circ...\circ\varphi_{i_{n}}$
for $i_{1}...i_{n}=u\in\Lambda^{*}$, where $\Lambda^{*}$ denotes
the set of finite words over $\Lambda$. We consider $\mathbb{RP}^{1}$
as a multiplicative group by setting $z\mathbb{R}w\mathbb{R}:=zw\mathbb{R}$
for $z\mathbb{R},w\mathbb{R}\in\mathbb{RP}^{1}$. In the following
informal discussion, given $u\in\Lambda^{*}$ and $z\mathbb{R}\in\mathbb{RP}^{1}$,
whenever we refer to the entropy of $\pi_{z\mathbb{R}}\varphi_{u}\nu$
we mean its dyadic conditional entropy at appropriate scales (depending
on $u$) that are left unspecified.

Most of the proof of Proposition \ref{prop:lb on ent of proj of comp of nu}
is devoted to showing that entropies of measures of the form $\pi_{z\mathbb{R}}\varphi_{u}\nu$
are bounded away from below by $\dim\mu-1$ (see Proposition \ref{prop:lb on ent of proj of cylinders of nu}).
Here $u\in\Lambda^{*}$ is a word satisfying certain conditions that
hold with high probability. Note that, in contrast to the self-similar
setting, $\pi_{z\mathbb{R}}\circ\varphi_{u}$ is typically not an
affine map, which creates significant difficulties.\footnote{Note that in the reversed situation, where the maps $\varphi_{i}$
are all similarities and $\pi_{z\mathbb{R}}$ is replaced by an arbitrary
smooth regular map $F:\mathbb{C}\rightarrow\mathbb{R}$, the non-affinity
of $F\circ\varphi_{u}$ is less problematic. Indeed, in \cite{MR3263957,HoSh},
a version of (\ref{eq:dim of proj of SS}) is established for smooth
images of self-similar measures.}

To deal with these difficulties, we use the recursive structure of
$\nu$, together with the concavity of entropy, to bound the entropy
of $\pi_{z\mathbb{R}}\varphi_{u}\nu$ from below by an average of
entropies of measures of the form $\pi_{z\mathbb{R}}\varphi_{uv_{1}v_{2}}\nu$.
Here $v_{1},v_{2}\in\Lambda^{*}$ are chosen at random with respect
to certain natural distributions induced by $p$, the word $v_{1}$
is typically much longer than $v_{2}$, and $uv_{1}v_{2}$ denotes
the concatenation of $u$,$v_{1}$ and $v_{2}$. It is not hard to
show that, with high probability, the entropy of $\pi_{z\mathbb{R}}\varphi_{uv_{1}v_{2}}\nu$
is at least $\dim\mu-1$ up to an arbitrarily small error. Thus, in
order to prove the proposition, it suffices to show that, with nonnegligible
probability, the entropy of $\pi_{z\mathbb{R}}\varphi_{uv_{1}v_{2}}\nu$
is bounded away from below by $\dim\mu-1$.

To achieve this goal, we first carry out a linearization procedure
that allows us to approximate the entropy of $\pi_{z\mathbb{R}}\varphi_{uv_{1}v_{2}}\nu$
by the entropy of $\pi_{z\mathbb{R}\ell\left(u,v_{1},v_{2}\right)}\varphi_{v_{2}}\nu$,
where $\ell$ is an explicit function of $u$, $v_{1}$ and $v_{2}$
with values in $\mathbb{RP}^{1}$. Secondly, it is not difficult to
show that, for most words $v_{2}$, there exists a small interval
$I_{v_{2}}\subset\mathbb{RP}^{1}$ such that the entropy of $\pi_{w\mathbb{R}}\varphi_{v_{2}}\nu$
is at least $\frac{1}{2}\dim\mu$, up to an arbitrarily small error,
for all $w\mathbb{R}\in\mathbb{RP}^{1}\setminus I_{v_{2}}$. Note
that since $\dim\mu<2$, we have $\frac{1}{2}\dim\mu>\dim\mu-1$.

Taking these facts into account, and examining the definition of $\ell$,
it turns out that in order to achieve our goal it is necessary to
study the ergodic-theoretic properties of the direction cocycle $\alpha_{n}:\Lambda^{\mathbb{N}}\rightarrow\mathbb{RP}^{1}$,
defined by
\[
\alpha_{n}(\omega):=\varphi_{\omega|_{n}}'\left(\psi L\left(\sigma^{n}\omega\right)\right)\mathbb{R}\text{ for }n\ge0\text{ and }\beta\text{-a.e. }\omega\in\Lambda^{\mathbb{N}}.
\]
Here $\omega|_{n}$ denotes the prefix of $\omega$ of length $n$,
$\sigma:\Lambda^{\mathbb{N}}\rightarrow\Lambda^{\mathbb{N}}$ is the
left-shift map, and recall that $L:\Lambda^{\mathbb{N}}\rightarrow\mathbb{CP}^{1}$
is the Furstenberg boundary map. More precisely, what is needed is
to show that for every continuous $h:\Lambda^{\mathbb{N}}\rightarrow\mathbb{RP}^{1}$
and for $\beta$-a.e. $\omega$, the sequence $\left(\alpha_{n}(\omega)h\left(\sigma^{n}\omega\right)\right)_{n\ge0}$
does not equidistribute to a mass point (in the proof we actually
require a slightly stronger quantitative version of this property).

At this point we encounter another key difficulty, arising from the
fact that the action of $\mathrm{G}$ on $\mathbb{CP}^{1}$ is only
contracting on average. In situations where the action is strictly
contracting (e.g., in the classical self-similar setting), the Furstenberg
boundary map (often called the coding map in that context) is Hölder
continuous. In the contracting-on-average case, however, the boundary
map $L$ is in general only Borel measurable. This poses substantial
difficulties when studying the long-term behavior of $\alpha_{n}$,
and prevents the use of existing results on skew products of shifts
with compact groups (see, e.g., Parry \cite{parry_skew_prod}). Nevertheless,
using an ergodic-theoretic argument, we are still able to establish
the desired behavior of the sequences $\left(\alpha_{n}(\omega)h\left(\sigma^{n}\omega\right)\right)_{n\ge0}$.

The key step preceding the ergodic-theoretic argument is to show that
the cocycle $\alpha_{n}$ is not a coboundary; that is, there does
not exist a Borel measurable map $f:\Lambda^{\mathbb{N}}\rightarrow\mathbb{RP}^{1}$
such that $\alpha_{1}(\omega)=f(\omega)^{-1}f\left(\sigma\omega\right)$
for $\beta$-a.e. $\omega$. To establish this,  we show that if $\alpha_{n}$
were a coboundary, then it would necessarily follow that $\nu(C)>0$
for some generalized circle $C\subset\mathbb{C}_{\infty}$. However,
our standing assumptions on $\mathrm{S}_{\mathcal{G}}$ rule out this
possibility.

\subsubsection*{\textbf{\emph{Structure of the paper}}}

The rest of the paper is organized as follows. In Section \ref{sec:Preliminaries},
we introduce the necessary notation and definitions, and establish
several auxiliary results used throughout the paper. Section \ref{sec:Uniform-entropy-dimension}
establishes Proposition \ref{prop:uni ent dim}, showing that $\nu$
has uniform entropy dimension. In Section \ref{sec:Entropy-of-projections},
we prove Proposition \ref{prop:lb on ent of proj of comp of nu},
which bounds from below the entropy of projections of components of
$\nu$; this section contains the main novelty of our work. Section
\ref{sec:Proof-of-the ent inc res} derives the entropy increase result,
Theorem \ref{thm:ent inc}. In Section \ref{sec:Proof-of-the main result},
we complete the proof of our main result, Theorem \ref{thm:main result}.
Finally, in Appendix \ref{sec:appendix}, we use results from \cite{rapaport2020exact}
to deduce the exact dimensionality of $\mu$, together with a Ledrappier--Young-type
formula for its dimension.

\section{\label{sec:Preliminaries}Preliminaries}

\subsection{\label{subsec:Basic-notations}Basic notation and the setup}

Throughout this paper, the base of the logarithm is always $2$.

For a metric space $X$, denote by $\mathcal{M}(X)$ the collection
of all compactly supported Borel probability measures on $X$. Given
another metric space $Y$, a Borel map $f:X\rightarrow Y$, and a
measure $\nu\in\mathcal{M}(X)$, we write $f\nu:=\nu\circ f^{-1}$
for the pushforward of $\nu$ via $f$. For a Borel set $E\subset X$
with $\nu(E)>0$, we denote by $\nu_{E}$ the conditional measure
of $\nu$ on $E$; that is, $\nu_{E}:=\frac{1}{\nu(E)}\nu|_{E}$,
where $\nu|_{E}$ is the restriction of $\nu$ to $E$.

Given a partition $\mathcal{D}$ of a set $X$, for $x\in X$ we denote
by $\mathcal{D}(x)$ the unique $D\in\mathcal{D}$ containing $x$.

Given an integer $n\ge1$, let $\mathcal{N}_{n}:=\left\{ 1,...,n\right\} $,
and denote the normalized counting measure on $\mathcal{N}_{n}$ by
$\lambda_{n}$; that is, $\lambda_{n}\{i\}=1/n$ for each $1\le i\le n$.

\subsubsection*{Relations between parameters}

Given $R_{1},R_{2}\in\mathbb{R}$ with $R_{1},R_{2}\ge1$, we write
$R_{1}\ll R_{2}$ to indicate that $R_{2}$ is large with respect
to $R_{1}$. Formally, this means that $R_{2}\ge f(R_{1})$, where
$f$ is an unspecified function from $[1,\infty)$ into itself. The
values attained by $f$ are assumed to be sufficiently large, in a
manner depending on the specific context.

Similarly, given $0<\epsilon_{1},\epsilon_{2}<1$, we write $R_{1}\ll\epsilon_{1}^{-1}$,
$\epsilon_{2}^{-1}\ll R_{2}$, and $\epsilon_{1}^{-1}\ll\epsilon_{2}^{-1}$
to respectively indicate that $\epsilon_{1}$ is small with respect
to $R_{1}$, $R_{2}$ is large with respect to $\epsilon_{2}$, and
$\epsilon_{2}$ is small with respect to $\epsilon_{1}$.

The relation $\ll$ is clearly transitive. That is, if $R_{1}\ll R_{2}$
and for $R_{3}\ge1$ we have $R_{2}\ll R_{3}$, then also $R_{1}\ll R_{3}$.
For instance, the sentence ``Let $m\ge1$, $k\ge K(m)\ge1$ and $n\ge N(m,k)\ge1$
be given'' is equivalent to ``Let $m,k,n\ge1$ be with $m\ll k\ll n$''.

\subsubsection*{The setup}

As in Section \ref{sec:Introduction-and-the main result}, set $\mathrm{G}:=\mathrm{SL}(2,\mathbb{C})$,
let $\Lambda$ be a finite nonempty index set, fix a collection $\mathcal{G}=\{g_{i}\}_{i\in\Lambda}\subset\mathrm{G}$,
and fix a positive probability vector $p=(p_{i})_{i\in\Lambda}$.
Write $\mathrm{S}_{\mathcal{G}}$ for the subsemigroup of $\mathrm{G}$
generated by $\mathcal{G}$. For each $i\in\Lambda$, set $\varphi_{i}:=\varphi_{g_{i}}$,
where $\varphi_{g_{i}}:\mathbb{C}_{\infty}\rightarrow\mathbb{C}_{\infty}$
is the Möbius transformation induced by $g_{i}$.

In what follows, we always assume that $\mathrm{S}_{\mathcal{G}}$
is strongly irreducible, proximal, and does not fix a generalized
circle. We assume that $\mathcal{G}$ is weakly Diophantine only in
Section \ref{subsec:Proof-of-Theorem main}, where we prove our main
result.

As before, write $\mu\in\mathcal{M}\left(\mathbb{CP}^{1}\right)$
for the Furstenberg measure associated to $\mathcal{G}$ and $p$;
that is, $\mu$ is the unique element of $\mathcal{M}\left(\mathbb{CP}^{1}\right)$
satisfying $\mu=\sum_{i\in\Lambda}p_{i}\cdot g_{i}\mu$.

\subsection{\label{subsec:Algebraic-notation}Algebraic notation}

Given $w\in\mathbb{C}$, let $S_{w}:\mathbb{C}\rightarrow\mathbb{C}$
be defined by $S_{w}(z)=wz$ for $z\in\mathbb{C}$.

We denote by $\mathbb{RP}^{1}$ the set of real lines in $\mathbb{C}$;
that is, $\mathbb{RP}^{1}:=\left\{ z\mathbb{R}\::\:0\ne z\in\mathbb{C}\right\} $.
For $z\mathbb{R},w\mathbb{R}\in\mathbb{RP}^{1}$, we set $z\mathbb{R}w\mathbb{R}:=zw\mathbb{R}$,
which makes $\mathbb{RP}^{1}$ into a multiplicative group whose identity
element is $\mathbb{R}$. Let $S_{z\mathbb{R}}:\mathbb{RP}^{1}\rightarrow\mathbb{RP}^{1}$
be defined by $S_{z\mathbb{R}}\left(w\mathbb{R}\right)=zw\mathbb{R}$.

Given $z\mathbb{R}\in\mathbb{RP}^{1}$, we denote by $\pi_{z\mathbb{R}}:\mathbb{C}\rightarrow\mathbb{C}$
the orthogonal projection onto $z\mathbb{R}$, where $\mathbb{C}$
is identified with $\mathbb{R}^{2}$; that is,
\[
\pi_{z\mathbb{R}}(w)=|z|^{-2}\mathrm{Re}\left(w\overline{z}\right)z\text{ for }w\in\mathbb{C}.
\]

Let $\mathrm{SU}(2)$ denote the special unitary group of degree $2$,
which is a compact subgroup of $\mathrm{G}$. Given $g\in\mathrm{G}$
and setting $D:=\mathrm{diag}\left(\Vert g\Vert_{\mathrm{op}},\Vert g\Vert_{\mathrm{op}}^{-1}\right)\in\mathrm{G}$,
where $\Vert\cdot\Vert_{\mathrm{op}}$ is the operator norm, it is
well known that there exist $U,V\in\mathrm{SU}(2)$ such that $g=UDV$.
In this situation, we say that $UDV$ is a singular value decomposition
of $g$.

Let us define a Borel mapping $L:\mathrm{G}\rightarrow\mathbb{CP}^{1}$
as follows. Write $\left\{ e_{1},e_{2}\right\} $ for the standard
basis of $\mathbb{C}^{2}$. Let $g\in\mathrm{G}$, and let $g=UDV$
be a singular value decomposition of $g$. If $\Vert g\Vert_{\mathrm{op}}>1$,
then we define $L(g)=Ue_{1}\mathbb{C}$; otherwise, if $\Vert g\Vert_{\mathrm{op}}=1$,
we define $L(g)=e_{1}\mathbb{C}$. It is easy to see that this definition
is independent of the specific singular value decomposition of $g$,
and hence $L$ is well defined.

Let $\psi:\mathbb{CP}^{1}\rightarrow\mathbb{C}_{\infty}$ be defined
by
\[
\psi\left(z\mathbb{C}\right)=\begin{cases}
z_{1}/z_{2} & \text{ if }z_{2}\ne0\\
\infty & \text{ if }z_{2}=0
\end{cases}\text{ for all }(z_{1},z_{2})=z\in\mathbb{C}^{2}\setminus\{0\}.
\]
Note that $\psi$ is $\mathrm{G}$-equivariant, meaning that
\begin{equation}
\psi\left(gz\mathbb{C}\right)=\varphi_{g}\circ\psi\left(z\mathbb{C}\right)\text{ for all }g\in\mathrm{G}\text{ and }z\mathbb{C}\in\mathbb{CP}^{1}.\label{eq:psi is G-equi}
\end{equation}
Writing $\nu:=\psi\mu$, it follows that $\nu$ is the unique element
of $\mathcal{M}\left(\mathbb{C}_{\infty}\right)$ satisfying $\nu=\sum_{i\in\Lambda}p_{i}\cdot\varphi_{i}\nu$.

Given $\theta\in\mathcal{M}(\mathrm{G})$ and $\xi\in\mathcal{M}\left(\mathbb{CP}^{1}\right)$,
we write $\theta.\xi\in\mathcal{M}\left(\mathbb{CP}^{1}\right)$ for
the pushforward of $\theta\times\xi$ via the action map $\left(g,z\mathbb{C}\right)\mapsto gz\mathbb{C}$.
Similarly, given $\xi\in\mathcal{M}(\mathbb{C}_{\infty})$, we denote
by $\theta.\xi\in\mathcal{M}(\mathbb{C}_{\infty})$ the pushforward
of $\theta\times\xi$ via the map $(g,z)\mapsto\varphi_{g}(z)$. For
$z\in\mathbb{C}_{\infty}$, we write $\theta\ldotp z$ in place of
$\theta.\delta_{z}$, where $\delta_{z}$ is the Dirac mass at $z$.

\subsection{\label{subsec:Metric-preliminaries}Metric preliminaries}

In what follows, given a metric space $(X,d)$, a point $x\in X$,
and $r>0$, we write $B(x,r)$ for the closed ball in $X$ with center
$x$ and radius $r$. For a nonempty subset $E\subset X$, we write
$\mathrm{diam}(E)$ for its diameter, and denote by $E^{(r)}$ the
closed $r$-neighborhood of $E$; that is, $E^{(r)}:=\left\{ x\in X\::\:d(x,E)\le r\right\} $.

Given $m\in\mathbb{Z}_{>0}$, we denote by $\left\langle \cdot,\cdot\right\rangle $
and $\Vert\cdot\Vert$ the standard inner product and norm of $\mathbb{C}^{m}$.
We denote by $d_{\mathbb{C}^{m}}$ the metric induced by $\Vert\cdot\Vert$.
In particular, $d_{\mathbb{C}}$ is the metric induced by the standard
absolute value of $\mathbb{C}$.

For $(z_{1},z_{2})=z,(w_{1},w_{2})=w\in\mathbb{C}^{2}\setminus\{0\}$,
define
\[
d_{\mathbb{CP}^{1}}\left(z\mathbb{C},w\mathbb{C}\right):=\frac{1}{\Vert z\Vert\Vert w\Vert}\left|\det\left(\begin{array}{cc}
z_{1} & w_{1}\\
z_{2} & w_{2}
\end{array}\right)\right|.
\]
As pointed out in \cite[Section 13.1]{BQ}, this defines a metric
which induces the usual compact topology on $\mathbb{CP}^{1}$. Note
that $\mathrm{diam}\left(\mathbb{CP}^{1}\right)=1$. Additionally,
for each $U\in\mathrm{SU}(2)$, the map $z\mathbb{C}\mapsto Uz\mathbb{C}$
is an isometry of $\left(\mathbb{CP}^{1},d_{\mathbb{CP}^{1}}\right)$.
Moreover, it is easy to see that $d_{\mathbb{CP}^{1}}$ is bi-Lipschitz
equivalent to any Riemannian distance function on $\mathbb{CP}^{1}$.

For $z,w\in\mathbb{C}$ with $|z|=|w|=1$, write
\[
d_{\mathbb{RP}^{1}}\left(z\mathbb{R},w\mathbb{R}\right):=\left(1-\mathrm{Re}\left(z\overline{w}\right)^{2}\right)^{1/2},
\]
which defines a metric on $\mathbb{RP}^{1}$ (see \cite[Section III.4]{BL}).

Let $d_{\mathrm{G}}$ be the Riemannian distance function induced
by a left-invariant Riemannian metric on $\mathrm{G}$. Then $d_{\mathrm{G}}$
is also left-invariant, meaning that
\[
d_{\mathrm{G}}(hg,hg')=d_{\mathrm{G}}(g,g')\text{ for all }h,g,g'\in\mathrm{G}.
\]
It is easy to see that the metric space $\left(\mathrm{G},d_{\mathrm{G}}\right)$
is complete. Hence, by the Hopf--Rinow theorem (see \cite[Chapter 7]{dC}),
closed and bounded subsets of $\mathrm{G}$ are compact. In particular,
$B(1_{\mathrm{G}},r)$ is a compact subset of $\mathrm{G}$ for all
$r>0$, where $1_{\mathrm{G}}$ denotes the identity element of $\mathrm{G}$.

In what follows, all metric concepts in $\mathbb{C}^{m}$, $\mathbb{CP}^{1}$,
$\mathbb{RP}^{1}$ and $\mathrm{G}$ should be understood with respect
to $d_{\mathbb{C}^{m}}$, $d_{\mathbb{CP}^{1}}$, $d_{\mathbb{RP}^{1}}$,
and $d_{\mathrm{G}}$, respectively. We shall omit the subscripts
when there is no risk of confusion.

The following lemma, whose simple proof is omitted, will be used repeatedly.
\begin{lem}
\label{lem:bi-lip prop of psi}Given $R>0$,
\[
\psi^{-1}\left\{ z\in\mathbb{C}\::\:|z|<R\right\} =\mathbb{CP}^{1}\setminus B\left(e_{1}\mathbb{C},\left(1+R^{2}\right)^{-1/2}\right).
\]
Moreover, for each $z,z'\in\mathbb{C}$ with $|z|,|z'|<R$,
\[
\frac{1}{1+R^{2}}\left|z-z'\right|\le d\left(\psi^{-1}(z),\psi^{-1}(z')\right)\le\left|z-z'\right|.
\]
Consequently, for each $w\mathbb{C},w'\mathbb{C}\in\mathbb{CP}^{1}\setminus B\left(e_{1}\mathbb{C},1/R\right)$,
\[
d\left(w\mathbb{C},w'\mathbb{C}\right)\le\left|\psi\left(w\mathbb{C}\right)-\psi\left(w'\mathbb{C}\right)\right|\le\left(1+R^{2}\right)d\left(w\mathbb{C},w'\mathbb{C}\right).
\]
\end{lem}

We shall also need the following lemmas concerning metric properties
of the action of $\mathrm{G}$ on $\mathbb{CP}^{1}$.
\begin{lem}
\label{lem:zC to gzC is bi-Lip with const norm^2}Let $g\in\mathrm{G}$
be given. Then the map sending $z\mathbb{C}\in\mathbb{CP}^{1}$ to
$gz\mathbb{C}$ is bi-Lipschitz with bi-Lipschitz constant $\Vert g\Vert_{\mathrm{op}}^{2}$;
that is, for all $z\mathbb{C},w\mathbb{C}\in\mathbb{CP}^{1}$,
\[
\Vert g\Vert_{\mathrm{op}}^{-2}d\left(z\mathbb{C},w\mathbb{C}\right)\le d\left(gz\mathbb{C},gw\mathbb{C}\right)\le\Vert g\Vert_{\mathrm{op}}^{2}d\left(z\mathbb{C},w\mathbb{C}\right)
\]
\end{lem}

\begin{proof}
Let $(z_{1},z_{2})=z,(w_{1},w_{2})=w\in\mathbb{C}^{2}\setminus\{0\}$
be given. Setting $D:=\mathrm{diag}\left(\Vert g\Vert_{\mathrm{op}},\Vert g\Vert_{\mathrm{op}}^{-1}\right)$,
we have
\begin{eqnarray*}
d\left(Dz\mathbb{C},Dw\mathbb{C}\right) & = & \frac{1}{\Vert Dz\Vert\Vert Dw\Vert}\left|\det\left(\begin{array}{cc}
\Vert g\Vert_{\mathrm{op}}z_{1} & \Vert g\Vert_{\mathrm{op}}w_{1}\\
\Vert g\Vert_{\mathrm{op}}^{-1}z_{2} & \Vert g\Vert_{\mathrm{op}}^{-1}w_{2}
\end{array}\right)\right|\\
 & \le & \Vert g\Vert_{\mathrm{op}}^{2}d\left(z\mathbb{C},w\mathbb{C}\right).
\end{eqnarray*}
Moreover, as pointed out above,
\[
d\left(Uz\mathbb{C},Uw\mathbb{C}\right)=d\left(z\mathbb{C},w\mathbb{C}\right)\text{ for all }U\in\mathrm{SU}(2).
\]
Hence, by considering a singular value decomposition of $g$, we see
that the map $z\mathbb{C}\mapsto gz\mathbb{C}$ is $\Vert g\Vert_{\mathrm{op}}^{2}$-Lipschitz.
The lemma now follows by applying this also to the map $z\mathbb{C}\mapsto g^{-1}z\mathbb{C}$
and noting that $\Vert g^{-1}\Vert_{\mathrm{op}}=\Vert g\Vert_{\mathrm{op}}$.
\end{proof}
\begin{lem}
\label{lem:dist of gzC to L(g)}Let $g\in\mathrm{G}$ and $0<\epsilon<1$
be given. Then
\[
d\left(gz\mathbb{C},gw\mathbb{C}\right)\le\epsilon^{-2}\Vert g\Vert_{\mathrm{op}}^{-2}d\left(z\mathbb{C},w\mathbb{C}\right)\text{ for all }z\mathbb{C},w\mathbb{C}\in\mathbb{CP}^{1}\setminus B\left(L(g^{-1}),\epsilon\right),
\]
and
\[
d\left(L(g),gz\mathbb{C}\right)\le\epsilon^{-1}\Vert g\Vert_{\mathrm{op}}^{-2}\text{ for all }z\mathbb{C}\in\mathbb{CP}^{1}\setminus B\left(L(g^{-1}),\epsilon\right).
\]
\end{lem}

\begin{proof}
Set $M:=\Vert g\Vert_{\mathrm{op}}$. When $M=1$ we have $g\in\mathrm{SU}(2)$,
so in this case the lemma is clear.

Suppose that $M>1$, and let $g=UDV$ be a singular value decomposition
of $g$. Let $z,w\in\mathbb{C}^{2}$ be unit vectors with $z\mathbb{C},w\mathbb{C}\notin B\left(L(g^{-1}),\epsilon\right)$,
and let $a,b,a',b'\in\mathbb{C}$ be such that $Vz=(a,b)$ and $Vw=(a',b')$.
Note that $L(g^{-1})=V^{-1}e_{2}\mathbb{C}$. Hence,
\[
\left|a\right|=d\left(Vz\mathbb{C},e_{2}\mathbb{C}\right)=d\left(z\mathbb{C},V^{-1}e_{2}\mathbb{C}\right)>\epsilon,
\]
and similarly $\left|a'\right|>\epsilon$. Thus,
\begin{multline}
d\left(gz\mathbb{C},gw\mathbb{C}\right)=d\left(DVz\mathbb{C},DVw\mathbb{C}\right)\\
=\frac{1}{\left\Vert \left(Ma,M^{-1}b\right)\right\Vert }\frac{1}{\left\Vert \left(Ma',M^{-1}b'\right)\right\Vert }\left|\det\left(\begin{array}{cc}
Ma & Ma'\\
M^{-1}b & M^{-1}b'
\end{array}\right)\right|\\
\le\frac{d\left(Vz\mathbb{C},Vw\mathbb{C}\right)}{\left|a\right|\left|a'\right|M^{2}}=\frac{d\left(z\mathbb{C},w\mathbb{C}\right)}{\left|a\right|\left|a'\right|M^{2}}\le\frac{d\left(z\mathbb{C},w\mathbb{C}\right)}{\epsilon^{2}M^{2}},\label{eq:first part of d(gzC,L(g)) lemma}
\end{multline}
which proves the first part of the lemma.

Setting $w:=V^{-1}e_{1}$, we have $d\left(w\mathbb{C},L(g^{-1})\right)=1$,
$Vw=(1,0)$, and $gw\mathbb{C}=L(g)$. Hence, from (\ref{eq:first part of d(gzC,L(g)) lemma}),
\[
d\left(gz\mathbb{C},L(g)\right)\le\frac{d\left(z\mathbb{C},w\mathbb{C}\right)}{\left|a\right|M^{2}}\le\epsilon^{-1}M^{-2},
\]
which completes the proof of the lemma.
\end{proof}

\subsection{\label{subsec:Entropy}Entropy}

Let $(X,\mathcal{F})$ be a measurable space. Given a probability
measure $\theta$ on $X$ and a countable partition $\mathcal{D}\subset\mathcal{F}$
of $X$, the entropy of $\theta$ with respect to $\mathcal{D}$ is
defined by
\[
H(\theta,\mathcal{D}):=-\sum_{D\in\mathcal{D}}\theta(D)\log\theta(D).
\]
If $\mathcal{E}\subset\mathcal{F}$ is another countable partition
of $X$, the conditional entropy given $\mathcal{E}$ is defined by
\[
H(\theta,\mathcal{D}\mid\mathcal{E}):=\sum_{E\in\mathcal{E}}\theta(E)\cdot H(\theta_{E},\mathcal{D}).
\]

Throughout the paper, we repeatedly use basic properties of entropy
and conditional entropy, often without explicit reference. Readers
are advised to consult \cite[Section 3.1]{Ho1} for details.

In particular, we shall often use the fact that entropy and conditional
entropy are concave and almost convex in the measure argument. That
is, given probability measures $\theta_{1},...,\theta_{k}$ on $X$
and a probability vector $q=(q_{i})_{i=1}^{k}$ such that $\theta=\sum_{i=1}^{k}q_{i}\theta_{i}$,
we have
\[
\sum_{i=1}^{k}q_{i}H(\theta_{i},\mathcal{D})\le H(\theta,\mathcal{D})\le\sum_{i=1}^{k}q_{i}H(\theta_{i},\mathcal{D})+H(q),
\]
where $H(q):=-\sum_{i=1}^{k}q_{i}\log q_{i}$ is the entropy of $q$.
These inequalities remain valid with $H(\cdot,\mathcal{D}\mid\mathcal{E})$
in place of $H(\cdot,\mathcal{D})$.

\subsection{\label{subsec:Dyadic-partitions}Dyadic partitions}

For $m\ge1$ and $n\ge0$, denote by $\mathcal{D}_{n}^{\mathbb{C}^{m}}$
the level-$n$ dyadic partition of $\mathbb{C}^{m}$, where $\mathbb{C}^{m}$
is identified with $\mathbb{R}^{2m}$. For a real number $t\ge0$,
we write $\mathcal{D}_{t}^{\mathbb{C}^{m}}$ in place of $\mathcal{D}_{\left\lfloor t\right\rfloor }^{\mathbb{C}^{m}}$,
where $\left\lfloor t\right\rfloor $ denotes the integral part of
$t$. We extend these partitions to $\mathbb{C}_{\infty}$ by setting
\[
\mathcal{D}_{n}^{\mathbb{C}_{\infty}}:=\mathcal{D}_{n}^{\mathbb{C}}\cup\left\{ \{\infty\}\right\} .
\]
We usually omit the superscripts $\mathbb{C}^{m}$ and $\mathbb{C}_{\infty}$
when they are clear from the context. For instance, it is easy to
verify that
\begin{equation}
\frac{1}{k}H\left(\xi,\mathcal{D}_{n+k}\mid\mathcal{D}_{n}\right)\le2\text{ for every }\xi\in\mathcal{M}(\mathbb{C}),n\in\mathbb{Z}_{\ge0}\text{ and }k\in\mathbb{Z}_{>0}.\label{eq:norm cond ent on C <=00003D 2}
\end{equation}

We also need to introduce dyadic-like partitions for $\mathbb{CP}^{1}$
and $\mathrm{G}$. Letting $X$ denote either $\mathbb{CP}^{1}$ or
$\mathrm{G}$, it follows from \cite[Remark 2.2]{KRS} that there
exists a sequence $\{\mathcal{D}_{n}^{X}\}_{n\ge0}$ of Borel partitions
of $X$ such that:
\begin{enumerate}
\item $\mathcal{D}_{n+1}^{X}$ refines $\mathcal{D}_{n}^{X}$ for each $n\ge0$;
that is, for each $D\in\mathcal{D}_{n+1}^{X}$, there exists $D'\in\mathcal{D}_{n}^{X}$
with $D\subset D'$;
\item there exists a constant $C=C(X)>1$ such that for each $n\ge0$ and
$D\in\mathcal{D}_{n}^{X}$, there exists $x_{D}\in D$ with
\begin{equation}
B(x_{D},C^{-1}2^{-n})\subset D\subset B(x_{D},C2^{-n}).\label{eq:in second prop of D^X}
\end{equation}
\end{enumerate}
As mentioned above, for a real $t\ge0$, we shall write $\mathcal{D}_{t}^{X}$
in place of $\mathcal{D}_{\left\lfloor t\right\rfloor }^{X}$. When
there is no risk of confusion, we write $\mathcal{D}_{n}$ in place
of $\mathcal{D}_{n}^{X}$.

Recall that $\mathrm{diam}\left(\mathbb{CP}^{1}\right)=1$, and note
that $\mathbb{CP}^{1}$ has dimension $2$ as a real manifold. Hence,
by Lemma \ref{lem:ub on card of dyad atoms intersecting} below, there
exists a constant $C>1$ such that
\begin{equation}
\left|\mathcal{D}_{n}^{\mathbb{CP}^{1}}\right|\le C2^{2n}\text{ for all }n\ge0.\label{eq:ub on card of dyd part of CP^1}
\end{equation}

The following lemma, which relates dimension and entropy, follows
easily from \cite[Theorem 4.4]{Yo} and basic properties of entropy.
\begin{lem}
\label{lem:from dim to ent}Let $\xi\in\mathcal{M}\left(\mathbb{CP}^{1}\right)$
be exact dimensional. Then,
\[
\underset{n\rightarrow\infty}{\lim}\frac{1}{n}H\left(\xi,\mathcal{D}_{n}\right)=\dim\xi.
\]
\end{lem}

For the remainder of this subsection, let $X$ denote either $\mathbb{CP}^{1}$,
$\mathrm{G}$, or $\mathbb{C}^{m}$ for some $m\ge1$. The next lemma
will be used several times in what follows.
\begin{lem}
\label{lem:ub on card of dyad atoms intersecting}Let $R>1$ be given,
and write $q$ for the dimension of $X$ as a real manifold. Then
for every Borel set $\emptyset\ne F\subset X$ with $\mathrm{diam}(F)\le R$,
\[
\#\left\{ D\in\mathcal{D}_{n}^{X}\::\:D\cap F\ne\emptyset\right\} =O_{X,R}\left(1+2^{nq}\mathrm{diam}(F)^{q}\right)\text{ for all }n\in\mathbb{Z}_{\ge0}.
\]
\end{lem}

\begin{rem*}
The parameter $R$ in the statement of the lemma is in fact needed
only when $X=\mathrm{G}$, where it is required because $\mathrm{G}$
has exponential volume growth.
\end{rem*}
\begin{proof}
If $X=\mathbb{CP}^{1}$, let $\lambda$ denote the unique $\mathrm{SU}(2)$-invariant
member of $\mathcal{M}(X)$. If $X=\mathrm{G}$, let $\lambda$ denote
the Haar measure on $\mathrm{G}$ associated to the left-invariant
Riemannian metric inducing $d_{\mathrm{G}}$. If $X=\mathbb{C}^{m}$
for some $m\ge1$, let $\lambda$ denote the Lebesgue measure on $\mathbb{C}^{m}$.
In any case, there exists $M=M(X,R)>1$ such that
\[
M^{-1}r^{q}\le\lambda\left(B(x,r)\right)\le Mr^{q}\text{ for all }x\in X\text{ and }0<r\le3R.
\]

Let $\emptyset\ne F\subset X$ be a Borel set with $\mathrm{diam}(F)\le R$,
let $n\in\mathbb{Z}_{\ge0}$, and write
\[
\mathcal{E}:=\left\{ D\in\mathcal{D}_{n}^{X}\::\:D\cap F\ne\emptyset\right\} .
\]
Let $C=C(X)>1$ be a constant as appearing in (\ref{eq:in second prop of D^X}),
set $\rho:=\mathrm{diam}(F)$, and suppose first that $2^{-n}\le\frac{\rho}{2C}$.
For each $D\in\mathcal{E}$ there exists $x_{D}\in D$ such that
\[
B\left(x_{D},C^{-1}2^{-n}\right)\subset D\subset B\left(x_{D},C2^{-n}\right),
\]
which implies that $\mathrm{diam}(D)\le C2^{1-n}\le\rho$.

Fix some $y\in F$. Given $D\in\mathcal{E}$, there exists $z_{D}\in D\cap F$,
and so
\[
d\left(x_{D},y\right)\le d\left(x_{D},z_{D}\right)+d\left(z_{D},y\right)\le2\rho.
\]
Thus, since $\mathrm{diam}\left(B\left(x_{D},C^{-1}2^{-n}\right)\right)\le\rho$,
\[
B\left(x_{D},C^{-1}2^{-n}\right)\subset B\left(y,3\rho\right)\text{ for each }D\in\mathcal{E}.
\]
Hence, since the balls $\left\{ B\left(x_{D},C^{-1}2^{-n}\right)\right\} _{D\in\mathcal{E}}$are
disjoint,
\[
\left|\mathcal{E}\right|M^{-1}C^{-q}2^{-nq}\le\sum_{D\in\mathcal{E}}\lambda\left(B\left(x_{D},C^{-1}2^{-n}\right)\right)\le\lambda\left(B\left(y,3\rho\right)\right)\le M3^{q}\rho^{q},
\]
which gives
\[
\left|\mathcal{E}\right|\le M^{2}C^{q}3^{q}\cdot2^{nq}\rho^{q}=O_{X,R}\left(2^{nq}\rho^{q}\right).
\]

Suppose next that $2^{-n}>\frac{\rho}{2C}$, and let $k\in\mathbb{Z}_{>0}$
be with $2^{-k}\le\frac{\rho}{2C}<2^{1-k}$. Since $k>n$, it holds
that $\mathcal{D}_{k}^{X}$ refines $\mathcal{D}_{n}^{X}$. Hence,
by the preceding part of the proof,
\[
\left|\mathcal{E}\right|\le\#\left\{ D\in\mathcal{D}_{k}^{X}\::\:D\cap F\ne\emptyset\right\} =M^{2}C^{q}3^{q}\cdot2^{kq}\rho^{q}=O_{X,R}(1),
\]
which completes the proof of the lemma.
\end{proof}
The following statement follows directly from (\ref{eq:in second prop of D^X})
and Lemma \ref{lem:ub on card of dyad atoms intersecting}.
\begin{lem}
\label{lem:bd num of subatoms}There exists a constant $C=C(X)>1$
such that for every $n\ge0$ and $D\in\mathcal{D}_{n}^{X}$,
\[
\#\left\{ D'\in\mathcal{D}_{n+1}^{X}\::\:D'\subset D\right\} \le C.
\]
\end{lem}

In the following lemma, let $X'$ denote either $\mathbb{CP}^{1}$,
$\mathrm{G}$, or $\mathbb{C}^{m}$ for some $m\ge1$.
\begin{lem}
\label{lem:dyad ent =000026 lip func}Let $\theta\in\mathcal{M}(X)$,
$f:\mathrm{supp}(\theta)\rightarrow X'$, $s>0$, and $C\ge1$ be
such that
\[
C^{-1}s\cdot d\left(x_{1},x_{2}\right)\le d\left(f(x_{1}),f(x_{2})\right)\le Cs\cdot d\left(x_{1},x_{2}\right)\text{ for all }x_{1},x_{2}\in\mathrm{supp}(\theta).
\]
Then for each $n\ge\log C$ with $n+\log s\ge\log C$,
\begin{equation}
\left|H\left(f\theta,\mathcal{D}_{n}\right)-H\left(\theta,\mathcal{D}_{n+\log s}\right)\right|=O_{X,X'}\left(1+\log C\right).\label{eq:conc in lip func lemma}
\end{equation}
Moreover, (\ref{eq:conc in lip func lemma}) holds for all $n\ge\max\left\{ 0,-\log s\right\} $
whenever $X=X'=\mathbb{CP}^{1}$.
\end{lem}

\begin{rem*}
It is not difficult to see that the stronger assumptions $n\ge\log C$
and $n+\log s\ge\log C$ are in fact needed only when $X=\mathrm{G}$
or $X'=\mathrm{G}$. However, we will not need this refinement.
\end{rem*}
\begin{proof}
Let $n\ge0$ be given. If $X\ne\mathbb{CP}^{1}$ or $X'\ne\mathbb{CP}^{1}$,
assume that $n\ge\log C$ and $n+\log s\ge\log C$. Otherwise, if
$X=X'=\mathbb{CP}^{1}$, assume only that $n+\log s\ge0$.

For $D\in\mathcal{D}_{n+\log s}^{X}$ we have $\mathrm{diam}(D)=O_{X}\left(s^{-1}2^{-n}\right)$,
and so
\[
\mathrm{diam}\left(f\left(D\cap\mathrm{supp}(\theta)\right)\right)=O_{X}\left(C2^{-n}\right)
\]
(note that $C2^{-n}\le1$ when $X'\ne\mathbb{CP}^{1}$). Hence, by
applying Lemma \ref{lem:ub on card of dyad atoms intersecting} in
$X'$ with $F=f\left(D\cap\mathrm{supp}(\theta)\right)$,
\[
\log\left(\#\left\{ E\in f^{-1}\mathcal{D}_{n}^{X'}\::\:E\cap D\ne\emptyset\right\} \right)=O_{X,X'}\left(1+\log C\right)\text{ for }D\in\mathcal{D}_{n+\log s}^{X},
\]
which implies
\[
H\left(f\theta,\mathcal{D}_{n}\right)-H\left(\theta,\mathcal{D}_{n+\log s}\right)\le H\left(\theta,f^{-1}\mathcal{D}_{n}\mid\mathcal{D}_{n+\log s}\right)=O_{X,X'}\left(1+\log C\right).
\]

Set $\theta':=f\theta\in\mathcal{M}(X')$ and $h:=f^{-1}$, and note
that $h:\mathrm{supp}\left(\theta'\right)\rightarrow X$ satisfies
\[
C^{-1}s^{-1}\cdot d\left(x_{1}',x_{2}'\right)\le d\left(h(x_{1}'),h(x_{2}')\right)\le Cs^{-1}\cdot d\left(x_{1}',x_{2}'\right)
\]
for all $x_{1}',x_{2}'\in\mathrm{supp}\left(\theta'\right)$. Hence,
by applying the preceding argument with $\theta'$ in place of $\theta$,
$h$ in place of $f$, $s^{-1}$ in place of $s$, and $n':=n+\log s$
in place of $n$, we obtain
\[
H\left(\theta',h^{-1}\mathcal{D}_{n'}\right)-H\left(\theta',\mathcal{D}_{n'+\log s^{-1}}\right)\le O_{X,X'}\left(1+\log C\right).
\]
Since
\[
H\left(\theta',h^{-1}\mathcal{D}_{n'}\right)=H\left(\theta,\mathcal{D}_{n+\log s}\right)\text{ and }H\left(\theta',\mathcal{D}_{n'+\log s^{-1}}\right)=H\left(f\theta,\mathcal{D}_{n}\right),
\]
this completes the proof of the lemma.
\end{proof}
We shall also need the following statement. Its simple proof is similar
to that of Lemma \ref{lem:dyad ent =000026 lip func} and is therefore
omitted.
\begin{lem}
\label{lem:ent of push by close func}Let $\left(Z,\mathcal{F},\theta\right)$
be a probability space, and let $f,h:Z\rightarrow X$ be measurable.
Let $n\ge0$, and suppose that $d_{X}\left(f(z),h(z)\right)\le2^{-n}$
for all $z\in Z$. Then,
\[
H\left(f\theta,\mathcal{D}_{n}\right)=H\left(h\theta,\mathcal{D}_{n}\right)+O_{X}(1).
\]
\end{lem}

\subsection{\label{subsec:Component-measures}Component measures}

In this subsection, let $X$ denote either $\mathbb{CP}^{1}$, $\mathrm{G}$,
$\mathbb{C}_{\infty}$, or $\mathbb{C}^{m}$ for some $m\ge1$. Let
$\theta\in\mathcal{M}(X)$ be given. For $n\ge0$ and $x\in X$ with
$\theta\left(\mathcal{D}_{n}(x)\right)>0$, we write $\theta_{x,n}$
in place of the conditional measure $\theta_{\mathcal{D}_{n}(x)}$.
The measure $\theta_{x,n}$ is said to be a level-$n$ component of
$\theta$.

Throughout the rest of the paper, we use the probabilistic notations
introduced in \cite[Section 2.2]{Ho1}; readers are encouraged to
consult this reference for further details. In particular, we often
consider $\theta_{x,n}$ as a random measure in a natural way. Thus,
for an event $\mathcal{U}\subset\mathcal{M}(X)$,
\[
\mathbb{P}\left(\theta_{x,n}\in\mathcal{U}\right):=\theta\left\{ x\in X\::\:\theta_{\mathcal{D}_{n}(x)}\in\mathcal{U}\right\} .
\]
Additionally, for integers $n_{2}\ge n_{1}\ge0$, we write
\[
\mathbb{P}_{n_{1}\le i\le n_{2}}\left(\theta_{x,i}\in\mathcal{U}\right):=\frac{1}{n_{2}-n_{1}+1}\sum_{i=n_{1}}^{n_{2}}\mathbb{P}\left(\theta_{x,i}\in\mathcal{U}\right).
\]
Similarly, given a measurable $f:\mathcal{M}(X)\rightarrow[0,\infty)$,
\[
\mathbb{E}_{n_{1}\le i\le n_{2}}\left(f\left(\theta_{x,i}\right)\right):=\frac{1}{n_{2}-n_{1}+1}\sum_{i=n_{1}}^{n_{2}}\int f\left(\theta_{\mathcal{D}_{i}(x)}\right)\:d\theta(x).
\]

The proof of the following lemma is similar to that of \cite[Lemma 3.4]{Ho1}
and is therefore omitted.
\begin{lem}
\label{lem:glob ent to loc ent}Let $\theta\in\mathcal{M}(X)$, $n\ge m\ge1$,
$i\in\mathbb{Z}_{\ge0}$, and $C>1$ be given. Suppose that $\mathrm{diam}\left(\mathrm{supp}(\theta)\right)\le C2^{-i}$.
Then,
\[
\frac{1}{n}H\left(\theta,\mathcal{D}_{i+n}\right)=\mathbb{E}_{i\le j\le i+n}\left(\frac{1}{m}H\left(\theta_{x,j},\mathcal{D}_{j+m}\right)\right)+O_{X,C}\left(\frac{m}{n}\right).
\]
\end{lem}

\subsection{\label{subsec:Symbolic-related-notations}Symbolic notation}

Let $\Lambda^{*}$ denote the set of finite words over $\Lambda$,
including the empty word $\emptyset$. Given a group $\mathrm{H}$,
indexed elements $\{h_{i}\}_{i\in\Lambda}\subset\mathrm{H}$, and
a word $i_{1}...i_{n}=u\in\Lambda^{*}$, we shall write $h_{u}:=h_{i_{1}}...h_{i_{n}}$,
where $h_{\emptyset}$ denotes the identity element of $\mathrm{H}$.

Let $\Lambda^{\mathbb{N}}$ denote the set of one-sided infinite words
over $\Lambda$. We equip $\Lambda^{\mathbb{N}}$ with the product
topology, where each copy of $\Lambda$ is equipped with the discrete
topology. Let $\sigma:\Lambda^{\mathbb{N}}\rightarrow\Lambda^{\mathbb{N}}$
denote the left-shift map. That is, $\sigma(\omega)=(\omega_{n+1})_{n\ge0}$
for $(\omega_{n})_{n\ge0}=\omega\in\Lambda^{\mathbb{N}}$.

For $n\ge0$ and $\omega\in\Lambda^{\mathbb{N}}$ write $\omega|_{n}$
for the prefix of $\omega$ of length $n$. That is, $\omega|_{n}:=\omega_{0}...\omega_{n-1}$
with $\omega|_{0}:=\emptyset$. Given a word $u\in\Lambda^{n}$, denote
by $[u]$ the cylinder set in $\Lambda^{\mathbb{N}}$ corresponding
to $u$. That is,
\[
[u]:=\left\{ \omega\in\Lambda^{\mathbb{N}}\::\:\omega|_{n}=u\right\} .
\]
We denote by $\mathcal{P}_{n}:=\left\{ [u]\::\:u\in\Lambda^{n}\right\} $
the partition of $\Lambda^{\mathbb{N}}$ into level-$n$ cylinders.
For a set of words $\mathcal{U}\subset\Lambda^{*}$, we write $\left[\mathcal{U}\right]:=\cup_{u\in\mathcal{U}}[u]$.

Let $\beta:=p^{\mathbb{N}}$ denote the Bernoulli measure on $\Lambda^{\mathbb{N}}$
corresponding to $p$. That is, $\beta$ is the unique element in
$\mathcal{M}(\Lambda^{\mathbb{N}})$ such that $\beta([u])=p_{u}$
for each $u\in\Lambda^{*}$.

Given $u,v\in\Lambda^{*}$ and $\omega\in\Lambda^{\mathbb{N}}$, write
$uv$ and $u\omega$ for the concatenation of $u$ with $v$ and of
$u$ with $\omega$, respectively.

For $u\in\Lambda^{*}$ and $\eta>0$, write
\[
Y_{u,\eta}:=\mathbb{CP}^{1}\setminus B\left(L(g_{u}^{-1}),\eta\right).
\]
As in the proof of Lemma \ref{lem:dist of gzC to L(g)}, it is easy
to verify that
\begin{equation}
\Vert g_{u}z\Vert\ge\eta\Vert g_{u}\Vert_{\mathrm{op}}\Vert z\Vert\text{ for }0\ne z\in\mathbb{C}^{2}\text{ with }z\mathbb{C}\in Y_{u,\eta}.\label{eq:lb for norm of g_u z}
\end{equation}

For $u\in\Lambda^{*}$, set
\[
\chi_{u}:=2\log\Vert g_{u}\Vert_{\mathrm{op}}.
\]
Note that
\begin{equation}
\underset{n\rightarrow\infty}{\lim}\frac{1}{n}\chi_{\omega|_{n}}=2\chi\text{ for }\beta\text{-a.e. }\omega,\label{eq:lim on normalized chi_omega|_n}
\end{equation}
where recall from Section \ref{sec:Introduction-and-the main result}
that $\chi$ denotes the Lyapunov exponent associated to $\mathcal{G}$
and $p$.

Given integers $l,n\ge1$ and $0\le j<l$, let $\Psi\left(j,l;n\right)$
denote the set of words $u_{0}...u_{s}\in\Lambda^{*}$ such that $u_{0}\in\Lambda^{j}$,
$u_{i}\in\Lambda^{l}$ for $1\le i\le s$, $\chi_{u_{0}...u_{s}}>n$,
and $\chi_{u_{0}...u_{i}}\le n$ for $0\le i<s$. Note that there
exists a constant $C_{l}>1$, depending only on $\mathcal{G}$ and
$l$, such that
\begin{equation}
2^{n/2}<\Vert g_{u}\Vert_{\mathrm{op}}\le C_{l}2^{n/2}\text{ for all }u\in\Psi\left(j,l;n\right).\label{eq:op norm is comp to 2^(n/2) for u in Psi_n}
\end{equation}
Since $\chi>0$, we have $\beta\left(\left[\Psi\left(j,l;n\right)\right]\right)=1$.
From this, and the relation $\mu=\sum_{i\in\Lambda}p_{i}\cdot g_{i}\mu$,
it follows easily that
\begin{equation}
\mu=\sum_{u\in\Psi\left(j,l;n\right)}p_{u}\cdot g_{u}\mu.\label{eq:mu as conv comb with u in Psi_n}
\end{equation}
We shall write $\Psi_{n}$ in place of $\Psi\left(0,1;n\right)$.

It will sometimes be useful to choose words from $\Lambda^{n}$ and
$\Psi\left(j,l;n\right)$ at random. Let $\mathbf{U}_{n}$ and $\mathbf{I}(j,l;n)$
denote the random words with
\[
\mathbb{P}\left\{ \mathbf{U}_{n}=u\right\} =\begin{cases}
p_{u} & \text{ if }u\in\Lambda^{n}\\
0 & \text{ otherwise}
\end{cases}\:\text{ and }\:\mathbb{P}\left\{ \mathbf{I}(j,l;n)=u\right\} =\begin{cases}
p_{u} & \text{ if }u\in\Psi\left(j,l;n\right)\\
0 & \text{ otherwise}
\end{cases}.
\]
We shall write $\mathbf{I}_{n}$ in place of $\mathbf{I}(j,l;n)$.
Lemma \ref{lem:ac of words} in Section \ref{sec:Entropy-of-projections}
shows why $\Psi_{n}$ and $\mathbf{I}_{n}$ are not sufficient, and
why the more general $\Psi\left(j,l;n\right)$ and $\mathbf{I}(j,l;n)$
are required.

\subsection{\label{subsec:results-from-rand-prod-of-mat}Results from the theory
of random products of matrices}

Recall that $\mathrm{S}_{\mathcal{G}}$ is assumed to be strongly
irreducible and proximal, which implies that $\chi>0$. Moreover,
by \cite[Proposition 4.7]{BQ}, there exists a Borel map $L:\Lambda^{\mathbb{N}}\rightarrow\mathbb{CP}^{1}$,
called the Furstenberg boundary map, such that $L\beta=\mu$ and
\begin{equation}
L(\omega)=\underset{n\rightarrow\infty}{\lim}L\left(g_{\omega|_{n}}\right)\text{ for }\beta\text{-a.e. }\omega.\label{eq:def of L(omega)}
\end{equation}
Consequently, given $l\ge1$ and $0\le j<l$, the sequences of random
directions $\left\{ L\left(g_{\mathbf{U}_{n}}\right)\right\} _{n\ge1}$
and $\left\{ L\left(g_{\mathbf{I}(j,l;n)}\right)\right\} _{n\ge1}$
converge to $\mu$ in distribution. As shown in \cite[Lemma 5.11]{HR},
the boundary map is equivariant in the sense that
\begin{equation}
L(\omega)=g_{\omega_{0}}L\left(\sigma\omega\right)\text{ for }\beta\text{-a.e. }\omega.\label{eq:L is equivariant}
\end{equation}

Since $\mathrm{S}_{\mathcal{G}}$ is strongly irreducible and proximal,
the same holds for the semigroup generated by $\left\{ g_{i}^{t}\right\} _{i\in\Lambda}$,
where $g_{i}^{t}$ denotes the transpose of $g_{i}$. Write $\mu^{t}\in\mathcal{M}\left(\mathbb{CP}^{1}\right)$
for the Furstenberg measure associated to $\left\{ g_{i}^{t}\right\} _{i\in\Lambda}$
and $p$. That is, $\mu^{t}$ is the unique element in $\mathcal{M}\left(\mathbb{CP}^{1}\right)$
such that $\mu^{t}=\sum_{i\in\Lambda}p_{i}\cdot g_{i}^{t}\mu^{t}$.

By \cite[Proposition 4.7]{BQ}, it follows easily that for each $z\mathbb{C}\in\mathbb{CP}^{1}$,
the sequence $\left\{ g_{\mathbf{U}_{n}}^{t}z\mathbb{C}\right\} _{n\ge1}$
converges to $\mu^{t}$ in distribution, where $g_{u}^{t}:=(g_{u})^{t}$
for $u\in\Lambda^{*}$. In the case of real matrices, such a statement
is proved in \cite[Theorem III.4.3]{BL}, and the proof applies without
change here.

By \cite[Lemma 4.6]{BQ}, the measures $\mu$ and $\mu^{t}$ are nonatomic;
that is, $\mu\left\{ z\mathbb{C}\right\} =\mu^{t}\left\{ z\mathbb{C}\right\} =0$
for each $z\mathbb{C}\in\mathbb{CP}^{1}$. The following lemma follows
directly from this, by compactness, and by the aforementioned convergences
in distribution.
\begin{lem}
\label{lem:small ball --> small mass}For each $\epsilon>0$ there
exists $\eta>0$ such that
\[
\mu\left(B(z\mathbb{C},2\eta)\right),\mu^{t}\left(B(z\mathbb{C},2\eta)\right)<\epsilon/2\:\text{ for all }z\mathbb{C}\in\mathbb{CP}^{1}.
\]
Consequently, given $w\mathbb{C}\in\mathbb{CP}^{1}$, there exists
$N\ge1$ such that for all $n\ge N$ and $z\mathbb{C}\in\mathbb{CP}^{1}$,
\[
\mathbb{P}\left\{ L\left(g_{\mathbf{U}_{n}}\right)\in B(z\mathbb{C},\eta)\right\} ,\mathbb{P}\left\{ g_{\mathbf{U}_{n}}^{t}w\mathbb{C}\in B(z\mathbb{C},\eta)\right\} <\epsilon.
\]
Similarly, given $l\ge1$ and $0\le j<l$, there exists $N'\ge1$
such that
\[
\mathbb{P}\left\{ L\left(g_{\mathbf{I}(j,l;n)}\right)\in B(z\mathbb{C},\eta)\right\} <\epsilon\:\text{ for all }n\ge N'\text{ and }z\mathbb{C}\in\mathbb{CP}^{1}.
\]
\end{lem}

\subsection{\label{subsec:Zariski-density-of S_G}Zariski density of $\mathrm{S}_{\mathcal{G}}$}

Write $\mathrm{M}_{2}(\mathbb{C})$ for the vector space of $2\times2$
matrices with entries in $\mathbb{C}$. By a real polynomial function
on $\mathrm{M}_{2}(\mathbb{C})$, we mean a function from $\mathrm{M}_{2}(\mathbb{C})$
to $\mathbb{R}$ which may be expressed as a real polynomial in the
real and imaginary parts of the matrix entries. In what follows, whenever
we refer to the Zariski topology, we mean the Zariski topology generated
by the real polynomial functions. For the definition and basic facts
on the Zariski topology, see for instance \cite[Section 6.1]{BQ}.
\begin{lem}
\label{lem:zariski density}The semigroup $\mathrm{S}_{\mathcal{G}}$
is Zariski dense in $\mathrm{G}$. That is, every real polynomial
function on $\mathrm{M}_{2}(\mathbb{C})$ vanishing on $\mathrm{S}_{\mathcal{G}}$
also vanishes on $\mathrm{G}$.
\end{lem}

\begin{proof}
Write $\mathrm{H}$ for the Zariski closure of $\mathrm{S}_{\mathcal{G}}$.
By \cite[Lemma 6.15]{BQ} it follows that $\mathrm{H}$ is a Lie subgroup
of $\mathrm{G}$. Set $\mathfrak{g}:=\mathfrak{sl}(2,\mathbb{C})\subset\mathrm{M}_{2}(\mathbb{C})$,
and write $\mathfrak{h}\subset\mathfrak{g}$ for the Lie algebra of
$\mathrm{H}$. In order to show that $\mathrm{H}=\mathrm{G}$ and
complete the proof, it suffices to show that $\mathfrak{h}=\mathfrak{g}$.

First, assume by contradiction that $\mathfrak{h}$ is solvable. By
Lie's theorem, this implies that there exists a common eigenvector
in $\mathbb{C}^{2}$ for the elements of $\mathfrak{h}$. Moreover,
by \cite[Theorem 3]{Wh}, it follows that $\mathrm{H}$ has finitely
many connected components with respect to the standard metric topology
of $\mathrm{G}$. The last two facts together imply that $\mathrm{H}$,
and hence $\mathrm{S}_{\mathcal{G}}$, is not strongly irreducible.
But this contradicts our standing assumption, and so $\mathfrak{h}$
cannot be solvable.

Set $\mathfrak{h}':=\mathfrak{h}+i\mathfrak{h}$, and note that $\mathfrak{h}'$
is a complex Lie subalgebra of $\mathfrak{g}$. If $\mathfrak{h}'\ne\mathfrak{g}$,
then $\dim_{\mathbb{C}}\mathfrak{h}'<3$, from which it follows that
$\mathfrak{h}'$ is solvable. But this implies that $\mathfrak{h}$
is also solvable. Hence we must have $\mathfrak{h}'=\mathfrak{g}$,
and in particular $\mathfrak{h}'$ is semisimple. From this, and by
Cartan's criterion of semisimplicity, it follows easily that $\mathfrak{h}$
is also semisimple.

Since $\mathfrak{h}$ is a real semisimple subalgebra of $\mathfrak{g}$,
exactly one of the following holds: $\mathfrak{h}=\mathfrak{g}$,
$\mathfrak{h}$ is isomorphic to $\mathfrak{su}(2)$, or $\mathfrak{h}$
is isomorphic to $\mathfrak{sl}(2,\mathbb{R})$. If $\mathfrak{h}\cong\mathfrak{su}(2)$,
then $\mathrm{H}$ is conjugate to $\mathrm{SU}(2)$, which is impossible
since $\mathrm{S}_{\mathcal{G}}$ is proximal and so $\mathrm{H}$
cannot be compact. If $\mathfrak{h}\cong\mathfrak{sl}(2,\mathbb{R})$,
then $\mathrm{H}$ is conjugate to $\mathrm{SL}(2,\mathbb{R})$ or
to its normalizer $\mathrm{N}_{\mathrm{G}}\left(\mathrm{SL}(2,\mathbb{R})\right)$,
which is equal to the group generated by $\mathrm{SL}(2,\mathbb{R})$
and the element $\mathrm{diag}\left(i,-i\right)$. But, as $\varphi_{g}\left(\mathbb{R}\right)=\mathbb{R}$
for all $g\in\mathrm{N}_{\mathrm{G}}\left(\mathrm{SL}(2,\mathbb{R})\right)$,
this contradicts the assumption that $\mathrm{S}_{\mathcal{G}}$ does
not fix a generalized circle. Hence we must have $\mathfrak{h}=\mathfrak{g}$,
which completes the proof.
\end{proof}

\subsection{The $\nu$-measure of generalized circles}

Write $\mathrm{Circ}(\mathbb{C}_{\infty})$ for the collection of
all generalized circles in $\mathbb{C}_{\infty}$.
\begin{lem}
\label{lem:No Q exists}There does not exist a finite nonempty subset
$\mathcal{Q}$ of $\mathrm{Circ}(\mathbb{C}_{\infty})$ such that
$\varphi_{i}(C)\in\mathcal{Q}$ for all $i\in\Lambda$ and $C\in\mathcal{Q}$.
\end{lem}

\begin{proof}
Assume by contradiction that such a $\mathcal{Q}\subset\mathrm{Circ}(\mathbb{C}_{\infty})$
does exist, which implies that
\begin{equation}
\varphi_{g}(C)\in\mathcal{Q}\text{ for all }g\in\mathrm{S}_{\mathcal{G}}\text{ and }C\in\mathcal{Q}.\label{eq:all g in S_G and C in Q}
\end{equation}

Fix $z\in\mathbb{C}$ belonging to one of the circles in $\mathcal{Q}$.
Given $C\in\mathcal{Q}$, there exists a polynomial $p_{C}\in\mathbb{R}\left[X,Y\right]$,
of degree at most $2$, such that
\[
\left\{ w\in\mathbb{C}\::\:p_{C}\left(\mathrm{Re}(w),\mathrm{Im}(w)\right)=0\right\} =C\setminus\left\{ \infty\right\} .
\]
Let $p_{z,C}:\mathrm{M}_{2}\left(\mathbb{C}\right)\rightarrow\mathbb{C}$
be defined by $p_{z,C}(A)=0$ for $(a_{i,j})=A\in\mathrm{M}_{2}\left(\mathbb{C}\right)$
with $a_{2,1}z+a_{2,2}=0$, and
\[
p_{z,C}(A)=\left|a_{2,1}z+a_{2,2}\right|^{4}p_{C}\left(\mathrm{Re}\left(\frac{a_{1,1}z+a_{1,2}}{a_{2,1}z+a_{2,2}}\right),\mathrm{Im}\left(\frac{a_{1,1}z+a_{1,2}}{a_{2,1}z+a_{2,2}}\right)\right)
\]
for $(a_{i,j})=A\in\mathrm{M}_{2}\left(\mathbb{C}\right)$ with $a_{2,1}z+a_{2,2}\ne0$.
It is easy to verify that $p_{z,C}$ is a real polynomial function
on $\mathrm{M}_{2}\left(\mathbb{C}\right)$, and that for $g\in\mathrm{G}$
\begin{equation}
p_{z,C}(g)=0\text{ if and only if }\varphi_{g}(z)\in C\cup\left\{ \infty\right\} .\label{eq:iff cond}
\end{equation}

Let $q:\mathrm{M}_{2}\left(\mathbb{C}\right)\rightarrow\mathbb{C}$
be the real polynomial function defined by $q(A)=\prod_{C\in\mathcal{Q}}p_{z,C}(A)$
for $A\in\mathrm{M}_{2}\left(\mathbb{C}\right)$. From (\ref{eq:all g in S_G and C in Q})
and (\ref{eq:iff cond}), and since $z\in C$ for some $C\in\mathcal{Q}$,
it follows that $q(g)=0$ for all $g\in\mathrm{S}_{\mathcal{G}}$.
Thus, by Lemma \ref{lem:zariski density}, we have $q(g)=0$ for all
$g\in\mathrm{G}$. This, together with (\ref{eq:iff cond}), implies
that $\varphi_{g}(z)\in\left\{ \infty\right\} \cup\bigcup_{C\in\mathcal{Q}}C$
for all $g\in\mathrm{G}$. But, since $\mathrm{G}$ acts transitively
on $\mathbb{C}_{\infty}$ and $\mathcal{Q}$ is finite, this is clearly
impossible, completing the proof of the lemma.
\end{proof}
\begin{lem}
\label{lem:nu(gen circ)=00003D0}For each generalized circle $C\subset\mathbb{C}_{\infty}$
we have $\nu(C)=0$.
\end{lem}

\begin{proof}
Set
\[
s=\sup\left\{ \nu(C)\::\:C\in\mathrm{Circ}(\mathbb{C}_{\infty})\right\} \text{ and }\mathcal{Q}:=\left\{ C\in\mathrm{Circ}(\mathbb{C}_{\infty})\::\:\nu(C)=s\right\} ,
\]
and assume by contradiction that $s>0$. Since $\mu$ is nonatomic
and $\nu=\psi\mu$, it follows that $\nu$ is also nonatomic. Thus,
$\nu\left(C_{1}\cap C_{2}\right)=0$ for all distinct $C_{1},C_{2}\in\mathrm{Circ}(\mathbb{C}_{\infty})$,
from which it follows that $\mathcal{Q}$ is nonempty and finite. 

Given $C\in\mathcal{Q}$,
\[
s=\nu(C)=\sum_{i\in\Lambda}p_{i}\cdot\nu\left(\varphi_{i}^{-1}(C)\right).
\]
Hence, since $\varphi_{i}^{-1}(C)\in\mathrm{Circ}(\mathbb{C}_{\infty})$
for $i\in\Lambda$, and by the definitions of $s$ and $\mathcal{Q}$,
it follows that $\varphi_{i}^{-1}(C)\in\mathcal{Q}$ for all $i\in\Lambda$.
But this contradicts Lemma \ref{lem:No Q exists}, which completes
the proof of the lemma.
\end{proof}

\subsection{\label{subsec:Exact-dimensionality-and LY formula}Exact dimensionality
and Ledrappier--Young formula}

Given $n\ge1$, recall that $\mathcal{P}_{n}$ denotes the partition
of $\Lambda^{\mathbb{N}}$ into level-$n$ cylinders. Write $\mathcal{B}_{\mathbb{CP}^{1}}$
for the Borel $\sigma$-algebra of $\mathbb{CP}^{1}$, and set
\[
\Delta:=H\left(\beta,\mathcal{P}_{1}\mid L^{-1}\mathcal{B}_{\mathbb{CP}^{1}}\right),
\]
where the right-hand side stands for the entropy of $\beta$ with
respect to $\mathcal{P}_{1}$ conditioned on the $\sigma$-algebra
$L^{-1}\mathcal{B}_{\mathbb{CP}^{1}}$. Let $\left\{ \beta_{\omega}\right\} _{\omega\in\Lambda^{\mathbb{N}}}\subset\mathcal{M}(\Lambda^{\mathbb{N}})$
denote the disintegration of $\beta$ with respect to $L^{-1}\mathcal{B}_{\mathbb{CP}^{1}}$
(for details on disintegrations, see e.g. \cite[Section 5.3]{EiWa}).
\begin{thm}
\label{thm:exact dim of mu and LY}The measure $\mu$ is exact dimensional
with $\dim\mu=\frac{H(p)-\Delta}{2\chi}$. Moreover,
\[
\underset{n\to\infty}{\lim}\frac{1}{n}H\left(\beta_{\omega},\mathcal{P}_{n}\right)=\Delta\text{ for }\beta\text{-a.e. }\omega.
\]
\end{thm}

The proof of Theorem \ref{thm:exact dim of mu and LY}, which is given
in Appendix \ref{sec:appendix}, relies on the results of \cite{rapaport2020exact}.
Note that \cite{rapaport2020exact} deals with Furstenberg measures
on real projective spaces under the standard proximality assumption.
On the other hand, if one considers $\mathrm{G}$ as a subgroup of
$\mathrm{GL}(4,\mathbb{R})$ in the natural way, then the corresponding
action on $\mathbb{RP}^{3}$ is not proximal. For that reason, the
derivation of Theorem \ref{thm:exact dim of mu and LY} from \cite{rapaport2020exact}
is somewhat technical and relies on a different representation of
$\mathrm{G}$.

\section{\label{sec:Uniform-entropy-dimension}Uniform entropy dimension}

In this section we prove Proposition \ref{prop:uni ent dim}. Section
\ref{subsec:Neighborhoods-of-dyadic} establishes a necessary preliminary
statement concerning the $\nu$-measure of neighborhoods of dyadic
cubes, and Section \ref{subsec:Proof-of-Proposition uni ent dim}
contains the proof of Proposition \ref{prop:uni ent dim}.

\subsection{\label{subsec:Neighborhoods-of-dyadic}Neighborhoods of dyadic cubes
have small $\nu$-measure}

The purpose of this subsection is to prove the following proposition.
Recall from Section \ref{subsec:Metric-preliminaries} that, for $r>0$
and a nonempty subset $E$ of a metric space, the closed $r$-neighborhood
of $E$ is denoted by $E^{(r)}$.
\begin{prop}
\label{prop:nu-mass of neigh of dyd cubes}For each $\epsilon>0$
there exists $\delta>0$ such that,
\[
\nu\left(\cup_{D\in\mathcal{D}_{n}^{\mathbb{C}}}(\partial D)^{(\delta2^{-n})}\right)<\epsilon\text{ for all }n\ge1,
\]
where $\partial D$ denotes the boundary of $D$.
\end{prop}

The proof of Proposition \ref{prop:nu-mass of neigh of dyd cubes}
requires the following statement. Recall the sets of words $\Psi_{n}$
defined in Section \ref{subsec:Symbolic-related-notations}.
\begin{lem}
\label{lem:g_u mu(neigh of psi^-1(C))<epsilon}For each $\epsilon>0$,
there exists $\delta>0$ such that $g_{u}\mu\left(\left(\psi^{-1}C\right)^{(\delta2^{-n})}\right)<\epsilon$
for all $n\ge1$, $u\in\Psi_{n}$, and generalized circle $C\subset\mathbb{C}_{\infty}$.
\end{lem}

\begin{proof}
Given $\epsilon>0$, by Lemma \ref{lem:nu(gen circ)=00003D0} and
a compactness argument, there exists $\delta>0$ such that $\mu\left(\left(\psi^{-1}C\right)^{(\delta)}\right)<\epsilon$
for every generalized circle $C\subset\mathbb{C}_{\infty}$. Also
note that $\varphi_{g}(C)$ is a generalized circle for all $g\in\mathrm{G}$
and generalized circle $C\subset\mathbb{C}_{\infty}$. The lemma now
follows from these facts together with Lemma \ref{lem:zC to gzC is bi-Lip with const norm^2},
(\ref{eq:op norm is comp to 2^(n/2) for u in Psi_n}), and (\ref{eq:psi is G-equi}).
\end{proof}
\begin{proof}[Proof of Proposition \ref{prop:nu-mass of neigh of dyd cubes}]
It clearly suffices to prove the proposition for all $n$ sufficiently
large. Let $\epsilon,\eta,\rho,\delta\in(0,1)$, $M>1$, and $n\in\mathbb{Z}_{>0}$
be with
\[
\epsilon^{-1}\ll\eta^{-1}\ll M\ll\rho^{-1}\ll\delta^{-1}\ll n.
\]

Fix $u\in\Psi_{n}$ such that $L(g_{u})\notin B(e_{1}\mathbb{C},2\eta)$,
and set $Y:=Y_{u,\eta}$, where recall that
\[
Y_{u,\eta}=\mathbb{CP}^{1}\setminus B(L(g_{u}^{-1}),\eta).
\]
By Lemma \ref{lem:small ball --> small mass}, we may assume that
$\mu(Y)>1-\epsilon/3$. By Lemma \ref{lem:dist of gzC to L(g)} and
since $u\in\Psi_{n}$,
\begin{equation}
\mathrm{supp}\left(g_{u}\mu_{Y}\right)\subset B\left(L(g_{u}),\eta^{-1}2^{-n}\right).\label{eq:supp contains in ball around L(g_u)}
\end{equation}
Thus, since $\eta^{-1}\ll n$ and $L(g_{u})\notin B(e_{1}\mathbb{C},2\eta)$,
\begin{equation}
\mathrm{supp}\left(g_{u}\mu_{Y}\right)\cap B(e_{1}\mathbb{C},\eta)=\emptyset.\label{supp cap B(e_1C,eta) =00003D empty}
\end{equation}

Note that, by Lemma \ref{lem:bi-lip prop of psi}, the restriction
of $\psi$ to $\mathbb{CP}^{1}\setminus B(e_{1}\mathbb{C},\eta/2)$
is a bi-Lipschitz map with bi-Lipschitz constant depending only on
$\eta$. Since $\eta^{-1}\ll M$, we may assume that this bi-Lipschitz
constant is at most $M$.

Let $C\subset\mathbb{C}_{\infty}$ be a generalized circle, and set
$C_{0}:=C\setminus\{\infty\}$. We have $C_{0}\subset\mathbb{C}$,
and so $C_{0}^{(\delta2^{-n})}$ denotes the closed $\delta2^{-n}$-neighborhood
of $C_{0}$ in $\mathbb{C}$. Given
\[
z\in C_{0}^{(\delta2^{-n})}\setminus\psi\left(B(e_{1}\mathbb{C},\eta)\right),
\]
there exists $w\in C_{0}$ such that $|z-w|\le\delta2^{-n}$. Since
$z\notin\psi\left(B(e_{1}\mathbb{C},\eta)\right)$ and $\eta^{-1}\ll\delta^{-1}$,
we may assume that $w\notin\psi\left(B(e_{1}\mathbb{C},\eta/2)\right)$.
This implies that $d\left(\psi^{-1}(z),\psi^{-1}(w)\right)\le M\delta2^{-n}$,
showing that
\[
\psi^{-1}\left(C_{0}^{(\delta2^{-n})}\setminus\psi\left(B(e_{1}\mathbb{C},\eta)\right)\right)\subset\left(\psi^{-1}C\right)^{(M\delta2^{-n})}.
\]
Thus, from (\ref{supp cap B(e_1C,eta) =00003D empty}), since $M,\rho^{-1}\ll\delta^{-1}$,
and by Lemma \ref{lem:g_u mu(neigh of psi^-1(C))<epsilon},
\[
\psi g_{u}\mu_{Y}\left(C_{0}^{(\delta2^{-n})}\right)\le g_{u}\mu_{Y}\left(\left(\psi^{-1}C\right)^{(M\delta2^{-n})}\right)<\rho.
\]
As this holds for every generalized circle $C\subset\mathbb{C}_{\infty}$,
\begin{equation}
\psi g_{u}\mu_{Y}\left((\partial D)^{(\delta2^{-n})}\right)<4\rho\text{ for all }D\in\mathcal{D}_{n}^{\mathbb{C}}.\label{eq:mass of neigh of bd of D <4rho}
\end{equation}

Additionally, from (\ref{eq:supp contains in ball around L(g_u)})
and (\ref{supp cap B(e_1C,eta) =00003D empty}),
\[
\mathrm{diam}\left(\mathrm{supp}\left(\psi g_{u}\mu_{Y}\right)\right)\le M\eta^{-1}2^{1-n}.
\]
Thus, by Lemma \ref{lem:ub on card of dyad atoms intersecting},
\[
\#\left\{ D\in\mathcal{D}_{n}^{\mathbb{C}}\::\:\mathrm{supp}\left(\psi g_{u}\mu_{Y}\right)\cap(\partial D)^{(\delta2^{-n})}\right\} =O_{\eta,M}(1).
\]
Setting $F:=\cup_{D\in\mathcal{D}_{n}^{\mathbb{C}}}(\partial D)^{(\delta2^{-n})}$,
it follows from this and (\ref{eq:mass of neigh of bd of D <4rho})
that $\psi g_{u}\mu_{Y}(F)=O_{\eta,M}(\rho)$. Hence, from $\epsilon^{-1},\eta^{-1},M\ll\rho^{-1}$
and $\mu(Y)>1-\epsilon/3$,
\begin{equation}
\psi g_{u}\mu(F)<2\epsilon/3\text{ for all }u\in\Psi_{n}\text{ with }L(g_{u})\notin B(e_{1}\mathbb{C},2\eta).\label{eq:psi g_u mu(F)<2epsilon/3}
\end{equation}

Now, from $\epsilon^{-1}\ll\eta^{-1}\ll n$ and by Lemma \ref{lem:small ball --> small mass},
\[
\mathbb{P}\left\{ L\left(g_{\mathbf{I}_{n}}\right)\in B(e_{1}\mathbb{C},2\eta)\right\} <\epsilon/3.
\]
Hence, from (\ref{eq:psi g_u mu(F)<2epsilon/3}) and by the decomposition
$\nu=\mathbb{E}\left(\psi g_{\mathbf{I}_{n}}\mu\right)$, we obtain
that $\nu(F)<\epsilon$, which completes the proof of the proposition.
\end{proof}

\subsection{\label{subsec:Proof-of-Proposition uni ent dim}Proof of Proposition
\ref{prop:uni ent dim}}

The following proposition is the main ingredient in the proof of Proposition
\ref{prop:uni ent dim}.
\begin{prop}
\label{prop:lb on ent of comp of nu}For each $\epsilon>0$, $m\ge M(\epsilon)\ge1$
and $n\ge1$,
\[
\mathbb{P}\left(\frac{1}{m}H\left(\nu_{z,n},\mathcal{D}_{n+m}\right)>\dim\mu-\epsilon\right)>1-\epsilon.
\]
\end{prop}

The proof of Proposition \ref{prop:lb on ent of comp of nu} relies
on the following lemma. Given $u\in\Lambda^{*}$, recall from Section
\ref{subsec:Symbolic-related-notations} that $\chi_{u}:=2\log\Vert g_{u}\Vert_{\mathrm{op}}$.
\begin{lem}
\label{lem:ent of psi g mu}For each $\epsilon>0$, $0<\eta<\eta(\epsilon)$,
and $m\ge M(\epsilon,\eta)\ge1$ the following holds. Let $u\in\Lambda^{*}$
be with $\Vert g_{u}\Vert_{\mathrm{op}}\ge\eta^{-1}$ and $L(g_{u})\notin B(e_{1}\mathbb{C},2\eta)$.
Then,
\[
\frac{1}{m}H\left(\psi g_{u}\mu_{Y_{u,\eta}},\mathcal{D}_{\chi_{u}+m}\right)>\dim\mu-\epsilon.
\]
\end{lem}

\begin{proof}
Let $\epsilon,\eta\in(0,1)$ and $m\in\mathbb{Z}_{>0}$ be with $\epsilon^{-1}\ll\eta^{-1}\ll m$.
Fix $u\in\Lambda^{*}$ such that $\Vert g_{u}\Vert_{\mathrm{op}}\ge\eta^{-1}$
and $L(g_{u})\notin B(e_{1}\mathbb{C},2\eta)$, and set $Y:=Y_{u,\eta}$.
By Lemma \ref{lem:small ball --> small mass}, we may assume that
$\mu(Y)>1-\epsilon$. Set $D:=\mathrm{diag}\left(\Vert g_{u}\Vert_{\mathrm{op}}^{-1},\Vert g_{u}\Vert_{\mathrm{op}}\right)$,
and let $U,V\in\mathrm{SU}(2)$ be such that $g_{u}=UDV$.

By Lemma \ref{lem:dist of gzC to L(g)} and since $\Vert g_{u}\Vert_{\mathrm{op}}\ge\eta^{-1}$,
\begin{equation}
\mathrm{supp}\left(g_{u}\mu_{Y}\right)\subset B\left(L(g_{u}),\eta^{-1}\Vert g_{u}\Vert_{\mathrm{op}}^{-2}\right)\subset B\left(L(g_{u}),\eta\right).\label{eq:supp(g_u mu_E) subset of ball}
\end{equation}
Thus, since $L(g_{u})\notin B(e_{1}\mathbb{C},2\eta)$,
\[
\mathrm{supp}\left(g_{u}\mu_{Y}\right)\cap B(e_{1}\mathbb{C},\eta)=\emptyset.
\]
From this, by Lemmas \ref{lem:bi-lip prop of psi} and \ref{lem:dyad ent =000026 lip func},
and since $z\mathbb{C}\mapsto Uz\mathbb{C}$ is an isometry of $\mathbb{CP}^{1}$,
\begin{equation}
\left|\frac{1}{m}H\left(\psi g_{u}\mu_{Y},\mathcal{D}_{\chi_{u}+m}\right)-\frac{1}{m}H\left(DV\mu_{Y},\mathcal{D}_{\chi_{u}+m}\right)\right|<\epsilon.\label{eq:step 1 in ent of psi g mu lem}
\end{equation}

Since $L(g_{u})=Ue_{2}\mathbb{C}$ and by (\ref{eq:supp(g_u mu_E) subset of ball}),
\[
\mathrm{supp}\left(DV\mu_{Y}\right)\subset B\left(e_{2}\mathbb{C},\eta\right).
\]
Hence, by Lemmas \ref{lem:bi-lip prop of psi} and \ref{lem:dyad ent =000026 lip func},
and since $\psi\left(Dz\mathbb{C}\right)=\Vert g_{u}\Vert_{\mathrm{op}}^{-2}\psi\left(z\mathbb{C}\right)$
for $z\in\mathbb{C}^{2}\setminus\{e_{1}\mathbb{C}\}$,
\[
\left|\frac{1}{m}H\left(DV\mu_{Y},\mathcal{D}_{\chi_{u}+m}\right)-\frac{1}{m}H\left(S_{\Vert g_{u}\Vert_{\mathrm{op}}^{-2}}\psi V\mu_{Y},\mathcal{D}_{\chi_{u}+m}\right)\right|<\epsilon.
\]
Thus, since $\chi_{u}=2\log\Vert g_{u}\Vert_{\mathrm{op}}$ and $\epsilon^{-1}\ll m$,
\begin{equation}
\left|\frac{1}{m}H\left(DV\mu_{Y},\mathcal{D}_{\chi_{u}+m}\right)-\frac{1}{m}H\left(\psi V\mu_{Y},\mathcal{D}_{m}\right)\right|<2\epsilon.\label{eq:step 2 in ent of psi g mu lem}
\end{equation}

We have $L(g_{u}^{-1})=V^{-1}e_{1}\mathbb{C}$, and so
\[
\mathrm{supp}\left(V\mu_{Y}\right)\cap B(e_{1}\mathbb{C},\eta/2)=\emptyset.
\]
From this and by Lemmas \ref{lem:bi-lip prop of psi} and \ref{lem:dyad ent =000026 lip func},
\begin{equation}
\left|\frac{1}{m}H\left(\psi V\mu_{Y},\mathcal{D}_{m}\right)-\frac{1}{m}H\left(\mu_{Y},\mathcal{D}_{m}\right)\right|<\epsilon.\label{eq:step 3 in ent of psi g mu lem}
\end{equation}

By Lemma \ref{lem:from dim to ent},
\[
\left|\frac{1}{m}H\left(\mu,\mathcal{D}_{m}\right)-\dim\mu\right|<\epsilon.
\]
Hence, by the almost-convexity of entropy (see Section \ref{subsec:Entropy}),
\[
\mu(Y)\frac{1}{m}H\left(\mu_{Y},\mathcal{D}_{m}\right)+\mu(Y^{c})\frac{1}{m}H\left(\mu_{Y^{c}},\mathcal{D}_{m}\right)+\frac{1}{m}>\dim\mu-\epsilon,
\]
where $Y^{c}:=\mathbb{CP}^{1}\setminus Y$. From this, since $\mu(Y^{c})<\epsilon$,
and from (\ref{eq:ub on card of dyd part of CP^1}),
\[
\frac{1}{m}H\left(\mu_{Y},\mathcal{D}_{m}\right)>\dim\mu-O(\epsilon).
\]
The lemma now follows from the last inequality and from (\ref{eq:step 1 in ent of psi g mu lem}),
(\ref{eq:step 2 in ent of psi g mu lem}), and (\ref{eq:step 3 in ent of psi g mu lem}).
\end{proof}
\begin{proof}[Proof of Proposition \ref{prop:lb on ent of comp of nu}]
Let $\epsilon,\eta\in(0,1)$ and $k,m,n\in\mathbb{Z}_{>0}$ be with
$\epsilon^{-1}\ll\eta^{-1}\ll k\ll m$. Let $\mathcal{U}_{1}$ be
the set of all words $u\in\Psi_{n+k}$ such that $L(g_{u})\notin B\left(e_{1}\mathbb{C},2\eta\right)$.
For each $u\in\mathcal{U}_{1}$ set $Y_{u}:=Y_{u,\eta}$. From $\epsilon^{-1}\ll\eta^{-1}\ll k$,
and by Lemma \ref{lem:small ball --> small mass}, we have $\beta\left(\left[\mathcal{U}_{1}\right]\right)>1-\epsilon$
and $\mu(Y_{u})>1-\epsilon/2$ for $u\in\mathcal{U}_{1}$, where recall
that $\left[\mathcal{U}_{1}\right]:=\cup_{u\in\mathcal{U}_{1}}[u]$.

Let $u\in\mathcal{U}_{1}$ be given. By Lemma \ref{lem:dist of gzC to L(g)}
and since $u\in\Psi_{n+k}$,
\[
\mathrm{supp}\left(g_{u}\mu_{Y_{u}}\right)\subset B\left(L(g_{u}),\eta^{-1}2^{-n-k}\right).
\]
Thus, since $\eta^{-1}\ll k$ and $L(g_{u})\notin B(e_{1}\mathbb{C},2\eta)$,
\[
\mathrm{supp}\left(g_{u}\mu_{Y_{u}}\right)\cap B(e_{1}\mathbb{C},\eta)=\emptyset.
\]
From these facts, by Lemma \ref{lem:bi-lip prop of psi}, and since
$\eta^{-1}\ll k$, we obtain
\begin{equation}
\mathrm{diam}\left(\mathrm{supp}\left(\psi g_{u}\mu_{Y_{u}}\right)\right)<\eta2^{-n}\text{ for }u\in\mathcal{U}_{1}.\label{eq:ub on diam(supp(psi g mu_E))}
\end{equation}

Let $\mathcal{U}_{2}$ be the set of all $u\in\mathcal{U}_{1}$ for
which there exists $D\in\mathcal{D}_{n}^{\mathbb{C}}$ such that $\mathrm{supp}\left(\psi g_{u}\mu_{Y_{u}}\right)\subset D$.
Setting
\[
F:=\cup_{D\in\mathcal{D}_{n}^{\mathbb{C}}}(\partial D)^{(\eta2^{-n})},
\]
it clearly follows from (\ref{eq:ub on diam(supp(psi g mu_E))}) that
$\psi g_{u}\mu_{Y_{u}}(F)=1$ for $u\in\mathcal{U}_{1}\setminus\mathcal{U}_{2}$.
Additionally, by Proposition \ref{prop:nu-mass of neigh of dyd cubes}
and since $\epsilon^{-1}\ll\eta^{-1}$, we have $\nu(F)<\epsilon$.
Thus, by (\ref{eq:mu as conv comb with u in Psi_n}) and since $\mu(Y_{u})>1-\epsilon/2>1/2$
for $u\in\mathcal{U}_{1}$,
\[
\epsilon>\psi\mu(F)=\sum_{u\in\Psi_{n}}p_{u}\cdot\psi g_{u}\mu(F)>\frac{1}{2}\sum_{u\in\mathcal{U}_{1}\setminus\mathcal{U}_{2}}p_{u}\cdot\psi g_{u}\mu_{Y_{u}}(F)=\frac{1}{2}\beta\left(\left[\mathcal{U}_{1}\setminus\mathcal{U}_{2}\right]\right).
\]
Since $\beta\left(\left[\mathcal{U}_{1}\right]\right)>1-\epsilon$,
this implies that $\beta\left(\left[\mathcal{U}_{2}\right]\right)>1-3\epsilon.$

Setting $q:=\sum_{u\in\mathcal{U}_{2}}p_{u}\mu(Y_{u})$,
\[
\nu_{1}:=\frac{1}{q}\sum_{u\in\mathcal{U}_{2}}p_{u}\mu(Y_{u})\cdot\psi g_{u}\mu_{Y_{u}},\text{ and }\nu_{2}:=\frac{1}{1-q}\left(\nu-q\nu_{1}\right),
\]
we have $\nu=q\nu_{1}+(1-q)\nu_{2}$ and $q>1-4\epsilon$. Let $\mathcal{E}$
denote the set of all $D\in\mathcal{D}_{n}^{\mathbb{C}}$ such that
$2\epsilon^{1/2}\nu(D)>(1-q)\nu_{2}(D)$. Since $q>1-4\epsilon$ and
by Markov's inequality,
\[
4\epsilon>\sum_{D\in\mathcal{D}_{n}^{\mathbb{C}}}\nu(D)\frac{(1-q)\nu_{2}(D)}{\nu(D)}\ge2\epsilon^{1/2}\cdot\nu\left(\bigcup(\mathcal{D}_{n}^{\mathbb{C}}\setminus\mathcal{E})\right),
\]
which implies that $\nu\left(\bigcup\mathcal{E}\right)>1-2\epsilon^{1/2}$.

By the definitions of $\mathcal{U}_{2}$ and $\nu_{1}$, given $D\in\mathcal{D}_{n}^{\mathbb{C}}$
with $\nu_{1}(D)>0$, there exist $u_{1},...,u_{l}\in\mathcal{U}_{1}$
and a probability vector $(\rho_{1},...,\rho_{l})$ such that
\[
(\nu_{1})_{D}=\sum_{i=1}^{l}\rho_{i}\cdot\psi g_{u_{i}}\mu_{Y_{u_{i}}}.
\]
Moreover, from $\epsilon^{-1}\ll\eta^{-1}\ll k\ll m$, from (\ref{eq:op norm is comp to 2^(n/2) for u in Psi_n}),
and by Lemma \ref{lem:ent of psi g mu},
\[
\frac{1}{m}H\left(\psi g_{u}\mu_{Y_{u}},\mathcal{D}_{n+m}\right)>\dim\mu-\epsilon\text{ for }u\in\mathcal{U}_{1}.
\]
Hence, by concavity of entropy,
\begin{equation}
\frac{1}{m}H\left((\nu_{1})_{D},\mathcal{D}_{n+m}\right)>\dim\mu-\epsilon\text{ for }D\in\mathcal{D}_{n}^{\mathbb{C}}\text{ with }\nu_{1}(D)>0.\label{eq:lb of ent of(nu_1)_D}
\end{equation}

Let $D\in\mathcal{E}$, and note that
\[
\nu_{D}=\frac{q\nu_{1}(D)}{\nu(D)}(\nu_{1})_{D}+\frac{(1-q)\nu_{2}(D)}{\nu(D)}(\nu_{2})_{D}.
\]
From this equality and by the definition of $\mathcal{E}$, we obtain
$\nu(D)^{-1}q\nu_{1}(D)>1-2\epsilon^{1/2}$. Thus, by concavity and
from (\ref{eq:lb of ent of(nu_1)_D}),
\[
\frac{1}{m}H\left(\nu_{D},\mathcal{D}_{n+m}\right)\ge\frac{q\nu_{1}(D)}{\nu(D)}\frac{1}{m}H\left((\nu_{1})_{D},\mathcal{D}_{n+m}\right)>\left(1-2\epsilon^{1/2}\right)\left(\dim\mu-\epsilon\right).
\]
As this holds for all $D\in\mathcal{E}$, and since $\nu\left(\bigcup\mathcal{E}\right)>1-2\epsilon^{1/2}$,
this completes the proof of the proposition.
\end{proof}
We can now prove Proposition \ref{prop:uni ent dim}, which is the
following statement.
\begin{prop*}
For every $\epsilon>0$, $m\ge M(\epsilon)\ge1$ and $n\ge N(\epsilon,m)\ge1$,
\[
\mathbb{P}_{1\le i\le n}\left\{ \left|\frac{1}{m}H\left(\nu_{z,i},\mathcal{D}_{i+m}\right)-\dim\mu\right|<\epsilon\right\} >1-\epsilon.
\]
\end{prop*}
\begin{proof}
Let $\epsilon\in(0,1)$, $R>1$, and $m,n\in\mathbb{Z}_{>0}$ be with
$\epsilon^{-1}\ll R\ll m\ll n$. Setting $B:=\left\{ z\in\mathbb{C}\::\:|z|\le R\right\} $,
it follows from $\epsilon^{-1}\ll R$ that $\nu(B)>1-\epsilon$.

Since $\mu\left(\psi^{-1}(B)\right)=\nu(B)>0$ and $\mu$ is exact
dimensional, $\mu_{\psi^{-1}(B)}$ is also exact dimensional with
dimension $\dim\mu$. Hence, by Lemmas \ref{lem:bi-lip prop of psi},
\ref{lem:from dim to ent} and \ref{lem:dyad ent =000026 lip func},
and since $\nu_{B}=\psi\mu_{\psi^{-1}(B)}$ and $\epsilon^{-1},R\ll n$,
\[
\left|\frac{1}{n}H\left(\nu_{B},\mathcal{D}_{n}\right)-\dim\mu\right|<\epsilon.
\]
Thus, by Lemma \ref{lem:glob ent to loc ent} and from $R,m\ll n$,
\[
\mathbb{E}_{1\le i\le n}\left(\frac{1}{m}H\left(\nu_{B},\mathcal{D}_{i+m}\mid\mathcal{D}_{i}\right)\right)=\dim\mu+O(\epsilon).
\]
From this, since $\nu(B)>1-\epsilon$, by concavity and almost-convexity
(see Section \ref{subsec:Entropy}), and from (\ref{eq:norm cond ent on C <=00003D 2}),
\[
\mathbb{E}_{1\le i\le n}\left(\frac{1}{m}H\left(\nu_{z,i},\mathcal{D}_{i+m}\right)\right)=\dim\mu+O(\epsilon).
\]

Additionally, by Proposition \ref{prop:lb on ent of comp of nu},
\[
\mathbb{P}_{1\le i\le n}\left(\frac{1}{m}H\left(\nu_{z,i},\mathcal{D}_{i+m}\right)>\dim\mu-\epsilon\right)>1-\epsilon.
\]
The proposition now follows directly from the last two formulas (by
starting with a smaller $\epsilon$).
\end{proof}

\section{\label{sec:Entropy-of-projections}Entropy of projections of components
of $\nu$}

In this section we prove Proposition \ref{prop:lb on ent of proj of comp of nu}.
Most of the argument is devoted to establishing the following statement.
\begin{prop}
\label{prop:lb on ent of proj of cylinders of nu}Suppose that $\dim\mu<2$.
Then there exist $\gamma,\eta_{0}\in(0,1)$ such that for every $0<\eta<\eta_{0}$,
$n\ge N(\eta)\ge1$, $z\mathbb{R}\in\mathbb{RP}^{1}$, and $u\in\Lambda^{*}$
with $\Vert g_{u}\Vert_{\mathrm{op}}\ge\eta^{-1}$ and $L(g_{u})\notin B\left(e_{1}\mathbb{C},2\eta\right)$,
\[
\frac{1}{n}H\left(\pi_{z\mathbb{R}}\varphi_{u}\nu,\mathcal{D}_{\chi_{u}+n}\mid\mathcal{D}_{\chi_{u}}\right)\ge\dim\mu-1+\gamma.
\]
\end{prop}

The proof of Proposition \ref{prop:lb on ent of proj of cylinders of nu}
follows the overview of the argument given in Section \ref{subsec:About-the-proof}.
In particular, the proof involves bounding from below entropies of
the form,
\begin{equation}
\frac{1}{m}H\left(\pi_{z\mathbb{R}}\varphi_{uv}\nu,\mathcal{D}_{\chi_{u}+\chi_{v}+m}\mid\mathcal{D}_{\chi_{u}+\chi_{v}}\right)\label{eq:entropies needed for prop}
\end{equation}
with $u,v\in\Lambda^{*}$, where $u$ is as in the statement of Proposition
\ref{prop:lb on ent of proj of cylinders of nu}.

In Section \ref{subsec:The-trivial-lower bd}, we show that most of
the entropies in (\ref{eq:entropies needed for prop}) are bounded
from below by $\dim\mu-1$ up to an arbitrarily small error. Section
\ref{subsec:The-direction-cocycle} is devoted to the study of the
direction cocycle $\alpha_{n}:\Lambda^{\mathbb{N}}\rightarrow\mathbb{RP}^{1}$
(defined in that section). We prove that it is not a coboundary, and
use this to derive an important non-concentration corollary (Corollary
\ref{cor:from equidist prop}). In Section \ref{subsec:The-nontrivial-lower lb},
we use this corollary in order to show that, when $v$ is chosen randomly
according to $\mathbb{E}_{1\le i\le n}\left(\delta_{\mathbf{U}_{i}}\right)$,
the entropies in (\ref{eq:entropies needed for prop}) are, with nonnegligible
probability, bounded from below by $\frac{1}{2}\dim\mu-\epsilon$,
where $\epsilon>0$ is arbitrarily small. In Section \ref{subsec:A-technical-lemma on random words},
we prove a lemma concerning random words, which implies the same conclusion
when the random words $\mathbf{I}(j,l;i)$ are used in place of $\mathbf{U}_{i}$.
Finally, in Section \ref{subsec:Proof-of-Propositions}, we complete
the proofs of Propositions \ref{prop:lb on ent of proj of comp of nu}
and \ref{prop:lb on ent of proj of cylinders of nu}.

\subsection{\label{subsec:The-trivial-lower bd}The trivial lower bound}

The purpose of this subsection is to prove Lemma \ref{lem:triv lb},
stated below. First we need some preliminary statements.
\begin{lem}
\label{lem:lb on ent wrt mod dyad part}For every $\epsilon>0$, $0<\eta<\eta(\epsilon)$,
and $m\ge M(\epsilon,\eta)\ge1$ the following holds. Let $u\in\Lambda^{*}$
be with $\Vert g_{u}\Vert_{\mathrm{op}}\ge\eta^{-1}$ and $L(g_{u})\notin B(e_{1}\mathbb{C},2\eta)$,
and let $z\mathbb{R},w\mathbb{R}\in\mathbb{RP}^{1}$ be with $d\left(z\mathbb{R},w\mathbb{R}\right)\ge\eta$.
Then,
\[
\frac{1}{m}H\left(\psi g_{u}\mu_{Y_{u,\eta}},\pi_{z\mathbb{R}}^{-1}\mathcal{D}_{\chi_{u}+m}\vee\pi_{w\mathbb{R}}^{-1}\mathcal{D}_{\chi_{u}+m}\right)>\dim\mu-\epsilon.
\]
\end{lem}

\begin{rem}
\label{rem:is a comp supp measure on C}Note that by the assumptions
on $u$, by Lemmas \ref{lem:bi-lip prop of psi} and \ref{lem:dist of gzC to L(g)},
and by an argument used a number of times in Section \ref{sec:Uniform-entropy-dimension}
(see e.g. the proof of Proposition \ref{prop:nu-mass of neigh of dyd cubes}),
it follows that $\psi g_{u}\mu_{Y_{u,\eta}}\in\mathcal{M}\left(\mathbb{C}\right)$
with
\[
\mathrm{diam}\left(\mathrm{supp}\left(\psi g_{u}\mu_{Y_{u,\eta}}\right)\right)=O_{\eta}\left(\Vert g_{u}\Vert_{\mathrm{op}}^{-2}\right).
\]
\end{rem}

\begin{proof}
Let $\epsilon,\eta\in(0,1)$ and $m\in\mathbb{Z}_{>0}$ be such that
$\epsilon^{-1}\ll\eta^{-1}\ll m$, let $u\in\Lambda^{*}$ be with
$\Vert g_{u}\Vert_{\mathrm{op}}\ge\eta^{-1}$ and $L(g_{u})\notin B(e_{1}\mathbb{C},2\eta)$,
let $z\mathbb{R},w\mathbb{R}\in\mathbb{RP}^{1}$ be with $d\left(z\mathbb{R},w\mathbb{R}\right)\ge\eta$,
and set
\[
\mathcal{E}:=\pi_{z\mathbb{R}}^{-1}\mathcal{D}_{\chi_{u}+m}\vee\pi_{w\mathbb{R}}^{-1}\mathcal{D}_{\chi_{u}+m}.
\]
From $d\left(z\mathbb{R},w\mathbb{R}\right)\ge\eta$ it follows easily
that the partitions $\mathcal{E}$ and $\mathcal{D}_{\chi_{u}+m}^{\mathbb{C}}$
are $O_{\eta}(1)$-commensurable. That is, for each $E\in\mathcal{E}$
and $D\in\mathcal{D}_{\chi_{u}+m}^{\mathbb{C}}$
\[
\#\left\{ D'\in\mathcal{D}_{\chi_{u}+m}^{\mathbb{C}}\::\:D'\cap E\ne\emptyset\right\} ,\#\left\{ E'\in\mathcal{E}\::\:E'\cap D\ne\emptyset\right\} =O_{\eta}(1).
\]
Hence, by \cite[Lemma 3.2]{Ho1} and since $\epsilon^{-1},\eta^{-1}\ll m$,
\[
\left|\frac{1}{m}H\left(\psi g_{u}\mu_{Y_{u,\eta}},\mathcal{E}\right)-\frac{1}{m}H\left(\psi g_{u}\mu_{Y_{u,\eta}},\mathcal{D}_{\chi_{u}+m}\right)\right|<\frac{\epsilon}{2}.
\]
Moreover, by Lemma \ref{lem:ent of psi g mu},
\[
\frac{1}{m}H\left(\psi g_{u}\mu_{Y_{u,\eta}},\mathcal{D}_{\chi_{u}+m}\right)>\dim\mu-\epsilon/2,
\]
which completes the proof.
\end{proof}
\begin{lem}
\label{lem:pre triv lb}For every $\epsilon>0$, $0<\eta<\eta(\epsilon)$,
$m\ge M(\epsilon,\eta)\ge1$, $z\mathbb{R}\in\mathbb{RP}^{1}$, and
$u\in\Lambda^{*}$ with $\Vert g_{u}\Vert_{\mathrm{op}}\ge\eta^{-1}$
and $L(g_{u})\notin B(e_{1}\mathbb{C},2\eta)$,
\[
\frac{1}{m}H\left(\pi_{z\mathbb{R}}\varphi_{u}\nu,\mathcal{D}_{\chi_{u}+m}\mid\mathcal{D}_{\chi_{u}}\right)>\dim\mu-1-\epsilon.
\]
\end{lem}

\begin{proof}
Let $\epsilon,\eta\in(0,1)$ and $m\in\mathbb{Z}_{>0}$ be such that
$\epsilon^{-1}\ll\eta^{-1}\ll m$, let $z\mathbb{R}\in\mathbb{RP}^{1}$,
let $u\in\Lambda^{*}$ be with $\Vert g_{u}\Vert_{\mathrm{op}}\ge\eta^{-1}$
and $L(g_{u})\notin B(e_{1}\mathbb{C},2\eta)$, and set
\[
H:=\frac{1}{m}H\left(\pi_{z\mathbb{R}}\varphi_{u}\nu,\mathcal{D}_{\chi_{u}+m}\mid\mathcal{D}_{\chi_{u}}\right).
\]

Since $\epsilon^{-1}\ll\eta^{-1}$, we may assume that $\mu\left(Y_{u,\eta}\right)>1-\epsilon$.
Hence, by concavity of conditional entropy, from $\varphi_{u}\nu=\psi g_{u}\mu$,
and from (\ref{eq:norm cond ent on C <=00003D 2}), we obtain that
\[
H\ge\frac{1}{m}H\left(\pi_{z\mathbb{R}}\psi g_{u}\mu_{Y_{u,\eta}},\mathcal{D}_{\chi_{u}+m}\mid\mathcal{D}_{\chi_{u}}\right)-2\epsilon.
\]
Thus, by Remark \ref{rem:is a comp supp measure on C},
\[
H\ge\frac{1}{m}H\left(\psi g_{u}\mu_{Y_{u,\eta}},\pi_{z\mathbb{R}}^{-1}\mathcal{D}_{\chi_{u}+m}\right)-3\epsilon.
\]

Let $\left(z\mathbb{R}\right)^{\perp}\in\mathbb{RP}^{1}$ denote the
line perpendicular to $z\mathbb{R}$, and set
\[
\mathcal{E}:=\pi_{z\mathbb{R}}^{-1}\mathcal{D}_{\chi_{u}+m}\vee\pi_{\left(z\mathbb{R}\right)^{\perp}}^{-1}\mathcal{D}_{\chi_{u}+m}.
\]
From the last inequality and by the conditional entropy formula,
\[
H\ge\frac{1}{m}H\left(\psi g_{u}\mu_{Y_{u,\eta}},\mathcal{E}\right)-\frac{1}{m}H\left(\psi g_{u}\mu_{Y_{u,\eta}},\mathcal{E}\mid\pi_{z\mathbb{R}}^{-1}\mathcal{D}_{\chi_{u}+m}\right)-3\epsilon.
\]
By Lemma \ref{lem:lb on ent wrt mod dyad part},
\[
\frac{1}{m}H\left(\psi g_{u}\mu_{Y_{u,\eta}},\mathcal{E}\right)>\dim\mu-\epsilon.
\]
Additionally, using $\epsilon^{-1},\eta^{-1}\ll m$ and Remark \ref{rem:is a comp supp measure on C},
it is easy to verify that
\[
\frac{1}{m}H\left(\psi g_{u}\mu_{Y_{u,\eta}},\mathcal{E}\mid\pi_{z\mathbb{R}}^{-1}\mathcal{D}_{\chi_{u}+m}\right)\le1+\epsilon.
\]
All of this completes the proof of the lemma.
\end{proof}
\begin{lem}
\label{lem:triv lb}For every $\epsilon>0$, $0<\eta<\eta(\epsilon)$,
and $m\ge M(\epsilon,\eta)\ge1$ the following holds. Let $u,v\in\Lambda^{*}$
be with $\Vert g_{u}\Vert_{\mathrm{op}}\ge\eta^{-1}$, $L(g_{u})\notin B(e_{1}\mathbb{C},2\eta)$,
$\Vert g_{v}\Vert_{\mathrm{op}}\ge3\eta^{-2}$, and $L(g_{v})\in Y_{u,2\eta}$.
Then for every $z\mathbb{R}\in\mathbb{RP}^{1}$,
\[
\frac{1}{m}H\left(\pi_{z\mathbb{R}}\varphi_{uv}\nu,\mathcal{D}_{\chi_{u}+\chi_{v}+m}\mid\mathcal{D}_{\chi_{u}+\chi_{v}}\right)\ge\dim\mu-1-\epsilon.
\]
\end{lem}

\begin{proof}
Let $\epsilon,\eta\in(0,1)$ and $m\in\mathbb{Z}_{>0}$ be such that
$\epsilon^{-1}\ll\eta^{-1}\ll m$, let $u,v\in\Lambda^{*}$ be such
that the conditions in the statement of the lemma are satisfied, and
fix $z\mathbb{R}\in\mathbb{RP}^{1}$. We may assume that $\eta$ is
sufficiently small so that $\mu\left(Y_{v,\eta}\cap Y_{uv,\eta}\right)>0$.
Let $w\in\mathbb{C}^{2}$ be a unit vector with $w\mathbb{C}\in Y_{v,\eta}\cap Y_{uv,\eta}$.
By Lemma \ref{lem:dist of gzC to L(g)},
\[
d\left(L(g_{v}),g_{v}w\mathbb{C}\right)\le\eta^{-1}\Vert g_{v}\Vert_{\mathrm{op}}^{-2}\le\eta^{3}/9.
\]
Together with $L(g_{v})\in Y_{u,2\eta}$, this implies that $g_{v}w\mathbb{C}\in Y_{u,\eta}$.
Thus, from (\ref{eq:lb for norm of g_u z}),
\begin{equation}
\Vert g_{uv}\Vert_{\mathrm{op}}\ge\Vert g_{uv}w\Vert\ge\eta^{2}\Vert g_{u}\Vert_{\mathrm{op}}\Vert g_{v}\Vert_{\mathrm{op}}\ge3\eta^{-1}.\label{eq:lb on norm of g_uv}
\end{equation}

By Lemma \ref{lem:dist of gzC to L(g)} and since $w\mathbb{C}\in Y_{uv,\eta}$,
\[
d\left(L(g_{uv}),g_{uv}w\mathbb{C}\right)\le\eta^{-1}\Vert g_{uv}\Vert_{\mathrm{op}}^{-2}\le\eta/9.
\]
Similarly, since $g_{v}w\mathbb{C}\in Y_{u,\eta}$,
\[
d\left(L(g_{u}),g_{uv}w\mathbb{C}\right)\le\eta^{-1}\Vert g_{u}\Vert_{\mathrm{op}}^{-2}\le\eta.
\]
Hence, from $L(g_{u})\notin B\left(e_{1}\mathbb{C},2\eta\right)$,
we obtain $L(g_{uv})\notin B\left(e_{1}\mathbb{C},8\eta/9\right)$.

By Lemma \ref{lem:pre triv lb} it now follows that,
\[
\frac{1}{m}H\left(\pi_{z\mathbb{R}}\varphi_{uv}\nu,\mathcal{D}_{\chi_{uv}+m}\mid\mathcal{D}_{\chi_{uv}}\right)>\dim\mu-1-\epsilon.
\]
Additionally, from $\Vert g_{uv}\Vert_{\mathrm{op}}\le\Vert g_{u}\Vert_{\mathrm{op}}\Vert g_{v}\Vert_{\mathrm{op}}$
and (\ref{eq:lb on norm of g_uv}),
\[
\chi_{uv}\le\chi_{u}+\chi_{v}\le\chi_{uv}+O_{\eta}(1).
\]
Thus, since $\epsilon^{-1},\eta^{-1}\ll m$,
\begin{multline*}
\frac{1}{m}H\left(\pi_{z\mathbb{R}}\varphi_{uv}\nu,\mathcal{D}_{\chi_{u}+\chi_{v}+m}\mid\mathcal{D}_{\chi_{u}+\chi_{v}}\right)\\
\ge\frac{1}{m}H\left(\pi_{z\mathbb{R}}\varphi_{uv}\nu,\mathcal{D}_{\chi_{uv}+m}\mid\mathcal{D}_{\chi_{uv}}\right)-\frac{1}{m}H\left(\pi_{z\mathbb{R}}\varphi_{uv}\nu,\mathcal{D}_{\chi_{u}+\chi_{v}}\mid\mathcal{D}_{\chi_{uv}}\right)\\
>\dim\mu-1-2\epsilon,
\end{multline*}
which completes the proof of the lemma.
\end{proof}

\subsection{\label{subsec:The-direction-cocycle}The direction cocycle}

Let $\alpha:\Lambda^{\mathbb{N}}\rightarrow\mathbb{RP}^{1}$ be such
that
\[
\alpha(\omega)=\begin{cases}
\varphi_{\omega_{0}}'\left(\psi L\left(\sigma\omega\right)\right)\mathbb{R} & \text{ if }\psi L\left(\sigma\omega\right)\notin\left\{ \infty,\varphi_{\omega_{0}}^{-1}(\infty)\right\} \\
\mathbb{R} & \text{ otherwise}
\end{cases}\text{ for }\omega\in\Lambda^{\mathbb{N}}.
\]
Define a cocycle, which we call the direction cocycle, by setting
\[
\alpha_{n}(\omega):=\prod_{i=0}^{n-1}\alpha\left(\sigma^{i}\omega\right)\text{ for }n\ge0\text{ and }\omega\in\Lambda^{\mathbb{N}},
\]
where recall from Section \ref{subsec:Algebraic-notation} that $\mathbb{RP}^{1}$
is considered as a multiplicative group. Note that, since $L\beta=\mu$
and $\mu$ is nonatomic, $\psi L\left(\sigma\omega\right)\notin\left\{ \infty,\varphi_{\omega_{0}}^{-1}(\infty)\right\} $
for $\beta$-a.e. $\omega$. Thus, by (\ref{eq:L is equivariant})
and the chain rule, for each $n\ge0$ we have
\begin{equation}
\alpha_{n}(\omega)=\varphi_{\omega|_{n}}'\left(\psi L\left(\sigma^{n}\omega\right)\right)\mathbb{R}\text{ for }\beta\text{-a.e. }\omega.\label{eq:alpha_n a.s. equals}
\end{equation}

Our goal in this subsection is to show that, in a certain quantitative
sense, sequences of the form $\left(\alpha_{n}(\omega)h\left(\sigma^{n}\omega\right)\right)_{n\ge0}$,
with $h:\Lambda^{\mathbb{N}}\rightarrow\mathbb{RP}^{1}$ continuous,
do not equidistribute to a mass point. The following statement is
the first step toward this.
\begin{prop}
\label{prop:alpha not coboundary}There does not exist a Borel measurable
map $f:\Lambda^{\mathbb{N}}\rightarrow\mathbb{RP}^{1}$ such that
$\alpha(\omega)=f(\omega)^{-1}f\left(\sigma\omega\right)$ for $\beta$-a.e.
$\omega$.
\end{prop}

\begin{rem*}
In the terminology of measurable cohomology (see \cite{MR578731}),
Proposition \ref{prop:alpha not coboundary} asserts that the cocycle
$\alpha_{n}$ is not a coboundary.
\end{rem*}
\begin{proof}
Assume by contradiction that there exists a Borel measurable $f:\Lambda^{\mathbb{N}}\rightarrow\mathbb{RP}^{1}$
such that $\alpha(\omega)=f(\omega)^{-1}f\left(\sigma\omega\right)$
for $\beta$-a.e. $\omega$. Then by (\ref{eq:alpha_n a.s. equals}),
\[
\varphi_{\omega|_{n}}'\left(\psi L\left(\sigma^{n}\omega\right)\right)\mathbb{R}=f(\omega)^{-1}f\left(\sigma^{n}\omega\right)\text{ for all }n\ge0\text{ and }\beta\text{-a.e. }\omega,
\]
which implies that
\begin{equation}
\varphi_{u}'\left(\psi L\left(\omega\right)\right)\mathbb{R}f(u\omega)=f(\omega)\text{ for all }u\in\Lambda^{*}\text{ and }\beta\text{-a.e. }\omega.\label{eq:first eq in comp arg}
\end{equation}

By Lemma \ref{lem:small ball --> small mass}, there exist $\epsilon>0$
and $N\ge1$ such that
\begin{equation}
\mathbb{P}\left\{ g_{\mathbf{U}_{n}}^{t}e_{2}\mathbb{C}\in B\left(z\mathbb{C},\epsilon\right)\right\} <1/2\text{ for all }n\ge N\text{ and }z\mathbb{C}\in\mathbb{CP}^{1},\label{eq:mass in ball < 1/2 in comp arg}
\end{equation}
where, as always, $e_{2}$ denotes the second vector of the standard
basis of $\mathbb{C}^{2}$.

By Lusin's theorem, there exists a compact subset $K$ of $\Lambda^{\mathbb{N}}$
such that $\beta(K)>4/5$ and $f|_{K}$ is continuous. Let $k\ge1$
be given. Since $K$ is compact, there exists $N'\ge1$ such that
$d\left(f(\omega),f(\omega')\right)<1/k$ for all $\omega,\omega'\in K$
with $\omega|_{N'}=\omega'|_{N'}$. By the martingale theorem,
\[
\underset{n\rightarrow\infty}{\lim}\beta_{[\omega|_{n}]}(K)=1\text{ for }\beta\text{-a.e. }\omega\in K.
\]
Hence, there exist $n\ge N'$ and a Borel set $K'\subset K$ such
that $\beta(K')>3/4$ and $\beta_{[\omega|_{n}]}(K)>1-2^{-1-k}$ for
$\omega\in K'$. Since $n\ge N'$ and by the choice of $N'$,
\begin{equation}
\beta_{[\omega|_{n}]}\left\{ \omega'\in\Lambda^{\mathbb{N}}\::\:d\left(f(\omega),f(\omega')\right)<1/k\right\} >1-2^{-1-k}\text{ for all }\omega\in K'.\label{eq:cond meas large mass in comp arg}
\end{equation}

Let $\tilde{f}:\Lambda^{n}\rightarrow\mathbb{RP}^{1}$ be defined
as follows. Given $u\in\Lambda^{n}$ with $[u]\cap K'\ne\emptyset$,
choose some $\omega\in[u]\cap K'$ and set $\tilde{f}(u)=f(\omega)$.
For $u\in\Lambda^{n}$ with $[u]\cap K'=\emptyset$, set $\tilde{f}(u)=\mathbb{R}$.
From (\ref{eq:cond meas large mass in comp arg}) and since $\beta(K')>3/4$,
\[
\mathbb{P}\left\{ \beta\left\{ \omega\in\Lambda^{\mathbb{N}}\::\:d\left(\tilde{f}\left(\mathbf{U}_{n}\right),f\left(\mathbf{U}_{n}\omega\right)\right)<1/k\right\} >1-2^{-1-k}\right\} >3/4.
\]

From the last inequality and from (\ref{eq:first eq in comp arg})
and (\ref{eq:mass in ball < 1/2 in comp arg}), it follows that for
each $k\ge1$ there exist $n_{k}\ge1$, $u_{k,1},u_{k,2}\in\Lambda^{n_{k}}$
with
\[
d\left(g_{u_{k,1}}^{t}e_{2}\mathbb{C},g_{u_{k,2}}^{t}e_{2}\mathbb{C}\right)\ge\epsilon,
\]
and $z_{k,1},z_{k,2}\in\mathbb{C}$ with $|z_{k,1}|=|z_{k,2}|=1$,
so that $\beta(E_{k})>1-2^{-k}$, where $E_{k}$ is the set of all
$\omega\in\Lambda^{\mathbb{N}}$ such that
\begin{equation}
d\left(z_{k,j}\mathbb{R},f(u_{k,j}\omega)\right)<1/k\;\text{ and }\;\varphi_{u_{k,j}}'\left(\psi L\left(\omega\right)\right)\mathbb{R}f(u_{k,j}\omega)=f(\omega)\;\text{ for }j=1,2.\label{eq:def prop of E_k in comp arg}
\end{equation}

By compactness, and by moving to a subsequence without changing the
notation, we may assume that there exist $w_{1},w_{2}\in\mathbb{C}^{2}$
and $t_{1},t_{2}\in\mathbb{R}$ such that
\begin{equation}
\underset{k\rightarrow\infty}{\lim}\frac{g_{u_{k,j}}^{t}e_{2}}{\left\Vert g_{u_{k,j}}^{t}e_{2}\right\Vert }=w_{j}\;\text{ and }\;\underset{k\rightarrow\infty}{\lim}z_{k,j}=e^{it_{j}}\text{ for }j=1,2.\label{eq:by moving to sub in comp arg}
\end{equation}
We clearly have $d\left(w_{1}\mathbb{C},w_{2}\mathbb{C}\right)\ge\epsilon$,
which implies that $w_{1}$ and $w_{2}$ are linearly independent
over $\mathbb{C}$.

Let $\mathrm{B}:\mathbb{C}^{2}\times\mathbb{C}^{2}\rightarrow\mathbb{C}$
denote the symmetric bilinear form defined by
\[
\mathrm{B}\left(\left(a_{1},a_{2}\right),\left(b_{1},b_{2}\right)\right)=a_{1}b_{1}+a_{2}b_{2}\text{ for }\left(a_{1},a_{2}\right),\left(b_{1},b_{2}\right)\in\mathbb{C}^{2},
\]
and set
\[
E:=\left\{ \omega\in\Lambda^{\mathbb{N}}\::\:\mathrm{B}\left(w_{j},\left(\psi L\left(\omega\right),1\right)\right)\ne0\text{ for }j=1,2\right\} \cap\left(\cup_{m\ge1}\cap_{k\ge m}E_{k}\right).
\]
Since $\nu=\psi L\beta$ is nonatomic, and by the Borel-Cantelli lemma,
$\beta(E)=1$. For $\left(\begin{array}{cc}
a & b\\
c & d
\end{array}\right)=g\in\mathrm{G}$ and $z\in\mathbb{C}\setminus\left\{ \varphi_{g}^{-1}(\infty)\right\} $,
\[
\varphi_{g}'(z)=\frac{1}{(cz+d)^{2}}=\mathrm{B}\left(g^{t}e_{2},(z,1)\right)^{-2}.
\]
Thus, by (\ref{eq:def prop of E_k in comp arg}) and (\ref{eq:by moving to sub in comp arg}),
\begin{equation}
f(\omega)=\mathrm{B}\left(w_{j},\left(\psi L\left(\omega\right),1\right)\right)^{-2}e^{it_{j}}\mathbb{R}\text{ for }\omega\in E\text{ and }j=1,2.\label{eq:f(omega)=00003D in comp arg}
\end{equation}

Write $A$ for the matrix whose rows are $w_{1}$ and $w_{2}$. Since
$w_{1}$ and $w_{2}$ are linearly independent, $A\in\mathrm{GL}(2,\mathbb{C})$.
By (\ref{eq:f(omega)=00003D in comp arg}), for each $\omega\in E$
\[
e^{i(t_{1}-t_{2})}\mathbb{R}=\frac{\mathrm{B}\left(w_{1},\left(\psi L\left(\omega\right),1\right)\right)^{2}}{\mathrm{B}\left(w_{2},\left(\psi L\left(\omega\right),1\right)\right)^{2}}\mathbb{R}=\left(\varphi_{A}\left(\psi L\left(\omega\right)\right)\right)^{2}\mathbb{R},
\]
where $\varphi_{A}$ is the Möbius transformation induced by $A$.
Hence, since $\nu=\psi L\beta$ and $\beta(E)=1$,
\[
\nu\left(\varphi_{A}^{-1}\left(e^{i(t_{1}-t_{2})/2}\mathbb{R}\right)\cup\varphi_{A}^{-1}\left(e^{i(t_{1}-t_{2}+\pi)/2}\mathbb{R}\right)\right)=1.
\]
But this contradicts Lemma \ref{lem:nu(gen circ)=00003D0}, which
completes the proof of the proposition.
\end{proof}
We can now establish the desired non-concentration property of the
sequences $\left(\alpha_{n}(\omega)h\left(\sigma^{n}\omega\right)\right)_{n\ge0}$,
for which we need the following definition.
\begin{defn}
Given $\delta>0$, we say that $\theta\in\mathcal{M}\left(\mathbb{RP}^{1}\right)$
is $\delta$-concentrated if there exists $z\mathbb{R}\in\mathbb{RP}^{1}$
such that $\theta\left(B\left(z\mathbb{R},\delta\right)\right)>1-\delta$.
\end{defn}

\begin{prop}
\label{prop:equidist prop}There exists $\delta>0$ such that for
every continuous $h:\Lambda^{\mathbb{N}}\rightarrow\mathbb{RP}^{1}$
and for $\beta$-a.e. $\omega$, the sequence $\left(\alpha_{n}(\omega)h\left(\sigma^{n}\omega\right)\right)_{n\ge0}$
is equidistributed with respect to some $\theta\in\mathcal{M}\left(\mathbb{RP}^{1}\right)$
that is not $\delta$-concentrated.
\end{prop}

\begin{proof}
Set $X:=\Lambda^{\mathbb{N}}\times\mathbb{RP}^{1}$, and let $T:X\rightarrow X$
and $\pi:X\rightarrow\Lambda^{\mathbb{N}}$ be defined by
\[
Tx=\left(\sigma\omega,\alpha(\omega)z\mathbb{R}\right)\;\text{ and }\;\pi x=\omega\;\text{ for }(\omega,z\mathbb{R})=x\in X.
\]
Writing $m_{\mathbb{RP}^{1}}$ for the normalized Haar measure of
$\mathbb{RP}^{1}$, it holds that $\zeta:=\beta\times m_{\mathbb{RP}^{1}}$
is $T$-invariant. Thus, from $\pi\circ T=\sigma\circ\pi$ and $\pi\zeta=\beta$,
since $\beta$ is $\sigma$-invariant and ergodic, and by considering
the ergodic decomposition of $\zeta$, it follows that there exists
a $T$-invariant and ergodic $\lambda\in\mathcal{M}(X)$ such that
$\pi\lambda=\beta$. Write $\left\{ \delta_{\omega}\times\xi_{\omega}\right\} _{\omega\in\Lambda^{\mathbb{N}}}$
for the disintegration of $\lambda$ over $\Lambda^{\mathbb{N}}$.
That is, $\xi_{\omega}\in\mathcal{M}\left(\mathbb{RP}^{1}\right)$
for $\omega\in\Lambda^{\mathbb{N}}$, and
\[
\lambda=\int\delta_{\omega}\times\xi_{\omega}\:d\beta(\omega).
\]

Given $\delta>0$, write $E_{\delta}$ for the set of $\omega\in\Lambda^{\mathbb{N}}$
for which $\xi_{\omega}$ is $\delta$-concentrated. Assuming by contradiction
that $\beta(E_{\delta})=1$ for all $\delta>0$, it follows that $\xi_{\omega}$
is a mass point for $\beta$-a.e. $\omega$, which implies that there
exists a Borel measurable $f:\Lambda^{\mathbb{N}}\rightarrow\mathbb{RP}^{1}$
such that $\lambda=\int\delta_{\left(\omega,f(\omega)\right)}\:d\beta(\omega)$.
Since $\lambda$ is $T$-invariant,
\[
\lambda=T\lambda=\int\delta_{T\left(\omega,f(\omega)\right)}\:d\beta(\omega)=\int\delta_{\left(\sigma\omega,\alpha(\omega)f(\omega)\right)}\:d\beta(\omega).
\]
Moreover, since $\beta$ is $\sigma$-invariant,
\[
\lambda=\int\delta_{\left(\sigma\omega,f(\sigma\omega)\right)}\:d\beta(\omega).
\]
The last two formulas clearly imply that $\alpha(\omega)=f(\omega)^{-1}f(\sigma\omega)$
for $\beta$-a.e. $\omega$. But this contradicts Proposition \ref{prop:alpha not coboundary},
and so it must hold that $\beta(E_{\delta})<1-\delta$ for some $\delta>0$,
which we fix.

By the ergodic theorem, and since $\pi\lambda=\beta$, for $\beta$-a.e.
$\omega$ there exists $x\in X$ such that $\pi x=\omega$ and $\left(T^{n}x\right)_{n\ge0}$
is equidistributed with respect to $\lambda$. Fix such $\omega$
and $x$, and let $z\mathbb{R}\in\mathbb{RP}^{1}$ be with $x=(\omega,z\mathbb{R})$.
Let $h:\Lambda^{\mathbb{N}}\rightarrow\mathbb{RP}^{1}$ be continuous,
and set
\[
\theta:=\int S_{z^{-1}\mathbb{R}h(\omega')}\xi_{\omega'}\:d\beta(\omega'),
\]
where recall that $S_{w\mathbb{R}}\left(w'\mathbb{R}\right):=ww'\mathbb{R}$
for $w\mathbb{R},w'\mathbb{R}\in\mathbb{RP}^{1}$. Note that, since
$\beta(E_{\delta})<1-\delta$, the probability measure $\theta$ is
not $\delta^{2}$-concentrated.

Let $\phi:\mathbb{RP}^{1}\rightarrow\mathbb{R}$ be continuous, and
let $\tilde{\phi}:X\rightarrow\mathbb{R}$ be defined by
\[
\tilde{\phi}(\omega',w\mathbb{R})=\phi\left(z^{-1}w\mathbb{R}h(\omega')\right)\text{ for }(\omega',w\mathbb{R})\in X.
\]
Since $\tilde{\phi}$ is continuous and $\left(T^{n}x\right)_{n\ge0}$
is equidistributed with respect to $\lambda$,
\begin{multline*}
\underset{n}{\lim}\frac{1}{n}\sum_{j=0}^{n-1}\phi\left(\alpha_{j}(\omega)h\left(\sigma^{j}\omega\right)\right)=\underset{n}{\lim}\frac{1}{n}\sum_{j=0}^{n-1}\tilde{\phi}\left(T^{j}x\right)=\int\tilde{\phi}\:d\lambda\\
=\int\int\tilde{\phi}\left(\omega',w\mathbb{R}\right)\:d\xi_{\omega'}(w\mathbb{R})\:d\beta(\omega')=\int\phi\:d\theta.
\end{multline*}
This shows that the sequence $\left(\alpha_{n}(\omega)h\left(\sigma^{n}\omega\right)\right)_{n\ge0}$
is equidistributed with respect to $\theta$. Since $\theta$ is not
$\delta^{2}$-concentrated, this completes the proof of the proposition.
\end{proof}
The following corollary is an immediate consequence of Proposition
\ref{prop:equidist prop}. Recall that $\lambda_{n}$ denotes the
uniform probability measure on $\mathcal{N}_{n}:=\left\{ 1,...,n\right\} $.
\begin{cor}
\label{cor:from equidist prop}There exists $0<\delta<1$ such that
for every continuous $h:\Lambda^{\mathbb{N}}\rightarrow\mathbb{RP}^{1}$
and for $\beta$-a.e. $\omega$, there exists $N_{h,\omega}\ge1$
so that for every $n\ge N_{h,\omega}$,
\[
\lambda_{n}\left\{ i\in\mathcal{N}_{n}\::\:d\left(\alpha_{i}(\omega)h\left(\sigma^{i}\omega\right),z\mathbb{R}\right)>\delta\right\} >\delta\text{ for all }z\mathbb{R}\in\mathbb{RP}^{1}.
\]
\end{cor}

\subsection{\label{subsec:The-nontrivial-lower lb}The nontrivial lower bound}

The purpose of this subsection is to prove the following proposition.
\begin{prop}
\label{prop:nontriv portion > dim(mu)/2}There exists $0<\delta<1$
such that for every $\epsilon>0$, $0<\eta<\eta(\epsilon)$, $m\ge M(\epsilon,\eta)\ge1$,
$n\ge N(\epsilon,\eta,m)\ge1$, $z\mathbb{R}\in\mathbb{RP}^{1}$,
and $u\in\Lambda^{*}$ with $\Vert g_{u}\Vert_{\mathrm{op}}\ge\eta^{-1}$
and $L(g_{u})\notin B\left(e_{1}\mathbb{C},2\eta\right)$,
\[
\mathbb{P}_{1\le i\le n}\left\{ \frac{1}{m}H\left(\pi_{z\mathbb{R}}\varphi_{u\mathbf{U}_{i}}\nu,\mathcal{D}_{\chi_{u}+\chi_{\mathbf{U}_{i}}+m}\mid\mathcal{D}_{\chi_{u}+\chi_{\mathbf{U}_{i}}}\right)>\frac{1}{2}\dim\mu-\epsilon\right\} >\delta.
\]
\end{prop}

The proof of the proposition relies on Corollary \ref{cor:from equidist prop},
a technical linearization argument, and the following simple lemma.
\begin{lem}
\label{lem:ent > 1/2 dim mu outside a small interval}For every $\epsilon>0$,
$0<\eta<\eta(\epsilon)$, and $m\ge M(\epsilon,\eta)\ge1$ the following
holds. Let $u\in\Lambda^{*}$ be with $\Vert g_{u}\Vert_{\mathrm{op}}\ge\eta^{-1}$
and $L(g_{u})\notin B(e_{1}\mathbb{C},2\eta)$. Then there exists
$z\mathbb{R}\in\mathbb{RP}^{1}$ such that,
\[
\frac{1}{m}H\left(\pi_{w\mathbb{R}}\psi g_{u}\mu_{Y_{u,\eta}},\mathcal{D}_{\chi_{u}+m}\right)>\frac{1}{2}\dim\mu-\epsilon\text{ for all }w\mathbb{R}\in\mathbb{RP}^{1}\setminus B\left(z\mathbb{R},\eta\right).
\]
\end{lem}

\begin{proof}
Let $\epsilon,\eta\in(0,1)$ and $m\in\mathbb{Z}_{>0}$ be with $\epsilon^{-1}\ll\eta^{-1}\ll m$,
and let $u\in\Lambda^{*}$ be such that $\Vert g_{u}\Vert_{\mathrm{op}}\ge\eta^{-1}$
and $L(g_{u})\notin B(e_{1}\mathbb{C},2\eta)$. For each $w\mathbb{R}\in\mathbb{RP}^{1}$
set
\[
H\left(w\mathbb{R}\right):=\frac{1}{m}H\left(\pi_{w\mathbb{R}}\psi g_{u}\mu_{Y_{u,\eta}},\mathcal{D}_{\chi_{u}+m}\right),
\]
and let $z\mathbb{R}\in\mathbb{RP}^{1}$ be such that
\[
H\left(z\mathbb{R}\right)\le\underset{w\mathbb{R}\in\mathbb{RP}^{1}}{\inf}H\left(w\mathbb{R}\right)+\epsilon.
\]
From Lemma \ref{lem:lb on ent wrt mod dyad part}, by the conditional
entropy formula, and by the last inequality, it follows that for each
$w\mathbb{R}\in\mathbb{RP}^{1}\setminus B\left(z\mathbb{R},\eta\right)$
\[
\dim\mu-\epsilon\le H\left(z\mathbb{R}\right)+H\left(w\mathbb{R}\right)\le2H\left(w\mathbb{R}\right)+\epsilon,
\]
which proves the lemma.
\end{proof}
The linearization argument mentioned above is contained in the proof
of the following lemma.
\begin{lem}
\label{lem:lemma with lots of conds}For every $\epsilon>0$, $0<\eta<\eta(\epsilon)$,
$m\ge M(\epsilon,\eta)\ge1$, and $k\ge K(\epsilon,\eta,m)\ge1$ the
following holds. Let $u\in\Lambda^{*}$ be with $\Vert g_{u}\Vert_{\mathrm{op}}\ge\eta^{-1}$
and $L(g_{u})\notin B\left(e_{1}\mathbb{C},2\eta\right)$. Additionally,
let $i\in\mathbb{Z}_{>0}$ and $\omega\in\Lambda^{\mathbb{N}}$ be
such that
\begin{equation}
\begin{aligned}\Vert g_{\omega|_{i}}\Vert_{\mathrm{op}}\ge\eta^{-1}, &  & \Vert g_{(\sigma^{i}\omega)|_{k}}\Vert_{\mathrm{op}}\ge2^{k\chi/2},\\
L(g_{\omega|_{i}})\in Y_{u,2\eta}\setminus B\left(e_{1}\mathbb{C},2\eta\right), &  & L\left(\sigma^{i}\omega\right)\in Y_{\omega|_{i},2\eta}\setminus B\left(e_{1}\mathbb{C},2\eta\right),\\
L\left(\sigma^{i+k}\omega\right)\in Y_{(\sigma^{i}\omega)|_{k},\eta}, &  & \alpha_{i}(\omega)=\varphi_{\omega|_{i}}'\left(\psi L\left(\sigma^{i}\omega\right)\right)\mathbb{R},\\
g_{\omega|_{i}}L\left(\sigma^{i}\omega\right)=L(\omega), &  & g_{(\sigma^{i}\omega)|_{k}}L\left(\sigma^{i+k}\omega\right)=L\left(\sigma^{i}\omega\right).
\end{aligned}
\label{eq:lots of conds}
\end{equation}
Then for each $z\mathbb{R}\in\mathbb{RP}^{1}$,
\begin{multline*}
\frac{1}{m}H\left(\pi_{z\mathbb{R}}\varphi_{u\omega|_{i+k}}\nu,\mathcal{D}_{\chi_{u}+\chi_{\omega|_{i+k}}+m}\mid\mathcal{D}_{\chi_{u}+\chi_{\omega|_{i+k}}}\right)+\epsilon\\
>\frac{1}{m}H\left(\pi_{\varphi_{u}'\left(\psi L(\omega)\right)^{-1}z\mathbb{R}\alpha_{i}(\omega)^{-1}}\psi g_{(\sigma^{i}\omega)|_{k}}\mu_{Y_{(\sigma^{i}\omega)|_{k},\eta}},\mathcal{D}_{\chi_{(\sigma^{i}\omega)|_{k}}+m}\right).
\end{multline*}
\end{lem}

The proof of the lemma requires the following first-order Taylor remainder
estimate, which follows directly from \cite[p. 126]{MR510197}.
\begin{lem}
\label{lem:first ord Taylor}Let $\Omega$ be an open subset of $\mathbb{C}$,
let $f:\Omega\rightarrow\mathbb{C}$ be holomorphic, and let $z_{0}\in\Omega$
and $r>0$ be such that $B(z_{0},2r)\subset\Omega$. Then, setting
$M:=\max\left\{ |f(z)|\::\:z\in\partial B(z_{0},2r)\right\} $,
\[
\left|f(z)-f(z_{0})-f'(z_{0})(z-z_{0})\right|\le\frac{1}{2}Mr^{-2}|z-z_{0}|^{2}\text{ for all }z\in B(z_{0},r).
\]
\end{lem}

\begin{proof}[Proof of Lemma \ref{lem:lemma with lots of conds}]
Let $0<\epsilon,\eta<1$ and $m,k\in\mathbb{Z}_{>0}$ be with $\epsilon^{-1}\ll\eta^{-1}\ll m\ll k$,
let $u\in\Lambda^{*}$ be with $\Vert g_{u}\Vert_{\mathrm{op}}\ge\eta^{-1}$
and $L(g_{u})\notin B\left(e_{1}\mathbb{C},2\eta\right)$, let $i\in\mathbb{Z}_{>0}$
and $\omega\in\Lambda^{\mathbb{N}}$ be such that the conditions in
(\ref{eq:lots of conds}) are all satisfied, and fix $z\mathbb{R}\in\mathbb{RP}^{1}$.
Set
\[
H:=\frac{1}{m}H\left(\pi_{z\mathbb{R}}\varphi_{u\omega|_{i+k}}\nu,\mathcal{D}_{\chi_{u}+\chi_{\omega|_{i+k}}+m}\mid\mathcal{D}_{\chi_{u}+\chi_{\omega|_{i+k}}}\right),
\]
and write $v_{1}:=\omega|_{i}$ and $v_{2}:=(\sigma^{i}\omega)|_{k}$.
Since $\epsilon^{-1}\ll\eta^{-1}$, we may assume that $\mu\left(Y_{v_{2},\eta}\right)>1-\epsilon$.

From $\nu=\psi\mu$ and by (\ref{eq:psi is G-equi}),
\[
H=\frac{1}{m}H\left(\pi_{z\mathbb{R}}\varphi_{uv_{1}}\psi g_{v_{2}}\mu,\mathcal{D}_{\chi_{u}+\chi_{v_{1}v_{2}}+m}\mid\mathcal{D}_{\chi_{u}+\chi_{v_{1}v_{2}}}\right).
\]
Hence, by concavity, since $\mu\left(Y_{v_{2},\eta}\right)>1-\epsilon$,
and from (\ref{eq:norm cond ent on C <=00003D 2}),
\begin{equation}
H\ge\frac{1}{m}H\left(\pi_{z\mathbb{R}}\varphi_{uv_{1}}\psi g_{v_{2}}\mu_{Y_{v_{2},\eta}},\mathcal{D}_{\chi_{u}+\chi_{v_{1}v_{2}}+m}\mid\mathcal{D}_{\chi_{u}+\chi_{v_{1}v_{2}}}\right)-2\epsilon.\label{eq:first lb for H(omega|_=00007Bi+k=00007D)}
\end{equation}

By Lemma \ref{lem:dist of gzC to L(g)}, from $L\left(\sigma^{i+k}\omega\right)\in Y_{v_{2},\eta}$,
and since $\Vert g_{v_{2}}\Vert_{\mathrm{op}}\ge2^{k\chi/2}$ and
$\eta^{-1}\ll k$, it follows that for each $w\mathbb{C}\in Y_{v_{2},\eta}$
\[
d\left(L\left(\sigma^{i}\omega\right),g_{v_{2}}w\mathbb{C}\right)=d\left(g_{v_{2}}L\left(\sigma^{i+k}\omega\right),g_{v_{2}}w\mathbb{C}\right)\le\eta^{-2}\Vert g_{v_{2}}\Vert_{\mathrm{op}}^{-2}\le\eta.
\]
Thus, since $L\left(\sigma^{i}\omega\right)\in Y_{v_{1},2\eta}\setminus B\left(e_{1}\mathbb{C},2\eta\right)$,
\begin{equation}
g_{v_{2}}\left(Y_{v_{2},\eta}\right)\subset B\left(L\left(\sigma^{i}\omega\right),\eta^{-2}\Vert g_{v_{2}}\Vert_{\mathrm{op}}^{-2}\right)\subset Y_{v_{1},\eta}\setminus B\left(e_{1}\mathbb{C},\eta\right).\label{eq:first containment}
\end{equation}

By Lemma \ref{lem:dist of gzC to L(g)} and since $L(g_{v_{1}})\in Y_{u,2\eta}\setminus B\left(e_{1}\mathbb{C},2\eta\right)$
and $\Vert g_{v_{1}}\Vert_{\mathrm{op}}\ge\eta^{-1}$,
\begin{equation}
g_{v_{1}}\left(Y_{v_{1},\eta}\right)\subset B\left(L\left(g_{v_{1}}\right),\eta^{-1}\Vert g_{v_{1}}\Vert_{\mathrm{op}}^{-2}\right)\subset Y_{u,\eta}\setminus B\left(e_{1}\mathbb{C},\eta\right).\label{eq:second containment}
\end{equation}
Similarly, by Lemma \ref{lem:dist of gzC to L(g)} and since $L(g_{u})\notin B\left(e_{1}\mathbb{C},2\eta\right)$
and $\Vert g_{u}\Vert_{\mathrm{op}}\ge\eta^{-1}$,
\begin{equation}
g_{u}\left(Y_{u,\eta}\right)\subset B\left(L\left(g_{u}\right),\eta^{-1}\Vert g_{u}\Vert_{\mathrm{op}}^{-2}\right)\subset\mathbb{CP}^{1}\setminus B\left(e_{1}\mathbb{C},\eta\right).\label{eq:third containment}
\end{equation}

For $j=1,2$ set
\[
\Omega_{j}:=\psi\left(Y_{v_{1},j\eta}\setminus B\left(e_{1}\mathbb{C},j\eta\right)\right),
\]
and let $0<\rho<1$ be such that $B(w,2\rho)\subset\Omega_{1}$ for
all $w\in\Omega_{2}$. Since $\eta^{-1}\ll k$, we may assume that
$\rho^{-1}\ll k$. Setting $w_{0}:=\psi\left(L\left(\sigma^{i}\omega\right)\right)$,
we have $w_{0}\in\Omega_{2}$. Moreover, from (\ref{eq:first containment}),
by Lemma \ref{lem:bi-lip prop of psi}, from $\Vert g_{v_{2}}\Vert_{\mathrm{op}}\ge2^{k\chi/2}$,
and since $\eta^{-1},\rho^{-1}\ll k$,
\begin{equation}
\mathrm{supp}\left(\psi g_{v_{2}}\mu_{Y_{v_{2},\eta}}\right)\subset B\left(w_{0},2\eta^{-4}\Vert g_{v_{2}}\Vert_{\mathrm{op}}^{-2}\right)\subset B\left(w_{0},\rho\right).\label{eq:supp cont in B(w_0,rho)}
\end{equation}
Additionally, from (\ref{eq:psi is G-equi}), since $w_{0}\in\Omega_{2}$,
from (\ref{eq:second containment}) and (\ref{eq:third containment}),
and by Lemmas \ref{lem:bi-lip prop of psi}, \ref{lem:zC to gzC is bi-Lip with const norm^2}
and \ref{lem:dist of gzC to L(g)}, it follows that for each $w\in\Omega_{1}\setminus\{w_{0}\}$
\begin{equation}
\frac{1}{2}\eta^{2}\Vert g_{u}\Vert_{\mathrm{op}}^{-2}\Vert g_{v_{1}}\Vert_{\mathrm{op}}^{-2}\le\frac{\left|\varphi_{uv_{1}}(w)-\varphi_{uv_{1}}(w_{0})\right|}{|w-w_{0}|}\le2\eta^{-6}\Vert g_{u}\Vert_{\mathrm{op}}^{-2}\Vert g_{v_{1}}\Vert_{\mathrm{op}}^{-2}.\label{eq:lip prop of varph_=00007Buv_1=00007D}
\end{equation}

By (\ref{eq:supp cont in B(w_0,rho)}) and (\ref{eq:lip prop of varph_=00007Buv_1=00007D}),
and since $B(w_{0},2\rho)\subset\Omega_{1}$,
\[
\mathrm{diam}\left(\mathrm{supp}\left(\pi_{z\mathbb{R}}\varphi_{uv_{1}}\psi g_{v_{2}}\mu_{Y_{v_{2},\eta}}\right)\right)=O_{\eta}\left(\Vert g_{u}\Vert_{\mathrm{op}}^{-2}\Vert g_{v_{1}v_{2}}\Vert_{\mathrm{op}}^{-2}\right).
\]
Thus, from (\ref{eq:first lb for H(omega|_=00007Bi+k=00007D)}) and
$\epsilon^{-1},\eta^{-1}\ll m$,
\begin{equation}
H\ge\frac{1}{m}H\left(\pi_{z\mathbb{R}}\varphi_{uv_{1}}\psi g_{v_{2}}\mu_{Y_{v_{2},\eta}},\mathcal{D}_{\chi_{u}+\chi_{v_{1}v_{2}}+m}\right)-3\epsilon.\label{eq:second lb for H(omega|_=00007Bi+k=00007D)}
\end{equation}

By (\ref{eq:lip prop of varph_=00007Buv_1=00007D}),
\[
\left|\varphi_{uv_{1}}(w)-\varphi_{uv_{1}}(w_{0})\right|\le4\rho\eta^{-6}\Vert g_{u}\Vert_{\mathrm{op}}^{-2}\Vert g_{v_{1}}\Vert_{\mathrm{op}}^{-2}\;\text{ for }w\in\partial B\left(w_{0},2\rho\right).
\]
From this and (\ref{eq:supp cont in B(w_0,rho)}), by applying Lemma
\ref{lem:first ord Taylor} with $f:=\varphi_{uv_{1}}-\varphi_{uv_{1}}(w_{0})$,
since $B(w_{0},2\rho)\subset\Omega_{1}$, from $\Vert g_{v_{2}}\Vert_{\mathrm{op}}\ge2^{k\chi/2}$,
and since $\eta^{-1},\rho^{-1},m\ll k$, it follows that for each
$w\in\mathrm{supp}\left(\psi g_{v_{2}}\mu_{Y_{v_{2},\eta}}\right)$
\begin{eqnarray*}
\left|\varphi_{uv_{1}}(w)-\varphi_{uv_{1}}(w_{0})-\varphi_{uv_{1}}'(w_{0})(w-w_{0})\right| & \le & 2\rho^{-1}\eta^{-6}\Vert g_{u}\Vert_{\mathrm{op}}^{-2}\Vert g_{v_{1}}\Vert_{\mathrm{op}}^{-2}|w-w_{0}|^{2}\\
 & \le & 8\rho^{-1}\eta^{-14}\Vert g_{u}\Vert_{\mathrm{op}}^{-2}\Vert g_{v_{1}}\Vert_{\mathrm{op}}^{-2}\Vert g_{v_{2}}\Vert_{\mathrm{op}}^{-4}\\
 & \le & 2^{-m}\Vert g_{u}\Vert_{\mathrm{op}}^{-2}\Vert g_{v_{1}v_{2}}\Vert_{\mathrm{op}}^{-2}.
\end{eqnarray*}
Hence, from (\ref{eq:second lb for H(omega|_=00007Bi+k=00007D)})
and by Lemma \ref{lem:ent of push by close func},
\begin{equation}
H\ge\frac{1}{m}H\left(\pi_{z\mathbb{R}}S_{\varphi_{uv_{1}}'(w_{0})}\psi g_{v_{2}}\mu_{Y_{v_{2},\eta}},\mathcal{D}_{\chi_{u}+\chi_{v_{1}v_{2}}+m}\right)-4\epsilon.\label{eq:third lb for H(omega|_=00007Bi+k=00007D)}
\end{equation}

We have,
\[
\varphi_{v_{1}}'(w_{0})\mathbb{R}=\varphi_{\omega|_{i}}'\left(\psi L\left(\sigma^{i}\omega\right)\right)\mathbb{R}=\alpha_{i}(\omega).
\]
Additionally, from (\ref{eq:psi is G-equi}) and $g_{\omega|_{i}}L\left(\sigma^{i}\omega\right)=L(\omega)$,
\[
\varphi_{v_{1}}(w_{0})=\varphi_{\omega|_{i}}\psi L\left(\sigma^{i}\omega\right)=\psi L(\omega).
\]
Hence,
\begin{eqnarray*}
\pi_{z\mathbb{R}}\circ S_{\varphi_{uv_{1}}'(w_{0})} & = & S_{\varphi_{uv_{1}}'(w_{0})}\circ\pi_{\varphi_{uv_{1}}'(w_{0})^{-1}z\mathbb{R}}\\
 & = & S_{\varphi_{uv_{1}}'(w_{0})}\circ\pi_{\varphi_{u}'\left(\psi L(\omega)\right)^{-1}z\mathbb{R}\alpha_{i}(\omega)^{-1}}.
\end{eqnarray*}
Together with (\ref{eq:third lb for H(omega|_=00007Bi+k=00007D)}),
this gives
\begin{equation}
H\ge\frac{1}{m}H\left(S_{\varphi_{uv_{1}}'(w_{0})}\pi_{\varphi_{u}'\left(\psi L(\omega)\right)^{-1}z\mathbb{R}\alpha_{i}(\omega)^{-1}}\psi g_{v_{2}}\mu_{Y_{v_{2},\eta}},\mathcal{D}_{\chi_{u}+\chi_{v_{1}v_{2}}+m}\right)-4\epsilon.\label{eq:fourth lb for H(omega|_=00007Bi+k=00007D)}
\end{equation}

By (\ref{eq:lb for norm of g_u z}) and (\ref{eq:first containment}),
\[
\Vert g_{v_{1}v_{2}}\Vert_{\mathrm{op}}\ge\eta^{2}\Vert g_{v_{1}}\Vert_{\mathrm{op}}\Vert g_{v_{2}}\Vert_{\mathrm{op}}.
\]
Moreover, from (\ref{eq:lip prop of varph_=00007Buv_1=00007D}),
\[
\left|\varphi_{uv_{1}}'(w_{0})\right|\ge\frac{1}{2}\eta^{2}\Vert g_{u}\Vert_{\mathrm{op}}^{-2}\Vert g_{v_{1}}\Vert_{\mathrm{op}}^{-2}.
\]
Thus, from (\ref{eq:fourth lb for H(omega|_=00007Bi+k=00007D)}) and
since $\epsilon^{-1},\eta^{-1}\ll m$,
\[
H\ge\frac{1}{m}H\left(\pi_{\varphi_{u}'\left(\psi L(\omega)\right)^{-1}z\mathbb{R}\alpha_{i}(\omega)^{-1}}\psi g_{v_{2}}\mu_{Y_{v_{2},\eta}},\mathcal{D}_{\chi_{v_{2}}+m}\right)-5\epsilon,
\]
which completes the proof of the lemma.
\end{proof}
\begin{proof}[Proof of Proposition \ref{prop:nontriv portion > dim(mu)/2}]
Let $0<\delta<1$ be as obtained in Corollary \ref{cor:from equidist prop},
let $0<\epsilon,\eta<1$ and $m,k,n\in\mathbb{Z}_{>0}$ be with $\delta^{-1}\ll\epsilon^{-1}\ll\eta^{-1}\ll m\ll k\ll n$,
fix $z\mathbb{R}\in\mathbb{RP}^{1}$, and fix $u\in\Lambda^{*}$ with
$\Vert g_{u}\Vert_{\mathrm{op}}\ge\eta^{-1}$ and $L(g_{u})\notin B\left(e_{1}\mathbb{C},2\eta\right)$.
Set
\[
H(v):=\frac{1}{m}H\left(\pi_{z\mathbb{R}}\varphi_{uv}\nu,\mathcal{D}_{\chi_{u}+\chi_{v}+m}\mid\mathcal{D}_{\chi_{u}+\chi_{v}}\right)\text{ for }v\in\Lambda^{*},
\]
and let
\[
P:=\mathbb{P}_{1\le i\le n}\left\{ H\left(\mathbf{U}_{i}\right)>\frac{1}{2}\dim\mu-\epsilon\right\} .
\]

Recalling that $\lambda_{n}$ denotes the uniform probability measure
on $\mathcal{N}_{n}:=\left\{ 1,...,n\right\} $,
\begin{eqnarray*}
P & = & \int\beta\left\{ \omega:H\left(\omega|_{i}\right)>\frac{1}{2}\dim\mu-\epsilon\right\} d\lambda_{n}(i)\\
 & = & \int\lambda_{n}\left\{ i\in\mathcal{N}_{n}:H\left(\omega|_{i}\right)>\frac{1}{2}\dim\mu-\epsilon\right\} d\beta(\omega).
\end{eqnarray*}
Hence, from $\epsilon^{-1},k\ll n$,
\begin{equation}
P\ge\int\lambda_{n}\left\{ i\in\mathcal{N}_{n}\::\:H\left(\omega|_{i+k}\right)>\frac{1}{2}\dim\mu-\epsilon\right\} \:d\beta(\omega)-\epsilon.\label{eq:first lb on P}
\end{equation}

Let $F$ denote the set of all $(i,\omega)\in\mathcal{N}_{n}\times\Lambda^{\mathbb{N}}$
such that $\Vert g_{(\sigma^{i}\omega)|_{k}}\Vert_{\mathrm{op}}\ge\eta^{-1}$,
$L\left(g_{(\sigma^{i}\omega)|_{k}}\right)\notin B\left(e_{1}\mathbb{C},2\eta\right)$,
and the conditions in (\ref{eq:lots of conds}) are all satisfied.
Note that $\beta$ is $\sigma$-invariant, and that for each $i\ge1$,
the maps $\omega\mapsto\omega|_{i}$ and $\omega\mapsto\sigma^{i}\omega$
are $\beta$-independent. Hence, by the results of Section \ref{subsec:results-from-rand-prod-of-mat},
from (\ref{eq:chi =00003D lim a.s.}) and (\ref{eq:alpha_n a.s. equals}),
and since $\epsilon^{-1}\ll\eta^{-1}\ll k,n$, we may assume that
$\lambda_{n}\times\beta(F)>1-\epsilon$.

By Lemma \ref{lem:lemma with lots of conds}, for each $(i,\omega)\in F$
we have
\begin{equation}
H\left(\omega|_{i+k}\right)\ge\frac{1}{m}H\left(\pi_{\varphi_{u}'\left(\psi L(\omega)\right)^{-1}z\mathbb{R}\alpha_{i}(\omega)^{-1}}\psi g_{(\sigma^{i}\omega)|_{k}}\mu_{Y_{(\sigma^{i}\omega)|_{k},\eta}},\mathcal{D}_{\chi_{(\sigma^{i}\omega)|_{k}}+m}\right)-\epsilon/2.\label{eq:by lem with lots of conds}
\end{equation}
Let $\mathcal{V}$ denote the set of all $v\in\Lambda^{k}$ such that
$\Vert g_{v}\Vert_{\mathrm{op}}\ge\eta^{-1}$ and $L\left(g_{v}\right)\notin B\left(e_{1}\mathbb{C},2\eta\right)$.
By Lemma \ref{lem:ent > 1/2 dim mu outside a small interval}, for
each $v\in\mathcal{V}$ there exists $w_{v}\mathbb{R}\in\mathbb{RP}^{1}$
such that
\begin{equation}
\frac{1}{m}H\left(\pi_{w\mathbb{R}}\psi g_{v}\mu_{Y_{v,\eta}},\mathcal{D}_{\chi_{v}+m}\right)>\frac{1}{2}\dim\mu-\epsilon/2\text{ for all }w\mathbb{R}\in\mathbb{RP}^{1}\setminus B\left(w_{v}\mathbb{R},\eta\right).\label{eq:>Delta-=00005Cepsilon/2 ouside of interval}
\end{equation}

Let $h:\Lambda^{\mathbb{N}}\rightarrow\mathbb{RP}^{1}$ be defined
by
\[
h(\omega)=\begin{cases}
w_{\omega|_{k}}\mathbb{R} & \text{ if }\omega|_{k}\in\mathcal{V}\\
\mathbb{R} & \text{ otherwise}
\end{cases}\text{ for }\omega\in\Lambda^{\mathbb{N}}.
\]
Note that, since $\epsilon^{-1},\eta^{-1},m,k\ll n$, we may assume
that $n$ is large with respect to $h$. From (\ref{eq:by lem with lots of conds})
and (\ref{eq:>Delta-=00005Cepsilon/2 ouside of interval}), it follows
that $H\left(\omega|_{i+k}\right)>\frac{1}{2}\dim\mu-\epsilon$ for
all $(i,k)\in F$ with
\[
\varphi_{u}'\left(\psi L(\omega)\right)^{-1}z\mathbb{R}\alpha_{i}(\omega)^{-1}\notin B\left(h(\sigma^{i}\omega),\eta\right).
\]
Hence, by (\ref{eq:first lb on P}) and since $\lambda_{n}\times\beta(F)>1-\epsilon$,
\[
P\ge\int\lambda_{n}\left\{ i\in\mathcal{N}_{n}\::\:d\left(\varphi_{u}'\left(\psi L(\omega)\right)^{-1}z\mathbb{R},\alpha_{i}(\omega)h(\sigma^{i}\omega)\right)>\eta\right\} \:d\beta(\omega)-2\epsilon.
\]
From this, by Corollary \ref{cor:from equidist prop}, since $n$
is large with respect to $h$, and since $\delta^{-1}\ll\epsilon^{-1},\eta^{-1}$,
it follows that $P\ge\delta/2$, which completes the proof of the
proposition.
\end{proof}

\subsection{\label{subsec:A-technical-lemma on random words}A lemma concerning
random words}

Recall the random words $\mathbf{I}(j,l;k)$ from Section \ref{subsec:Symbolic-related-notations}.
We shall need the following lemma in order to obtain the conclusion
of Proposition \ref{prop:nontriv portion > dim(mu)/2} with $\mathbf{I}(j,l;k)$
in place of $\mathbf{U}_{i}$.
\begin{lem}
\label{lem:ac of words}For every $\epsilon>0$ and $l\ge L(\epsilon)\ge1$
there exists $M=M(\epsilon,l)\in\mathbb{Z}_{>0}$ such that for every
$0\le j<l$ and $n\ge N(\epsilon,l)\ge1$, there exists $\mathcal{V}\subset\cup_{1\le k\le n}\Lambda^{j+lk}$
satisfying \noun{
\begin{equation}
\mathbb{P}_{1\le k\le n}\left\{ \mathbf{U}_{j+lk}\in\mathcal{V}\right\} \ge1-\epsilon\label{eq:big mass of V}
\end{equation}
}and
\begin{equation}
\mathbb{E}_{1\le k\le n}\left(\mathbf{1}_{\left\{ \mathbf{U}_{j+lk}\in\mathcal{V}\right\} }\delta_{\mathbf{U}_{j+lk}}\right)\ll\mathbb{E}_{1\le k\le nM}\left(\delta_{\mathbf{I}(j,l;k)}\right),\label{eq:ac of words}
\end{equation}
with Radon--Nikodym derivative bounded by $M$.
\end{lem}

\begin{proof}
Let $0<\epsilon,\delta<1$ and $l,j,n\in\mathbb{Z}_{\ge0}$ be such
that $\epsilon^{-1}\ll\delta^{-1}\ll l\ll n$ and $0\le j<l$. Given
$1\le k\le n$, let $\mathcal{V}_{k}$ denote the set of words $u_{0}...u_{k}=v\in\Lambda^{j+lk}$
such that $u_{0}\in\Lambda^{j}$, $u_{i}\in\Lambda^{l}$ for $1\le i\le k$,
and $\Vert g_{v}\Vert_{\mathrm{op}}^{2}>2\Vert g_{u_{0}...u_{i}}\Vert_{\mathrm{op}}^{2}$
for all $0\le i<k$. Set
\[
\mathcal{V}:=\cup_{k=1}^{n}\mathcal{V}_{k},\:R:=\max_{i\in\Lambda}\Vert g_{i}\Vert_{\mathrm{op}}^{2},\text{ and }M:=\left\lceil 2l\log R\right\rceil .
\]
Given $1\le k\le n$ and $v\in\Lambda^{j+lk}$, we have $\Vert g_{v}\Vert_{\mathrm{op}}^{2}<2^{nM}$.
This clearly implies that $\mathcal{V}\subset\cup_{k=1}^{nM}\Psi\left(j,l;k\right)$,
which gives (\ref{eq:ac of words}) with Radon--Nikodym derivative
bounded by $M$. Thus, in order to complete the proof of the lemma
it remains to establish (\ref{eq:big mass of V}).

Let $\mathcal{U}$ denote the set of words $u\in\Lambda^{l}$ with
$\Vert g_{u}\Vert_{\mathrm{op}}^{2}<2^{l(2\chi+\delta)}$. By (\ref{eq:lim on normalized chi_omega|_n})
and since $\delta^{-1}\ll l$, we may assume that $\beta\left(\left[\mathcal{U}\right]\right)>1-\delta/2$,
where recall that $\left[\mathcal{U}\right]:=\cup_{u\in\mathcal{U}}[u]$.
Let $\mathcal{W}$ denote the set of words $u_{0}...u_{n}=w\in\Lambda^{j+ln}$
such that $u_{0}\in\Lambda^{j}$, $u_{k}\in\Lambda^{l}$ for $1\le k\le n$,
$\Vert g_{w}\Vert_{\mathrm{op}}^{2}>2^{ln(2\chi-\delta)}$, and 
\[
\frac{1}{n}\#\left\{ 1\le k\le n\::\:u_{k}\in\mathcal{U}\right\} >1-\delta.
\]
By (\ref{eq:lim on normalized chi_omega|_n}), by the ergodicity of
$\left(\Lambda^{\mathbb{N}},\sigma^{l},\beta\right)$, from $\beta\left(\left[\mathcal{U}\right]\right)>1-\delta/2$,
and since $\delta^{-1},l\ll n$, we may assume that $\beta\left(\left[\mathcal{W}\right]\right)>1-\delta$.

Given $u_{0}...u_{n}=w\in\mathcal{W}$, let $K_{w}$ denote the set
of integers $1\le k\le n$ such that $u_{0}...u_{k}\notin\mathcal{V}$.
Let us show that $|K_{w}|\le\epsilon n/2$. Suppose that $K_{w}\ne\emptyset$,
set $m:=|K_{w}|$, and let $1\le k_{1}<...<k_{m}\le n$ be an enumeration
of $K_{w}$. Note that for each $1\le a\le m$ there exists $0\le i_{a}<k_{a}$
such that
\begin{equation}
\Vert g_{u_{0}...u_{k_{a}}}\Vert_{\mathrm{op}}^{2}\le2\Vert g_{u_{0}...u_{i_{a}}}\Vert_{\mathrm{op}}^{2}.\label{eq:prop of k_a}
\end{equation}

Let us construct by induction strictly decreasing sequences $\{b_{q}\}_{q=1}^{s}\subset\{i_{a}\}_{a=1}^{m}$
and $\{c_{q}\}_{q=1}^{s}\subset\{k_{a}\}_{a=1}^{m}$ as follows. Set
$b_{1}:=i_{m}$ and $c_{1}:=k_{m}$. Let $q\ge1$ and suppose that
$\{b_{t}\}_{t=1}^{q}$ and $\{c_{t}\}_{t=1}^{q}$ have already been
chosen. If $b_{q}<k_{1}$, then set $s:=q$ and terminate the construction.
Otherwise, if $b_{q}\ge k_{1}$, let $1\le a<m$ be such that $b_{q}\ge k_{a}$
and $b_{q}<k_{a+1}$, and set $b_{q+1}:=i_{a}$ and $c_{q+1}:=k_{a}$.
This completes the inductive construction.

Note that the intervals $(b_{1},c_{1}],...,(b_{s},c_{s}]$ are disjoint.
Using this and (\ref{eq:prop of k_a}), it is easy to show by induction
that for each $1\le q\le s$,
\begin{equation}
\Vert g_{w}\Vert_{\mathrm{op}}^{2}\le2^{q}\Vert g_{u_{0}...u_{b_{q}}}\Vert_{\mathrm{op}}^{2}\left(\prod_{t=2}^{q}\Vert g_{u_{c_{t}+1}...u_{b_{t-1}}}\Vert_{\mathrm{op}}^{2}\right)\Vert g_{u_{c_{1}+1}...u_{n}}\Vert_{\mathrm{op}}^{2}.\label{eq:ineq by ind on q}
\end{equation}
Let $J_{1}$ denote the set of $1\le k\le n$ such that $k\notin\cup_{q=1}^{s}(b_{q},c_{q}]$
and $u_{k}\in\mathcal{U}$, and let $J_{2}$ denote the set of $0\le k\le n$
such that $u_{k}\notin\mathcal{U}$. By applying (\ref{eq:ineq by ind on q})
with $q=s$,
\begin{equation}
\Vert g_{w}\Vert_{\mathrm{op}}^{2}\le2^{s}\left(\prod_{k\in J_{1}}\Vert g_{u_{k}}\Vert_{\mathrm{op}}^{2}\right)\left(\prod_{k\in J_{2}}\Vert g_{u_{k}}\Vert_{\mathrm{op}}^{2}\right).\label{eq:prod over J_1 times prod over J_2}
\end{equation}

By the construction of the sequences $\{b_{q}\}_{q=1}^{s}$ and $\{c_{q}\}_{q=1}^{s}$,
it follows that $K_{w}\subset\cup_{q=1}^{s}(b_{q},c_{q}]$, which
implies $|J_{1}|\le n-m$. Moreover, by the definitions of $J_{1}$
and $\mathcal{U}$, we have $\Vert g_{u_{k}}\Vert_{\mathrm{op}}^{2}\le2^{l(2\chi+\delta)}$
for each $k\in J_{1}$. Hence,
\[
\prod_{k\in J_{1}}\Vert g_{u_{k}}\Vert_{\mathrm{op}}^{2}\le2^{l(2\chi+\delta)(n-m)}.
\]
From $w\in\mathcal{W}$, we get $|J_{2}|<\delta n+1$. Additionally,
note that $\Vert g_{u_{k}}\Vert_{\mathrm{op}}^{2}\le R^{l}$ for $k\in J_{2}$.
Thus,
\[
\prod_{k\in J_{2}}\Vert g_{u_{k}}\Vert_{\mathrm{op}}^{2}\le R^{2l\delta n}.
\]
Since $w\in\mathcal{W}$, we also have $\Vert g_{w}\Vert_{\mathrm{op}}^{2}>2^{ln(2\chi-\delta)}$.
By combining these inequalities together with (\ref{eq:prod over J_1 times prod over J_2}),
and then taking the logarithm of both sides,
\[
ln(2\chi-\delta)<s+l(2\chi+\delta)(n-m)+2l\delta n\log R.
\]
Together with $s\le m$, this gives
\[
(2\chi+\delta)m<4\delta\left(1+\log R\right)n.
\]
From $\epsilon^{-1}\ll\delta^{-1}$, and since $\chi$ and $R$ are
positive global constants, we obtain $|K_{w}|=m\le\epsilon n/2$ for
$w\in\mathcal{W}$, as desired.

Now we can establish (\ref{eq:big mass of V}). Indeed,
\begin{eqnarray*}
\mathbb{P}_{1\le k\le n}\left\{ \mathbf{U}_{j+lk}\in\mathcal{V}\right\}  & = & \frac{1}{n}\sum_{k=1}^{n}\int\mathbf{1}_{\left\{ \omega|_{j+lk}\in\mathcal{V}_{k}\right\} }d\beta(\omega)\\
 & \ge & \int\mathbf{1}_{\left\{ \omega|_{j+ln}\in\mathcal{W}\right\} }\frac{1}{n}\sum_{k=1}^{n}\mathbf{1}_{\left\{ \omega|_{j+lk}\in\mathcal{V}_{k}\right\} }d\beta(\omega)\\
 & = & \int\mathbf{1}_{\left\{ \omega|_{j+ln}\in\mathcal{W}\right\} }\frac{1}{n}\left(n-\left|K_{\omega|_{j+ln}}\right|\right)d\beta(\omega)\\
 & \ge & \beta\left(\left[\mathcal{W}\right]\right)(1-\epsilon/2).
\end{eqnarray*}
Since $\beta\left(\left[\mathcal{W}\right]\right)>1-\delta$, this
gives (\ref{eq:big mass of V}), which completes the proof of the
lemma.
\end{proof}

\subsection{\label{subsec:Proof-of-Propositions}Proof of Propositions \ref{prop:lb on ent of proj of comp of nu}
and \ref{prop:lb on ent of proj of cylinders of nu}}

First we prove Proposition \ref{prop:lb on ent of proj of cylinders of nu},
which is the following statement.
\begin{prop*}
Suppose that $\dim\mu<2$. Then there exist $\gamma,\eta_{0}\in(0,1)$
such that for every $0<\eta<\eta_{0}$, $n\ge N(\eta)\ge1$, $z\mathbb{R}\in\mathbb{RP}^{1}$,
and $u\in\Lambda^{*}$ with $\Vert g_{u}\Vert_{\mathrm{op}}\ge\eta^{-1}$
and $L(g_{u})\notin B\left(e_{1}\mathbb{C},2\eta\right)$,
\[
\frac{1}{n}H\left(\pi_{z\mathbb{R}}\varphi_{u}\nu,\mathcal{D}_{\chi_{u}+n}\mid\mathcal{D}_{\chi_{u}}\right)\ge\dim\mu-1+\gamma.
\]
\end{prop*}
\begin{proof}
Let $0<\delta<1$ be as obtained in Proposition \ref{prop:nontriv portion > dim(mu)/2},
and let $\epsilon,\eta\in(0,1)$ and $l,m,n\in\mathbb{Z}_{>0}$ be
with $\delta^{-1}\ll l\ll\epsilon^{-1}\ll\eta^{-1}\ll m\ll n$. Let
$M=M(\delta/4,l)\in\mathbb{Z}_{>0}$ be as obtained in Lemma \ref{lem:ac of words}.
Since $\delta^{-1},l\ll\epsilon^{-1}$, we may assume that $M\ll\epsilon^{-1}$.
Fix $z\mathbb{R}\in\mathbb{RP}^{1}$ and $u\in\Lambda^{*}$ with $\Vert g_{u}\Vert_{\mathrm{op}}\ge\eta^{-1}$
and $L(g_{u})\notin B\left(e_{1}\mathbb{C},2\eta\right)$, and set
\[
H:=\frac{1}{n}H\left(\pi_{z\mathbb{R}}\varphi_{u}\nu,\mathcal{D}_{\chi_{u}+n}\mid\mathcal{D}_{\chi_{u}}\right).
\]

Set $n':=\left\lfloor n/M\right\rfloor $, and let $\mathcal{U}_{1}$
denote the set of $v\in\Lambda^{*}$ such that
\[
\frac{1}{m}H\left(\pi_{z\mathbb{R}}\varphi_{uv}\nu,\mathcal{D}_{\chi_{u}+\chi_{v}+m}\mid\mathcal{D}_{\chi_{u}+\chi_{v}}\right)>\frac{1}{2}\dim\mu-\epsilon.
\]
By Proposition \ref{prop:nontriv portion > dim(mu)/2} and since $\delta^{-1},\epsilon^{-1},l,M\ll n$,
\[
\mathbb{P}_{l\le i\le n'l+l-1}\left\{ \mathbf{U}_{i}\in\mathcal{U}_{1}\right\} >\delta/2.
\]
Hence, there exists $0\le j<l$ such that
\begin{equation}
\mathbb{P}_{1\le i\le n'}\left\{ \mathbf{U}_{j+li}\in\mathcal{U}_{1}\right\} >\delta/2.\label{eq:P=00007BU_=00007Bj+li=00007D in U_1=00007D>delta/2}
\end{equation}
Given $\mathcal{U}\subset\Lambda^{*}$, set
\[
\Gamma\left(\mathcal{U}\right):=\mathbb{P}_{1\le i\le n}\left\{ \mathbf{I}\left(j,l;i\right)\in\mathcal{U}\right\} .
\]

By Lemma \ref{lem:ac of words}, there exists $\mathcal{V}\subset\cup_{1\le i\le n'}\Lambda^{j+li}$
such that\noun{
\[
\mathbb{P}_{1\le i\le n'}\left\{ \mathbf{U}_{j+li}\in\mathcal{V}\right\} \ge1-\delta/4
\]
}and
\[
\mathbb{E}_{1\le i\le n'}\left(\mathbf{1}_{\left\{ \mathbf{U}_{j+li}\in\mathcal{V}\right\} }\delta_{\mathbf{U}_{j+li}}\right)\ll\mathbb{E}_{1\le i\le n'M}\left(\delta_{\mathbf{I}(j,l;i)}\right),
\]
with Radon--Nikodym derivative bounded by $M$. From this, by (\ref{eq:P=00007BU_=00007Bj+li=00007D in U_1=00007D>delta/2}),
and since $M,\delta^{-1}\ll n$, we obtain $\Gamma\left(\mathcal{U}_{1}\right)>\frac{\delta}{8M}$.

Let $\mathcal{U}_{2}$ denote the set of all $v\in\Lambda^{*}$ such
that $\Vert g_{v}\Vert_{\mathrm{op}}\ge3\eta^{-2}$ and $L(g_{v})\in Y_{u,2\eta}$.
Since $\epsilon^{-1}\ll\eta^{-1}\ll n$, and by Lemma \ref{lem:small ball --> small mass},
we have $\Gamma\left(\mathcal{U}_{2}\right)>1-\epsilon$. Additionally,
by Lemma \ref{lem:triv lb},
\begin{equation}
\frac{1}{m}H\left(\pi_{z\mathbb{R}}\varphi_{uv}\nu,\mathcal{D}_{\chi_{u}+\chi_{v}+m}\mid\mathcal{D}_{\chi_{u}+\chi_{v}}\right)\ge\dim\mu-1-\epsilon\text{ for }v\in\mathcal{U}_{2}.\label{eq:lb on ent for v in U_1}
\end{equation}

Since $\epsilon^{-1},m\ll n$, and by applying Lemma \ref{lem:glob ent to loc ent}
to the measures $\left(\pi_{z\mathbb{R}}\varphi_{u}\nu\right)_{D}$
with $D\in\mathcal{D}_{\chi_{u}}$,
\[
H\ge\mathbb{E}_{1\le i\le n}\left(\frac{1}{m}H\left(\pi_{z\mathbb{R}}\varphi_{u}\nu,\mathcal{D}_{\chi_{u}+i+m}\mid\mathcal{D}_{\chi_{u}+i}\right)\right)-\epsilon.
\]
By (\ref{eq:mu as conv comb with u in Psi_n}), we have $\varphi_{u}\nu=\mathbb{E}\left(\varphi_{u\mathbf{I}\left(j,l;i\right)}\nu\right)$
for each $i\ge1$. Hence, from the last formula, by the concavity
of conditional entropy, from (\ref{eq:op norm is comp to 2^(n/2) for u in Psi_n}),
and since $l,\epsilon^{-1}\ll m$,
\[
H\ge\mathbb{E}_{1\le i\le n}\left(\frac{1}{m}H\left(\pi_{z\mathbb{R}}\varphi_{u\mathbf{I}\left(j,l;i\right)}\nu,\mathcal{D}_{\chi_{u}+\chi_{\mathbf{I}\left(j,l;i\right)}+m}\mid\mathcal{D}_{\chi_{u}+\chi_{\mathbf{I}\left(j,l;i\right)}}\right)\right)-2\epsilon.
\]

From the last inequality, by the definition of $\mathcal{U}_{1}$,
by (\ref{eq:lb on ent for v in U_1}), and since $\Gamma\left(\mathcal{U}_{1}\right)>\frac{\delta}{8M}$
and $\Gamma\left(\mathcal{U}_{2}\right)>1-\epsilon$,
\begin{multline*}
H\ge\Gamma\left(\mathcal{U}_{1}\right)\left(\frac{1}{2}\dim\mu-\epsilon\right)+\Gamma\left(\mathcal{U}_{2}\setminus\mathcal{U}_{1}\right)\left(\dim\mu-1-\epsilon\right)-2\epsilon\\
\ge\dim\mu-1+\frac{\delta}{8M}\left(1-\frac{1}{2}\dim\mu\right)-4\epsilon.
\end{multline*}
Since $\dim\mu<2$ and $\delta^{-1},M\ll\epsilon^{-1}$, this completes
the proof of the proposition.
\end{proof}
We can now prove Proposition \ref{prop:lb on ent of proj of comp of nu},
which is the following statement.
\begin{prop*}
Suppose that $\dim\mu<2$. Then there exists $\gamma>0$ so that for
every $\epsilon>0$, $m\ge M(\epsilon)\ge1$ and $n\ge1$,
\[
\mathbb{P}\left\{ \underset{w\mathbb{R}\in\mathbb{RP}^{1}}{\inf}\frac{1}{m}H\left(\pi_{w\mathbb{R}}\nu_{z,n},\mathcal{D}_{n+m}\right)>\dim\mu-1+\gamma\right\} >1-\epsilon.
\]
\end{prop*}
\begin{proof}
Let $0<\gamma,\eta_{0}<1$ be as obtained in Proposition \ref{prop:lb on ent of proj of cylinders of nu},
and let $\epsilon,\eta\in(0,1)$ and $k,m,n\in\mathbb{Z}_{>0}$ be
with $\gamma^{-1},\eta_{0}^{-1}\ll\epsilon^{-1}\ll\eta^{-1}\ll k\ll m$.
Let $\mathcal{U}_{1}$ be the set of all words $u\in\Psi_{n+k}$ such
that $L(g_{u})\notin B\left(e_{1}\mathbb{C},2\eta\right)$. For each
$u\in\mathcal{U}_{1}$ set $Y_{u}:=Y_{u,\eta}$. Since $\epsilon^{-1}\ll\eta^{-1}\ll k$,
and by Lemma \ref{lem:small ball --> small mass}, we may assume that
$\beta\left(\left[\mathcal{U}_{1}\right]\right)>1-\epsilon$ and $\mu(Y_{u})>1-\epsilon$
for $u\in\mathcal{U}_{1}$.

Exactly as in the proof of Proposition \ref{prop:lb on ent of comp of nu},
we have
\begin{equation}
\mathrm{diam}\left(\mathrm{supp}\left(\psi g_{u}\mu_{Y_{u}}\right)\right)<\eta2^{-n}\text{ for all }u\in\mathcal{U}_{1}.\label{eq:ub on diam(supp(psi g mu_E))-1}
\end{equation}

Let $u\in\mathcal{U}_{1}$ and $z\mathbb{R}\in\mathbb{RP}^{1}$ be
given. Since $\eta^{-1}\ll k$ and $u\in\Psi_{n+k}$, we may assume
that $\Vert g_{u}\Vert_{\mathrm{op}}\ge\eta^{-1}$. Thus, by Proposition
\ref{prop:lb on ent of proj of cylinders of nu},
\[
\frac{1}{m}H\left(\pi_{z\mathbb{R}}\varphi_{u}\nu,\mathcal{D}_{\chi_{u}+m}\mid\mathcal{D}_{\chi_{u}}\right)\ge\dim\mu-1+\gamma.
\]
From this, from $\varphi_{u}\nu=\psi g_{u}\mu$, by the almost-convexity
of entropy (see Section \ref{subsec:Entropy}), since $\mu(Y_{u})>1-\epsilon$,
and from (\ref{eq:norm cond ent on C <=00003D 2}),
\[
\frac{1}{m}H\left(\pi_{z\mathbb{R}}\psi g_{u}\mu_{Y_{u}},\mathcal{D}_{\chi_{u}+m}\mid\mathcal{D}_{\chi_{u}}\right)\ge\dim\mu-1+\gamma-3\epsilon.
\]
Hence, from $u\in\Psi_{n+k}$, (\ref{eq:op norm is comp to 2^(n/2) for u in Psi_n}),
and $\epsilon^{-1},k\ll m$,
\begin{equation}
\frac{1}{m}H\left(\pi_{z\mathbb{R}}\psi g_{u}\mu_{Y_{u}},\mathcal{D}_{n+m}\right)\ge\dim\mu-1+\gamma-4\epsilon\text{ for }u\in\mathcal{U}_{1}\text{ and }z\mathbb{R}\in\mathbb{RP}^{1}.\label{eq:lb on ent of proj of cyl for all u in U_1}
\end{equation}

Let $\mathcal{U}_{2}$ be the set of all $u\in\mathcal{U}_{1}$ for
which there exists $D\in\mathcal{D}_{n}^{\mathbb{C}}$ such that $\mathrm{supp}\left(\psi g_{u}\mu_{Y_{u}}\right)\subset D$.
Exactly as in the proof of Proposition \ref{prop:lb on ent of comp of nu},
using $\beta\left(\left[\mathcal{U}_{1}\right]\right)>1-\epsilon$,
(\ref{eq:ub on diam(supp(psi g mu_E))-1}), and Proposition \ref{prop:nu-mass of neigh of dyd cubes},
it can be shown that $\beta\left(\left[\mathcal{U}_{2}\right]\right)>1-3\epsilon$.

Setting $q:=\sum_{u\in\mathcal{U}_{2}}p_{u}\mu(Y_{u})$,
\[
\nu_{1}:=\frac{1}{q}\sum_{u\in\mathcal{U}_{2}}p_{u}\mu(Y_{u})\cdot\psi g_{u}\mu_{Y_{u}},\text{ and }\nu_{2}:=\frac{1}{1-q}\left(\nu-q\nu_{1}\right),
\]
we have $\nu=q\nu_{1}+(1-q)\nu_{2}$ and $q>1-4\epsilon$. Let $\mathcal{E}$
denote the set of all $D\in\mathcal{D}_{n}^{\mathbb{C}}$ such that
$2\epsilon^{1/2}\nu(D)>(1-q)\nu_{2}(D)$. As in the proof of Proposition
\ref{prop:lb on ent of comp of nu}, from $q>1-4\epsilon$ and by
Markov's inequality, it follows that $\nu\left(\bigcup\mathcal{E}\right)>1-2\epsilon^{1/2}$.

By the definitions of $\mathcal{U}_{2}$ and $\nu_{1}$, given $D\in\mathcal{D}_{n}^{\mathbb{C}}$
with $\nu_{1}(D)>0$, there exist $u_{1},...,u_{l}\in\mathcal{U}_{1}$
and a probability vector $(\rho_{1},...,\rho_{l})$ such that
\[
(\nu_{1})_{D}=\sum_{i=1}^{l}\rho_{i}\cdot\psi g_{u_{i}}\mu_{Y_{u_{i}}}.
\]
Hence, by (\ref{eq:lb on ent of proj of cyl for all u in U_1}) and
the concavity of entropy, for all $z\mathbb{R}\in\mathbb{RP}^{1}$
and $D\in\mathcal{D}_{n}^{\mathbb{C}}$ with $\nu_{1}(D)>0$, 
\begin{equation}
\frac{1}{m}H\left(\pi_{z\mathbb{R}}(\nu_{1})_{D},\mathcal{D}_{n+m}\right)\ge\dim\mu-1+\gamma-4\epsilon.\label{eq:lb on ent of orth proj of (nu_1)}
\end{equation}

Let $D\in\mathcal{E}$, and note that
\[
\nu_{D}=\frac{q\nu_{1}(D)}{\nu(D)}(\nu_{1})_{D}+\frac{(1-q)\nu_{2}(D)}{\nu(D)}(\nu_{2})_{D}.
\]
From this equality and by the definition of $\mathcal{E}$, we obtain
$\nu(D)^{-1}q\nu_{1}(D)>1-2\epsilon^{1/2}$. Thus, by concavity, from
(\ref{eq:lb on ent of orth proj of (nu_1)}), and since $\gamma^{-1}\ll\epsilon^{-1}$,
for each $z\mathbb{R}\in\mathbb{RP}^{1}$ we have
\begin{multline*}
\frac{1}{m}H\left(\pi_{z\mathbb{R}}\nu_{D},\mathcal{D}_{n+m}\right)\ge\frac{q\nu_{1}(D)}{\nu(D)}\frac{1}{m}H\left(\pi_{z\mathbb{R}}(\nu_{1})_{D},\mathcal{D}_{n+m}\right)\\
>\left(1-2\epsilon^{1/2}\right)\left(\dim\mu-1+\gamma-4\epsilon\right)>\dim\mu-1+\gamma/2.
\end{multline*}
As this holds for all $D\in\mathcal{E}$, and since $\nu\left(\bigcup\mathcal{E}\right)>1-2\epsilon^{1/2}$,
this completes the proof of the proposition.
\end{proof}

\section{\label{sec:Proof-of-the ent inc res}Proof of the entropy increase
result}

In this section we establish Theorem \ref{thm:ent inc}. Section \ref{subsec:Entropy-growth-under conv in C}
concerns entropy growth under convolution in $\mathbb{C}$. In Section
\ref{subsec:From-entropy-on G to ent on C}, we show that, in a suitable
sense, nonnegligible entropy on $\mathrm{G}$ translates to nonnegligible
entropy on $\mathbb{C}$. Section \ref{subsec:Linearization} concerns
the linearization part of the argument, and the proof of Theorem \ref{thm:ent inc}
is carried out in Section \ref{subsec:Proof-of-Theorem ent inc}.

\subsection{\label{subsec:Entropy-growth-under conv in C}Entropy growth under
convolution in $\mathbb{C}$}

The following theorem is a direct corollary of Hochman's \cite{Ho}
inverse theorem for entropy growth under convolutions in $\mathbb{R}^{d}$.
We include the derivation for the reader's convenience.
\begin{thm}
\label{thm:ent inc in C}For every $0<\epsilon<1$, $m\ge1$ and $0<\eta<\eta(\epsilon)$,
there exists $\delta=\delta(\epsilon,m,\eta)>0$, such that for all
$n\ge N(\epsilon,m,\eta)\ge1$ the following holds. Let $i\in\mathbb{Z}_{>0}$
and $\theta,\xi\in\mathcal{M}\left(\mathbb{C}\right)$ be such that
\[
\mathrm{diam}(\mathrm{supp}(\theta)),\mathrm{diam}(\mathrm{supp}(\xi))\le\epsilon^{-1}2^{-i},
\]
\begin{equation}
\mathbb{P}_{i\le j\le i+n}\left\{ \frac{1}{m}H\left(\xi_{z,j},\mathcal{D}_{j+m}\right)<2-\epsilon\right\} >1-\eta,\label{eq:dim comp < 2-epsilon in ent inc thm}
\end{equation}
\begin{equation}
\mathbb{P}_{i\le j\le i+n}\left\{ \underset{w\mathbb{R}\in\mathbb{RP}^{1}}{\inf}\frac{1}{m}H\left(\pi_{w\mathbb{R}}\xi_{z,j},\mathcal{D}_{j+m}\right)>\frac{1}{m}H\left(\xi_{z,j},\mathcal{D}_{j+m}\right)-1+\epsilon\right\} >1-\eta,\label{eq:.dim proj comp > dim comp -1 + eps in ent inc thm}
\end{equation}
and
\[
\frac{1}{n}H\left(\theta,\mathcal{D}_{i+n}\right)>\epsilon.
\]
Then,
\begin{equation}
\frac{1}{n}H\left(\theta*\xi,\mathcal{D}_{i+n}\right)\ge\frac{1}{n}H\left(\xi,\mathcal{D}_{i+n}\right)+\delta.\label{eq:ent inc ineq in int inc in C}
\end{equation}
\end{thm}

\begin{proof}
Given an $\mathbb{R}$-linear subspace $V$ of $\mathbb{C}$, we write
$\pi_{V}:\mathbb{C}\rightarrow\mathbb{C}$ for its orthogonal projection,
and $V^{\perp}$ for its orthogonal complement, where $\mathbb{C}$
is identified with $\mathbb{R}^{2}$. Given $\zeta\in\mathcal{M}\left(\mathbb{C}\right)$
and $\rho>0$, we say that $\zeta$ is $(V,\rho)$-concentrated if
$\zeta\left(z+V^{(\rho)}\right)\ge1-\rho$ for some $z\in\mathbb{C}$,
where recall that $V^{(\rho)}$ denotes the closed $\rho$-neighborhood
of $V$ in $\mathbb{C}$.

Let $\epsilon,\eta,\delta\in(0,1)$ and $m,n\in\mathbb{Z}_{>0}$ be
such that $\epsilon^{-1}\ll\eta^{-1}$ and $m,\eta^{-1}\ll\delta^{-1}\ll n$,
let $i\in\mathbb{Z}_{>0}$ and $\theta,\xi\in\mathcal{M}\left(\mathbb{C}\right)$
be such that the conditions of the theorem are satisfied, and assume
by contradiction that (\ref{eq:ent inc ineq in int inc in C}) does
not hold. By \cite[Theorem 2.8]{Ho}, there exist $\mathbb{R}$-linear
subspaces $V_{i},...,V_{i+n}\subset\mathbb{C}$ such that
\[
\mathbb{P}_{i\le j\le i+n}\left\{ \begin{array}{c}
\frac{1}{m}H\left(\xi_{z,j},\mathcal{D}_{j+m}\right)\ge\frac{1}{m}H\left(\pi_{V_{j}^{\perp}}\xi_{z,j},\mathcal{D}_{j+m}\right)+\dim_{\mathbb{R}}V_{j}-\eta\\
\text{ and }S_{2^{j}}\theta_{w,j}\text{ is }(V_{j},\eta)\text{-concentrated}
\end{array}\right\} >1-\eta.
\]
Hence, since Properties (\ref{eq:dim comp < 2-epsilon in ent inc thm})
and (\ref{eq:.dim proj comp > dim comp -1 + eps in ent inc thm})
are satisfied,
\begin{equation}
\mathbb{P}_{i\le j\le i+n}\left\{ S_{2^{j}}\theta_{w,j}\text{ is }(\{0\},\eta)\text{-concentrated}\right\} >1-3\eta.\label{eq:P=00007BdimV_j=00003D0 and ...=00007D>1-3eta}
\end{equation}

On the other hand, since $\frac{1}{n}H\left(\theta,\mathcal{D}_{i+n}\right)>\epsilon$,
by Lemma \ref{lem:glob ent to loc ent}, and since $\epsilon^{-1}\ll\eta^{-1}\ll n$,
it is easy to see that (\ref{eq:P=00007BdimV_j=00003D0 and ...=00007D>1-3eta})
cannot hold. This contradiction completes the proof of the theorem.
\end{proof}

\subsection{\label{subsec:From-entropy-on G to ent on C}Entropy on $\mathrm{G}$
translates to entropy on $\mathbb{C}$}

The purpose of this subsection is to prove the following proposition.
Recall that $1_{\mathrm{G}}$ denotes the identity element of $\mathrm{G}$.
Given $\theta\in\mathcal{M}\left(\mathrm{G}\right)$ and $z\in\mathbb{C}_{\infty}$,
recall also that $\theta.z$ denotes the pushforward of $\theta$
via the map $g\mapsto\varphi_{g}(z)$.
\begin{prop}
\label{prop:from ent on G to ent on C}Let $\xi\in\mathcal{M}(\mathbb{C})$
be nonatomic, set $Q:=\mathrm{supp}(\xi)$, and let $0<r\le1$ be
such that $-g\notin B(1_{\mathrm{G}},r)$ and $\varphi_{g}(z)\ne\infty$
for all $g\in B(1_{\mathrm{G}},r)$ and $z\in Q$. Then, for every
$\epsilon>0$, there exists $\epsilon_{0}=\epsilon_{0}(\xi,r,\epsilon)>0$
such that for all $k\ge K(\xi,r,\epsilon)\ge1$, $n\ge N(\xi,r,\epsilon,k)\ge1$,
and $\theta\in\mathcal{M}\left(B(1_{\mathrm{G}},r)\right)$ with $\frac{1}{n}H(\theta,\mathcal{D}_{n})\ge\epsilon$,
we have
\begin{equation}
\int\mathbb{P}_{1\le i\le n}\left\{ \frac{1}{k}H\left(\theta_{g,i}.z,\mathcal{D}_{i+k}\right)>\epsilon_{0}\right\} \:d\xi(z)>\epsilon_{0}.\label{eq:ent of nu.x nontrivial}
\end{equation}
\end{prop}

The proof of Proposition \ref{prop:from ent on G to ent on C} requires
the following lemma. Given $\left(z_{1},z_{2},z_{3}\right)=z\in\mathbb{C}_{\infty}^{3}$,
let $F_{z}:\mathrm{G}\rightarrow\mathbb{C}_{\infty}^{3}$ be defined
by
\[
F_{z}(g):=\left(\varphi_{g}(z_{1}),\varphi_{g}(z_{2}),\varphi_{g}(z_{3})\right)\text{ for }g\in\mathrm{G}.
\]

\begin{lem}
\label{lem:L_z bi-lip with uni const}Let $Q$ be a compact subset
of $\mathbb{C}$, and let $r>0$ be such that $-g\notin B(1_{\mathrm{G}},r)$
and $\varphi_{g}(z)\ne\infty$ for all $g\in B(1_{\mathrm{G}},r)$
and $z\in Q$. Then, for every $\epsilon>0$, there exists $C=C(Q,r,\epsilon)>1$
such that for all $\left(z_{1},z_{2},z_{3}\right)=z\in Q^{3}$ with
$|z_{i}-z_{j}|\ge\epsilon$ for $1\le i<j\le3$, we have
\begin{equation}
C^{-1}d(g_{1},g_{2})\le\Vert F_{z}(g_{1})-F_{z}(g_{2})\Vert\le Cd(g_{1},g_{2})\text{ for all }g_{1},g_{2}\in B(1_{\mathrm{G}},r),\label{eq:L_z bi-lip with uni const}
\end{equation}
where $\Vert\cdot\Vert$ denotes the standard norm on $\mathbb{C}^{3}$.
\end{lem}

\begin{proof}
Let $r'>r$ be such that $-g\notin B(1_{\mathrm{G}},r')$ and $\varphi_{g}(z)\ne\infty$
for all $g\in B(1_{\mathrm{G}},r')$ and $z\in Q$, and write $U$
for the open ball in $\mathrm{G}$ with center $1_{\mathrm{G}}$ and
radius $r'$. For $g\in\mathrm{G}$ and $z\in\mathbb{C}_{\infty}^{3}$
write $g.z:=F_{z}(g)$, which defines a smooth action of $\mathrm{G}$
on $\mathbb{C}_{\infty}^{3}$.

Let $(z_{1},z_{2},z_{3})=z\in\mathbb{C}^{3}$ be such that $z_{i}\ne z_{j}$
for $1\le i<j\le3$. Since $F_{z}(hg)=h.F_{z}(g)$ for $h,g\in\mathrm{G}$,
it follows that the smooth map $F_{z}:\mathrm{G}\rightarrow\mathbb{C}_{\infty}^{3}$
is of constant rank (see \cite[Theorem 7.25]{Le}). Additionally,
since $-g\notin U$ for $g\in U$ and $z_{1},z_{2},z_{3}$ are distinct,
it follows that $F_{z}|_{U}$ is injective\footnote{Here we use the fact that a Möbius transformation is uniquely determined
by its values on any three distinct points.}. Hence, by the global rank theorem (see \cite[Theorem 4.14]{Le}),
$F_{z}$ is an immersion. Since the manifolds $\mathrm{G}$ and $\mathbb{C}_{\infty}^{3}$
are of the same dimension, it follows that $d(F_{z})_{g}$ is invertible
for each $g\in\mathrm{G}$, where $d(F_{z})_{g}$ is the differential
of $F_{z}$ at $g$.

Let $\epsilon>0$, and write $E$ for the set of $(z_{1},z_{2},z_{3})=z\in Q^{3}$
such that $|z_{i}-z_{j}|\ge\epsilon$ for $1\le i<j\le3$. In what
follows, we equip $\mathrm{G}$ with the left-invariant Riemannian
metric that induces $d_{\mathrm{G}}$, and equip $\mathbb{C}^{3}$
with its standard Riemannian metric. By compactness, and by the preceding
paragraph, it follows that there exists $C_{1}>1$ such that
\[
\left\Vert d(F_{z})_{g}\right\Vert _{\mathrm{op}},\left\Vert \left(d(F_{z})_{g}\right)^{-1}\right\Vert _{\mathrm{op}}\le C_{1}\text{ for all }z\in E\text{ and }g\in B(1_{\mathrm{G}},r').
\]
By compactness, and since $F_{z}|_{B(1_{\mathrm{G}},r')}$ is injective
for $z\in E$, it also follows easily that there exists $\delta>0$
such that $B(F_{z}(g),\delta)\subset F_{z}(U)$ for each $z\in E$
and $g\in B(1_{\mathrm{G}},r)$. Combining these facts, we obtain
that there exists $C>1$ such that (\ref{eq:L_z bi-lip with uni const})
holds for all $z\in E$.
\end{proof}
\begin{proof}[Proof of Proposition \ref{prop:from ent on G to ent on C}]
Since $\xi$ is nonatomic, there exists $0<\delta<1$ such that $\xi\left(B(z,\delta)\right)<1/4$
for all $z\in\mathbb{C}$. Let $0<\epsilon<1$, $C>1$, and $k,n\in\mathbb{Z}_{>0}$
be with
\[
\delta^{-1},\epsilon^{-1}\ll C\ll k\ll n,
\]
suppose that $C$ is also large with respect to $Q$ and $r$, and
let $\theta\in\mathcal{M}\left(B(1_{\mathrm{G}},r)\right)$ be with
$\frac{1}{n}H(\theta,\mathcal{D}_{n})\ge\epsilon$.

By Lemma \ref{lem:glob ent to loc ent} and since $\epsilon^{-1},k\ll n$,
\[
\mathbb{E}_{1\le i\le n}\left(\frac{1}{k}H\left(\theta_{g,i},\mathcal{D}_{i+k}\right)\right)\ge\frac{1}{n}H(\theta,\mathcal{D}_{n})-\epsilon/2\ge\epsilon/2.
\]
Moreover, by Lemma \ref{lem:bd num of subatoms},
\[
\frac{1}{k}H\left(\theta_{D},\mathcal{D}_{i+k}\right)\le C\text{ for all }i\ge0\text{ and }D\in\mathcal{D}_{i}^{\mathrm{G}}\text{ with }\theta(D)>0.
\]
Hence,
\begin{equation}
\mathbb{P}_{1\le i\le n}\left\{ \frac{1}{k}H\left(\theta_{g,i},\mathcal{D}_{i+k}\right)\ge\frac{\epsilon}{4}\right\} \ge\frac{\epsilon}{4C}.\label{eq:comp of theta have non neg ent with no neg prob}
\end{equation}

Write $\xi^{\times3}\in\mathcal{M}\left(\mathbb{C}^{3}\right)$ for
the $3$-fold product of $\xi$ with itself. Let $E$ be the set of
$(z_{1},z_{2},z_{3})=z\in Q^{3}$ such that $|z_{i}-z_{j}|\ge\delta$
for all $1\le i<j\le3$. Since $\xi\left(B(z,\delta)\right)<1/4$
for all $z\in\mathbb{C}$, and by a Fubini-type argument, $\xi^{\times3}(E)\ge1/4$.

Let $i\ge0$ and $D\in\mathcal{D}_{i}^{\mathrm{G}}$ be with $\theta(D)>0$
and $\frac{1}{k}H\left(\theta_{D},\mathcal{D}_{i+k}\right)\ge\frac{\epsilon}{4}$.
By Lemmas \ref{lem:dyad ent =000026 lip func} and \ref{lem:L_z bi-lip with uni const},
and since $\delta^{-1},\epsilon^{-1}\ll C\ll k$, for each $z\in E$
\[
\frac{1}{k}H\left(F_{z}\theta_{D},\mathcal{D}_{i+k}\right)\ge\frac{1}{k}H\left(\theta_{D},\mathcal{D}_{i+k}\right)-\frac{\epsilon}{8}\ge\frac{\epsilon}{8}.
\]
Together with $\xi^{\times3}(E)\ge1/4$, this gives
\begin{equation}
\int\frac{1}{k}H\left(F_{z}\theta_{D},\mathcal{D}_{i+k}\right)\:d\xi^{\times3}(z)\ge2^{-5}\epsilon.\label{eq:integral >=00003D2^-5 epsilon}
\end{equation}

For $1\le j\le3$, let $\pi_{j}:\mathbb{C}^{3}\rightarrow\mathbb{C}$
be the projection onto the $j$th coordinate of $\mathbb{C}^{3}$.
Given $(z_{1},z_{2},z_{3})=z\in\mathbb{C}^{3}$, note that $\pi_{j}F_{z}\theta_{D}=\theta_{D}.z_{j}$
for $1\le j\le3$. Hence, by the conditional entropy formula,
\[
H\left(F_{z}\theta_{D},\mathcal{D}_{i+k}\right)\le\sum_{j=1}^{3}H\left(\pi_{j}F_{z}\theta_{D},\mathcal{D}_{i+k}\right)=\sum_{j=1}^{3}H\left(\theta_{D}.z_{j},\mathcal{D}_{i+k}\right).
\]
Together with (\ref{eq:integral >=00003D2^-5 epsilon}), this gives
\[
2^{-5}\epsilon\le\sum_{j=1}^{3}\int\frac{1}{k}H\left(\theta_{D}.z_{j},\mathcal{D}_{i+k}\right)\:d\xi^{\times3}(z_{1},z_{2},z_{3})=3\int\frac{1}{k}H\left(\theta_{D}.z,\mathcal{D}_{i+k}\right)\:d\xi(z).
\]

We have thus shown that for all $i\ge0$ and $D\in\mathcal{D}_{i}^{\mathrm{G}}$
with $\theta(D)>0$ and $\frac{1}{k}H\left(\theta_{D},\mathcal{D}_{i+k}\right)\ge\frac{\epsilon}{4}$,
\[
\int\frac{1}{k}H\left(\theta_{D}.z,\mathcal{D}_{i+k}\right)\:d\xi(z)\ge2^{-7}\epsilon.
\]
Together with (\ref{eq:comp of theta have non neg ent with no neg prob}),
this implies
\begin{equation}
\int\mathbb{E}_{1\le i\le n}\left(\frac{1}{k}H\left(\left(\theta_{g,i}\right).z,\mathcal{D}_{i+k}\right)\right)\:d\xi(z)\ge2^{-9}C^{-1}\epsilon^{2}.\label{eq:bd of exp of theta.z}
\end{equation}

Given $i\ge0$, $D\in\mathcal{D}_{i}^{\mathrm{G}}$ with $\theta(D)>0$,
and $z\in Q$, we have
\[
\mathrm{diam}\left(\mathrm{supp}\left(\left(\theta_{D}\right).z\right)\right)=O_{Q,r}\left(2^{-i}\right).
\]
Hence, since $k$ is large with respect to $Q$ and $r$, we may assume
that
\[
\frac{1}{k}H\left(\left(\theta_{D}\right).z,\mathcal{D}_{i+k}\right)\le3.
\]
Setting $\epsilon_{0}:=2^{-12}C^{-1}\epsilon^{2}$, together with
(\ref{eq:bd of exp of theta.z}) this gives (\ref{eq:ent of nu.x nontrivial}),
which completes the proof of the proposition.
\end{proof}

\subsection{\label{subsec:Linearization}Linearization}
\begin{lem}
\label{lem:step1 in ent inc pf}Let $Q$ be a compact subset of $\mathbb{C}$,
let $r>0$ be such that $\varphi_{g}(z)\ne\infty$ for all $g\in B(1_{\mathrm{G}},r)$
and $z\in Q$, and let $\theta\in\mathcal{M}\left(B(1_{\mathrm{G}},r)\right)$
and $\xi\in\mathcal{M}(Q)$ be given. Then for all $1\le k\le n$,
\[
\frac{1}{n}H\left(\theta.\xi,\mathcal{D}_{n}\right)\ge\mathbb{E}_{1\le i\le n}\left(\frac{1}{k}H\left(\theta_{g,i}.\xi_{z,i},\mathcal{D}_{i+k}\right)\right)-O_{Q,r}\left(\frac{k}{n}+\frac{1}{k}\right).
\]
\end{lem}

\begin{proof}
By the smoothness of the action map $(g,z)\mapsto\varphi_{g}(z)$,
by the compactness of $B(1_{\mathrm{G}},r)\times Q$, and since $\varphi_{g}(z)\ne\infty$
for all $g\in B(1_{\mathrm{G}},r)$ and $z\in Q$, there exists $C>1$
such that for all $g,h\in B(1_{\mathrm{G}},r)$ and $z,w\in Q$
\[
\left|\varphi_{g}(z)-\varphi_{h}(w)\right|\le C\left(d(g,h)+|z-w|\right).
\]
Using this fact, the lemma follows by an argument similar to that
in the proof of \cite[Lemma 6.9]{HR}.
\end{proof}
\begin{lem}
\label{lem:linearization}Let $Q$ be a compact subset of $\mathbb{C}$,
and let $r>0$ be such that $\varphi_{g}(z)\ne\infty$ for all $g\in B(1_{\mathrm{G}},r)$
and $z\in Q$. Then for every $\epsilon>0$, $k\ge K(\epsilon)\ge1$,
and $0<\delta<\delta(Q,r,\epsilon,k)$ the following holds. Let $g\in B(1_{\mathrm{G}},r)$,
$z\in Q$, $\theta\in\mathcal{M}\left(B(1_{\mathrm{G}},r)\right)$
and $\xi\in\mathcal{M}(Q)$ be such that $d(g,h)\le\delta$ for all
$h\in\mathrm{supp}(\theta)$ and $|z-w|\le\delta$ for all $w\in\mathrm{supp}(\xi)$.
Then,
\[
\left|\frac{1}{k}H\left(\theta.\xi,\mathcal{D}_{k-\log\delta}\right)-\frac{1}{k}H\left(\left(\theta.z\right)*\left(S_{\varphi_{g}'(z)}\xi\right),\mathcal{D}_{k-\log\delta}\right)\right|<\epsilon.
\]
\end{lem}

\begin{proof}
Let $V$ and $U$ be open subsets of $\mathrm{GL}(2,\mathbb{C})$
and $\mathbb{C}$, respectively, such that $B(1_{\mathrm{G}},r)\subset V$,
$Q\subset U$, and $\varphi_{g}(z)\ne\infty$ for all $g\in V$ and
$z\in U$. Let $f:V\times U\rightarrow\mathbb{C}$ be defined by $f(g,z)=\varphi_{g}(z)$
for $(g,z)\in V\times U$. Given $z\in U$, let $f_{z}:V\rightarrow\mathbb{C}$
be defined by $f_{z}(g)=\varphi_{g}(z)$ for $g\in V$. It is easy
to verify that the differential of $f$ at a point $(g,z)\in V\times U$
is given by
\[
df_{(g,z)}(h,w)=d(f_{z})_{g}(h)+\varphi_{g}'(z)w\text{ for }(h,w)\in\mathrm{M}_{2}(\mathbb{C})\times\mathbb{C},
\]
where $d(f_{z})_{g}$ is the differential of $f_{z}$ at $g$, and
$\mathrm{M}_{2}(\mathbb{C})$ denotes the vector space of $2\times2$
complex matrices. Using this fact, the lemma follows by an argument
similar to that in the proof of \cite[Lemma 4.2]{BHR}.
\end{proof}

\subsection{\label{subsec:Proof-of-Theorem ent inc}Proof of Theorem \ref{thm:ent inc}}

We can now prove Theorem \ref{thm:ent inc}, which is the following
statement.
\begin{thm*}
Suppose that $\dim\mu<2$. Then there exists $0<r<1$ such that for
every $\epsilon>0$, there exists $\delta=\delta(\epsilon)>0$ so
that $\frac{1}{n}H(\theta.\mu,\mathcal{D}_{n})>\dim\mu+\delta$ for
all $n\ge N(\epsilon)\ge1$ and $\theta\in\mathcal{M}\left(B(1_{\mathrm{G}},r)\right)$
with $\frac{1}{n}H(\theta,\mathcal{D}_{n})\ge\epsilon$.
\end{thm*}
\begin{proof}
Since $\nu\{\infty\}=0$, there exists $b\in\mathbb{Z}_{>0}$ such
that for
\[
S:=\left\{ z\in\mathbb{C}\::\:\mathrm{Re}(z),\mathrm{Im}(z)\in[-b,b)\right\} 
\]
we have $\nu(S)\ge1/2$. Let $0<r<1$ be such that $-g\notin B(1_{\mathrm{G}},r)$,
$\varphi_{g}(z)\ne\infty$, and $1/2\le\left|\varphi_{g}'(z)\right|\le2$
for all $g\in B(1_{\mathrm{G}},r)$ and $z\in\overline{S}$.

Let $0<\gamma<1$ be as obtained in Proposition \ref{prop:lb on ent of proj of comp of nu},
let $\epsilon,\epsilon_{0},\eta,\delta,\rho\in(0,1)$ and $m,k,n\in\mathbb{Z}_{>0}$
be such that
\[
\gamma^{-1},\epsilon^{-1}\ll\epsilon_{0}^{-1}\ll\eta^{-1}\ll m\ll\delta^{-1}\ll\rho^{-1}\ll k\ll n,
\]
suppose that $\epsilon_{0}^{-1}$ is also large with respect to $S$
and $r$, and let $\theta\in\mathcal{M}\left(B(1_{\mathrm{G}},r)\right)$
be with $\frac{1}{n}H(\theta,\mathcal{D}_{n})\ge\epsilon$.

Setting $\xi:=\nu_{S}$, by Lemma \ref{lem:step1 in ent inc pf} we
have
\[
\frac{1}{n}H\left(\theta.\xi,\mathcal{D}_{n}\right)\ge\mathbb{E}_{1\le i\le n}\left(\frac{1}{k}H\left(\theta_{g,i}.\xi_{z,i},\mathcal{D}_{i+k}\right)\right)-\rho.
\]
Hence, by Lemma \ref{lem:linearization},
\[
\frac{1}{n}H\left(\theta.\xi,\mathcal{D}_{n}\right)\ge\mathbb{E}_{1\le i\le n}\left(\frac{1}{k}H\left(\left(\theta_{g,i}.z\right)*\left(S_{\varphi_{g}'(z)}\xi_{z,i}\right),\mathcal{D}_{i+k}\right)\right)-2\rho.
\]
Thus, since $1/2\le\left|\varphi_{g}'(z)\right|\le2$ for all $g\in B(1_{\mathrm{G}},r)$
and $z\in\overline{S}$,
\begin{equation}
\frac{1}{n}H\left(\theta.\xi,\mathcal{D}_{n}\right)+3\rho\ge\mathbb{E}_{1\le i\le n}\left(\frac{1}{k}H\left(S_{\varphi_{g}'(z)}^{-1}\left(\theta_{g,i}.z\right)*\xi_{z,i},\mathcal{D}_{i+k}\right)\right).\label{eq:by linearization ect}
\end{equation}

Recall the notation $\mathcal{N}_{n}$ and $\lambda_{n}$ from Section
\ref{subsec:Basic-notations}, write $\Gamma:=\lambda_{n}\times\xi\times\theta$,
and let $E_{1}$ be the set of all $(i,z,g)\in\mathcal{N}_{n}\times S\times B(1_{\mathrm{G}},r)$
such that
\[
\frac{1}{k}H\left(\xi_{z,i},\mathcal{D}_{i+k}\right)\ge\dim\mu-\rho.
\]
By Proposition \ref{prop:lb on ent of comp of nu}, we may assume
that $\Gamma(E_{1})>1-\rho$. Also, by \cite[Corollary 4.10]{Ho1},
\begin{equation}
\frac{1}{k}H\left(S_{\varphi_{g}'(z)}^{-1}\left(\theta_{g,i}.z\right)*\xi_{z,i},\mathcal{D}_{i+k}\right)>\dim\mu-2\rho\;\text{ for }(i,z,g)\in E_{1}.\label{eq:lb on E_1}
\end{equation}

Let $E_{2}$ be the set of all $(i,z,g)\in E_{1}$ such that
\[
\mathbb{P}_{i\le j\le i+k}\left\{ \frac{1}{m}H\left(\left(\xi_{z,i}\right)_{w,j},\mathcal{D}_{j+m}\right)<1+\frac{1}{2}\dim\mu\right\} >1-\eta,
\]
\[
\mathbb{P}_{i\le j\le i+k}\left\{ \begin{array}{c}
\underset{u\mathbb{R}\in\mathbb{RP}^{1}}{\inf}\frac{1}{m}H\left(\pi_{u\mathbb{R}}\left(\xi_{z,i}\right)_{w,j},\mathcal{D}_{j+m}\right)\\
>\frac{1}{m}H\left(\left(\xi_{z,i}\right)_{w,j},\mathcal{D}_{j+m}\right)-1+\gamma/2
\end{array}\right\} >1-\eta,
\]
and
\[
\frac{1}{k}H\left(S_{\varphi_{g}'(z)}^{-1}\left(\theta_{g,i}.z\right),\mathcal{D}_{i+k}\right)>\epsilon_{0}.
\]
By Propositions \ref{prop:uni ent dim}, \ref{prop:lb on ent of proj of comp of nu}
and \ref{prop:from ent on G to ent on C}, from \cite[Lemma 2.7]{Ho},
and since $\dim\mu<2$ and $1/2\le\left|\varphi_{g}'(z)\right|\le2$
for all $g\in B(1_{\mathrm{G}},r)$ and $z\in\overline{S}$, we may
assume that $\Gamma(E_{2})>\epsilon_{0}$.

Given $(i,z,g)\in E_{2}$, note that
\[
\mathrm{diam}\left(S_{\varphi_{g}'(z)}^{-1}\left(\theta_{g,i}.z\right)\right),\mathrm{diam}\left(\xi_{z,i}\right)=O_{S,r}\left(2^{-i}\right).
\]
Hence, by Theorem \ref{thm:ent inc in C},
\[
\frac{1}{k}H\left(S_{\varphi_{g}'(z)}^{-1}\left(\theta_{g,i}.z\right)*\xi_{z,i},\mathcal{D}_{i+k}\right)\ge\frac{1}{k}H\left(\xi_{z,i},\mathcal{D}_{i+k}\right)+\delta.
\]
Thus, since $E_{2}\subset E_{1}$,
\begin{equation}
\frac{1}{k}H\left(S_{\varphi_{g}'(z)}^{-1}\left(\theta_{g,i}.z\right)*\xi_{z,i},\mathcal{D}_{i+k}\right)\ge\dim\mu-\rho+\delta\;\text{ for }(i,z,g)\in E_{2}.\label{eq:lb on E_2}
\end{equation}

Now, from (\ref{eq:by linearization ect}), (\ref{eq:lb on E_1})
and (\ref{eq:lb on E_2}),
\[
\frac{1}{n}H\left(\theta.\xi,\mathcal{D}_{n}\right)+3\rho\ge\Gamma\left(E_{1}\setminus E_{2}\right)\left(\dim\mu-2\rho\right)+\Gamma\left(E_{2}\right)\left(\dim\mu-\rho+\delta\right).
\]
Hence, recalling that $\xi:=\nu_{S}$ and since $\Gamma(E_{1})>1-\rho$
and $\Gamma(E_{2})>\epsilon_{0}$,
\begin{equation}
\frac{1}{n}H\left(\theta.\nu_{S},\mathcal{D}_{n}\right)\ge\dim\mu+\epsilon_{0}\delta-O(\rho).\label{eq:lb on ent of theta.nu_S}
\end{equation}

Setting
\[
K:=\left\{ \varphi_{g}(z)\::\:g\in B(1_{\mathrm{G}},r)\text{ and }z\in\overline{S}\right\} ,
\]
it holds that $K$ is a compact subset of $\mathbb{C}$. Hence, by
Lemma \ref{lem:bi-lip prop of psi}, the restriction of $\psi^{-1}$
to $K$ is a bi-Lipschitz map with bi-Lipschitz constant depending
only on $S$ and $r$. Since $\epsilon_{0}^{-1}$ is large with respect
to $S$ and $r$, we may assume that this bi-Lipschitz constant is
at most $\epsilon_{0}^{-1}$. Note also that $\mathrm{supp}\left(\theta.\nu_{S}\right)\subset K$,
and that $\psi^{-1}(\theta.\nu_{S})=\theta.\mu_{\psi^{-1}(S)}$. Thus,
from (\ref{eq:lb on ent of theta.nu_S}), by Lemma \ref{lem:dyad ent =000026 lip func},
and since $\epsilon_{0}^{-1},\rho\ll n$,
\begin{equation}
\frac{1}{n}H\left(\theta.\mu_{\psi^{-1}(S)},\mathcal{D}_{n}\right)\ge\dim\mu+\epsilon_{0}\delta-O(\rho).\label{eq:lb on ent of theta.mu_psi^-1(S)}
\end{equation}

Assuming $\nu\left(\mathbb{C}\setminus S\right)>0$, the exact dimensionality
of $\mu$ implies that $\mu_{\psi^{-1}\left(\mathbb{C}\setminus S\right)}$
is also exact dimensional with dimension $\dim\mu$. Hence, by Lemma
\ref{lem:from dim to ent} and since $n$ is large with respect to
$S$ and $\rho$,
\[
\frac{1}{n}H\left(\mu_{\psi^{-1}\left(\mathbb{C}\setminus S\right)},\mathcal{D}_{n}\right)>\dim\mu-\rho.
\]

Since $B(1_{\mathrm{G}},r)$ is compact, we may assume that the map
sending $z\mathbb{C}\in\mathbb{CP}^{1}$ to $gz\mathbb{C}$ is bi-Lipschitz,
with bi-Lipschitz constant at most $\epsilon_{0}^{-1}$, for all $g\in B(1_{\mathrm{G}},r)$.
From this, by concavity of entropy, by Lemma \ref{lem:dyad ent =000026 lip func},
since $\epsilon_{0}^{-1},\rho\ll n$, and by the last inequality,
\[
\frac{1}{n}H\left(\theta.\mu_{\psi^{-1}\left(\mathbb{C}\setminus S\right)},\mathcal{D}_{n}\right)\ge\int\frac{1}{n}H\left(g\mu_{\psi^{-1}\left(\mathbb{C}\setminus S\right)},\mathcal{D}_{n}\right)\:d\theta(g)>\dim\mu-2\rho.
\]
Thus, by concavity, from (\ref{eq:lb on ent of theta.mu_psi^-1(S)}),
and since $\nu(S)\ge1/2$,
\[
\frac{1}{n}H\left(\theta.\mu,\mathcal{D}_{n}\right)\ge\dim\mu+\frac{1}{2}\epsilon_{0}\delta-O(\rho).
\]
Since $\epsilon_{0}^{-1},\delta^{-1}\ll\rho^{-1}$, this completes
the proof of the theorem.
\end{proof}

\section{\label{sec:Proof-of-the main result}Proof of the main result}

In this section we establish Theorem \ref{thm:main result}. Section
\ref{subsec:Preparations-for-the proof of main thm} contains preparations
for the proof, which is carried out in Section \ref{subsec:Proof-of-Theorem main}.

\subsection{\label{subsec:Preparations-for-the proof of main thm}Preparations
for the proof}

We begin by establishing the natural upper bound. Recall the definition
of $h_{\mathrm{RW}}$ from (\ref{eq:def of h_RW}).
\begin{lem}
\label{lem:h/2chi is ub for dim mu}It always holds that $\dim\mu\le\min\left\{ 2,\frac{h_{\mathrm{RW}}}{2\chi}\right\} $.
\end{lem}

\begin{proof}
Since $\dim\mathbb{CP}^{1}=2$ as a real manifold, we clearly have
$\dim\mu\le2$.

Given $n\ge1$, write $\mathcal{G}_{n}:=\left\{ g_{u}\::\:u\in\Lambda^{n}\right\} $,
and denote by $\mathrm{S}_{\mathcal{G}_{n}}$the subsemigroup of $\mathrm{G}$
generated by $\mathcal{G}_{n}$. Since $\mathrm{S}_{\mathcal{G}}$
is strongly irreducible and proximal, it is easy to see that the same
holds for $\mathrm{S}_{\mathcal{G}_{n}}$. Additionally, by Lemma
\ref{lem:No Q exists}, it follows easily that $\mathrm{S}_{\mathcal{G}_{n}}$
does not fix a generalized circle.

For $g\in\mathcal{G}_{n}$, set
\[
q_{n,g}:=\sum_{u\in\Lambda^{n},g_{u}=g}p_{u},
\]
and note that $\mu$ equals the Furstenberg measure associated to
$\mathcal{G}_{n}$ and the probability vector $q_{n}:=\left(q_{n,g}\right)_{g\in\mathcal{G}_{n}}$.
Moreover, the Lyapunov exponent associated to $\mathcal{G}_{n}$ and
$q_{n}$ equals $n\chi$. Hence, by Theorem \ref{thm:exact dim of mu and LY}
and since $\Delta\ge0$,
\[
\dim\mu\le\frac{H\left(q_{n}\right)}{2n\chi}\text{ for all }n\ge1.
\]

On the other hand, by the definition of $h_{\mathrm{RW}}$,
\[
h_{\mathrm{RW}}:=\underset{n\to\infty}{\lim}\frac{1}{n}H\left(q_{n}\right).
\]
Thus, $\dim\mu\le h_{\mathrm{RW}}/\left(2\chi\right)$, which completes
the proof of the lemma.
\end{proof}
From (\ref{eq:def of L(omega)}) it follows that the sequence $\left\{ \omega\mapsto L\left(g_{\omega|_{n}}\right)\right\} _{n\ge1}$
converges in probability to $\omega\mapsto L(\omega)$. The following
lemma provides a quantitative rate for this convergence. It could
be deduced from Ruelle\textquoteright s proof of the multiplicative
ergodic theorem (see \cite[Lemma I.4]{Ru}), but we include a complete
proof for the reader's convenience.
\begin{lem}
\label{lem:dist est bet L(om) and L(g_om|_n)}For every $\eta>0$
and $n\ge N(\eta)\ge1$,
\[
\beta\left\{ \omega\in\Lambda^{\mathbb{N}}\::\:d\left(L\left(\omega\right),L\left(g_{\omega|_{n}}\right)\right)\le2^{-n\left(2\chi-\eta\right)}\right\} >1-\eta.
\]
\end{lem}

\begin{proof}
Let $\eta,\delta\in(0,1)$ and $n\in\mathbb{Z}_{>0}$ be with $\eta^{-1}\ll\delta^{-1}\ll n$,
and let $E$ be the set of all $\omega\in\Lambda^{\mathbb{N}}$ such
that
\[
L\left(\omega\right)=g_{\omega|_{n}}L\left(\sigma^{n}\omega\right),\:d\left(L\left(g_{\omega|_{n}}^{-1}\right),L\left(\sigma^{n}\omega\right)\right)>\delta,\text{ and }\Vert g_{\omega|_{n}}\Vert_{\mathrm{op}}\ge2^{n(\chi-\eta/4)}.
\]
By Lemma \ref{lem:small ball --> small mass} and $\eta^{-1}\ll\delta^{-1}$,
\[
\mu\left(B\left(z\mathbb{C},\delta\right)\right)<\eta/2\text{ for all }z\mathbb{C}\in\mathbb{CP}^{1}.
\]
Thus, since the maps $\omega\mapsto L\left(g_{\omega|_{n}}^{-1}\right)$
and $\omega\mapsto L\left(\sigma^{n}\omega\right)$ are $\beta$-independent,
since $\omega\mapsto L\left(\sigma^{n}\omega\right)$ is distributed
according to $\mu$, from (\ref{eq:chi =00003D lim a.s.}) and (\ref{eq:L is equivariant}),
and since $\eta^{-1}\ll n$, we may assume that $\beta(E)>1-\eta$.

Additionally, from Lemma \ref{lem:dist of gzC to L(g)} and since
$\eta^{-1},\delta^{-1}\ll n$, it follows that for $\omega\in E$
\[
d\left(L\left(\omega\right),L\left(g_{\omega|_{n}}\right)\right)=d\left(g_{\omega|_{n}}L\left(\sigma^{n}\omega\right),L\left(g_{\omega|_{n}}\right)\right)\le\delta^{-1}\Vert g_{\omega|_{n}}\Vert_{\mathrm{op}}^{-2}\le2^{-n(2\chi-\eta)}.
\]
Since $\beta(E)>1-\eta$, this completes the proof of the lemma.
\end{proof}
The proof of Theorem \ref{thm:main result} requires partitioning
subsets of $\mathbb{CP}^{1}$ and $\mathrm{G}$ into smaller pieces,
while controlling the cardinality of the partition. This is the content
of the following lemma.
\begin{lem}
\label{lem:part of bd cardin}Let $X$ denote $\mathbb{CP}^{1}$ or
$\mathrm{G}$, and let $R>1$ be given. Then for every $0<\epsilon<1$
and Borel set $\emptyset\ne F\subset X$ with $\epsilon\le\mathrm{diam}(F)\le R$,
there exists a Borel partition $\mathcal{E}$ of $F$ such that
\[
\log\left|\mathcal{E}\right|=O_{X,R}\left(1+\log\left(\mathrm{diam}(F)/\epsilon\right)\right)
\]
and $\mathrm{diam}(E)\le\epsilon$ for each $E\in\mathcal{E}$.
\end{lem}

\begin{proof}
Let $0<\epsilon<1$, and let $\emptyset\ne F\subset X$ be a Borel
set with $\epsilon\le\mathrm{diam}(F)\le R$. Let $C=C(X)>1$ be the
constant appearing in (\ref{eq:in second prop of D^X}), let $n\in\mathbb{Z}_{>0}$
be with $2^{-n}\le\frac{\epsilon}{2C}<2^{1-n}$, and set
\[
\mathcal{E}:=\left\{ D\cap F\::\:D\in\mathcal{D}_{n}^{X}\text{ and }D\cap F\ne\emptyset\right\} .
\]

By (\ref{eq:in second prop of D^X}), for each $D\in\mathcal{D}_{n}^{X}$
we have $\mathrm{diam}(D)\le2C2^{-n}\le\epsilon$. Additionally, by
Lemma \ref{lem:ub on card of dyad atoms intersecting} and since $\frac{\epsilon}{2C}<2^{1-n}$,
\[
\log\left|\mathcal{E}\right|=O_{X,R}\left(1+\log\left(\mathrm{diam}(F)/\epsilon\right)\right),
\]
which completes the proof of the lemma.
\end{proof}
The following lemma provides a uniform upper bound on the diameter
of certain subsets of $\mathrm{G}$. This will be needed when applying
Lemma \ref{lem:part of bd cardin} with $X=\mathrm{G}$.
\begin{lem}
\label{lem:dist est in G}There exists $R>1$ such that $d\left(g_{1},g_{2}\right)\le R$
for all $g_{1},g_{2}\in\mathrm{G}$ with
\begin{equation}
\frac{1}{2}\le\frac{\Vert g_{1}\Vert_{\mathrm{op}}}{\Vert g_{2}\Vert_{\mathrm{op}}}\le2\text{ and }d\left(L(g_{1}),L(g_{2})\right)\le\Vert g_{1}\Vert_{\mathrm{op}}^{-2}.\label{eq:assump on g_1 and g_2}
\end{equation}
\end{lem}

\begin{proof}
Let $g_{1},g_{2}\in\mathrm{G}$ be such that (\ref{eq:assump on g_1 and g_2})
holds. If $\Vert g_{1}\Vert_{\mathrm{op}}=1$ or $\Vert g_{2}\Vert_{\mathrm{op}}=1$,
then $g_{1}$ and $g_{2}$ both belong to the compact set $\left\{ g\in\mathrm{G}\::\:\Vert g\Vert_{\mathrm{op}}\le2\right\} $.
Hence, we may assume that $\Vert g_{1}\Vert_{\mathrm{op}},\Vert g_{2}\Vert_{\mathrm{op}}>1$.
For $i=1,2$, let $U_{i}D_{i}V_{i}$ be a singular value decomposition
of $g_{i}$ (see Section \ref{subsec:Algebraic-notation}).

Set $z:=U_{2}^{-1}U_{1}e_{1}$, and let $z_{1},z_{2}\in\mathbb{C}$
be with $z=\left(z_{1},z_{2}\right)$. By the definition of $d_{\mathbb{CP}^{1}}$,
and since then map $w\mathbb{C}\mapsto U_{2}w\mathbb{C}$ is an isometry
of $\mathbb{CP}^{1}$, 
\begin{multline*}
|z_{2}|=\left|\det\left(\begin{array}{cc}
1 & z_{1}\\
0 & z_{2}
\end{array}\right)\right|=d\left(e_{1}\mathbb{C},z\mathbb{C}\right)=d\left(U_{2}e_{1}\mathbb{C},U_{1}e_{1}\mathbb{C}\right)\\
=d\left(L(g_{1}),L(g_{2})\right)\le\Vert g_{1}\Vert_{\mathrm{op}}^{-2}.
\end{multline*}
From this and since $\frac{1}{2}\le\frac{\Vert g_{1}\Vert_{\mathrm{op}}}{\Vert g_{2}\Vert_{\mathrm{op}}}\le2$,
\begin{multline*}
\Vert g_{2}^{-1}g_{1}V_{1}^{-1}e_{1}\Vert=\Vert D_{2}^{-1}U_{2}^{-1}U_{1}D_{1}e_{1}\Vert=\Vert g_{1}\Vert_{\mathrm{op}}\Vert D_{2}^{-1}z\Vert\\
=\left(\frac{\Vert g_{1}\Vert_{\mathrm{op}}^{2}}{\Vert g_{2}\Vert_{\mathrm{op}}^{2}}|z_{1}|^{2}+\Vert g_{1}\Vert_{\mathrm{op}}^{2}\Vert g_{2}\Vert_{\mathrm{op}}^{2}|z_{2}|^{2}\right)^{1/2}\le8^{1/2}.
\end{multline*}
Set $w=U_{2}^{-1}U_{1}e_{2}$, and let $w_{1},w_{2}\in\mathbb{C}$
be with $w=\left(w_{1},w_{2}\right)$. Since $\frac{1}{2}\le\frac{\Vert g_{1}\Vert_{\mathrm{op}}}{\Vert g_{2}\Vert_{\mathrm{op}}}\le2$,
\begin{multline*}
\Vert g_{2}^{-1}g_{1}V_{1}^{-1}e_{2}\Vert=\Vert D_{2}^{-1}U_{2}^{-1}U_{1}D_{1}e_{2}\Vert=\Vert g_{1}\Vert_{\mathrm{op}}^{-1}\Vert D_{2}^{-1}w\Vert\\
=\Vert g_{1}\Vert_{\mathrm{op}}^{-1}\left(\Vert g_{2}\Vert_{\mathrm{op}}^{-2}|w_{1}|^{2}+\Vert g_{2}\Vert_{\mathrm{op}}^{2}|w_{2}|^{2}\right)^{1/2}\le5^{1/2}.
\end{multline*}

Since $\left\{ V_{1}^{-1}e_{1},V_{1}^{-1}e_{2}\right\} $ is an orthonormal
basis of $\mathbb{C}^{2}$, the inequalities above imply that $g_{2}^{-1}g_{1}$
belongs to the compact set 
\[
\left\{ g\in\mathrm{G}\::\:\Vert g\Vert_{\mathrm{op}}\le5^{1/2}+8^{1/2}\right\} ,
\]
from which it follows that $d\left(g_{2}^{-1}g_{1},1_{\mathrm{G}}\right)=O(1)$.
Thus, by the left invariance of $d_{\mathrm{G}}$, we obtain $d\left(g_{1},g_{2}\right)=O(1)$,
which completes the proof.
\end{proof}
The following lemma will be useful for applying Theorem \ref{thm:ent inc}
in situations where the measure $\theta\in\mathcal{M}\left(\mathrm{G}\right)$
is supported far from the identity.
\begin{lem}
\label{lem:scaling of ent in CP^1}For every $0<\eta<1$ and $n\ge N(\eta)\ge1$
the following holds. Let $g\in\mathrm{G}$ be with $\left|\frac{1}{n}\log\Vert g\Vert_{\mathrm{op}}-\chi\right|<\eta$.
Then for every $\theta\in\mathcal{M}\left(B\left(1_{\mathrm{G}},1\right)\right)$
and $M\ge0$,
\[
\left|\frac{1}{n}H\left(g\left(\theta.\mu\right),\mathcal{D}_{\left(M+2\chi\right)n}\right)-\frac{1}{n}H\left(\theta.\mu,\mathcal{D}_{Mn}\right)\right|=O(\eta(1+M)).
\]
\end{lem}

\begin{proof}
Let $\eta,\delta\in(0,1)$ and $n\in\mathbb{Z}_{>0}$ be such that
$\eta^{-1}\ll\delta^{-1}\ll n$, let $g\in\mathrm{G}$ be with $\left|\frac{1}{n}\log\Vert g\Vert_{\mathrm{op}}-\chi\right|<\eta$,
and let $\theta\in\mathcal{M}\left(B\left(1_{\mathrm{G}},1\right)\right)$
and $M\ge0$ be given. Set $\xi:=\theta.\mu$ and $Y:=\mathbb{CP}^{1}\setminus B\left(L\left(g^{-1}\right),\delta\right)$.

Since $B\left(1_{\mathrm{G}},1\right)$ is a compact subset of $\mathrm{G}$,
and by Lemmas \ref{lem:zC to gzC is bi-Lip with const norm^2} and
\ref{lem:small ball --> small mass}, we may assume that $g'\mu(Y)>1-\eta$
for all $g'\in B\left(1_{\mathrm{G}},1\right)$. Since $\xi=\int g'\mu\:d\theta(g')$,
this gives $\xi(Y)>1-\eta$. From this, from (\ref{eq:ub on card of dyd part of CP^1}),
and by concavity and almost-convexity (see Section \ref{subsec:Entropy}),
we obtain
\begin{equation}
\left|\frac{1}{n}H\left(\xi_{Y},\mathcal{D}_{Mn}\right)-\frac{1}{n}H\left(\xi,\mathcal{D}_{Mn}\right)\right|=O\left(\eta(1+M)\right)\label{eq:ent of xi_Y close to ent of xi}
\end{equation}
and
\begin{equation}
\left|\frac{1}{n}H\left(g\xi_{Y},\mathcal{D}_{\left(M+2\chi\right)n}\right)-\frac{1}{n}H\left(g\xi,\mathcal{D}_{\left(M+2\chi\right)n}\right)\right|=O\left(\eta(1+M)\right).\label{eq:ent of g xi_Y close to ent of g xi}
\end{equation}

Since $\eta^{-1},\delta^{-1}\le n$, we may assume that $\delta^{-2}\le2^{n\eta}$.
From this, from $\left|\frac{1}{n}\log\Vert g\Vert_{\mathrm{op}}-\chi\right|<\eta$,
and by Lemmas \ref{lem:zC to gzC is bi-Lip with const norm^2} and
\ref{lem:dist of gzC to L(g)}, it follows that for every $z\mathbb{C},w\mathbb{C}\in Y$
\[
2^{-2n\eta}2^{-2n\chi}d\left(z\mathbb{C},w\mathbb{C}\right)\le d\left(gz\mathbb{C},gw\mathbb{C}\right)\le2^{3n\eta}2^{-2n\chi}d\left(z\mathbb{C},w\mathbb{C}\right).
\]
Hence, by applying Lemma \ref{lem:dyad ent =000026 lip func} with
$s=2^{-2n\chi}$ and $C=2^{3n\eta}$,
\[
\left|\frac{1}{n}H\left(g\xi_{Y},\mathcal{D}_{\left(M+2\chi\right)n}\right)-\frac{1}{n}H\left(\xi_{Y},\mathcal{D}_{Mn}\right)\right|=O\left(\eta\right).
\]
This, together with (\ref{eq:ent of xi_Y close to ent of xi}) and
(\ref{eq:ent of g xi_Y close to ent of g xi}), completes the proof
of the lemma.
\end{proof}

\subsection{\label{subsec:Proof-of-Theorem main}Proof of Theorem \ref{thm:main result}}

We can now prove our main result. For the reader's convenience, we
recall the statement of Theorem \ref{thm:main result} before its
proof.
\begin{thm*}
Suppose that $\mathrm{S}_{\mathcal{G}}$ is strongly irreducible,
proximal, and does not fix a generalized circle. Assume moreover that
$\mathcal{G}$ is weakly Diophantine. Then,
\[
\dim\mu=\min\left\{ 2,\frac{h_{\mathrm{RW}}}{2\chi}\right\} .
\]
\end{thm*}
\begin{proof}
By Lemma \ref{lem:h/2chi is ub for dim mu}, we only need to show
that $\dim\mu\ge\min\left\{ 2,\frac{h_{\mathrm{RW}}}{2\chi}\right\} $.
Assume by contradiction that $\dim\mu<\min\left\{ 2,\frac{h_{\mathrm{RW}}}{2\chi}\right\} $.
From this and by Theorem \ref{thm:exact dim of mu and LY}, it follows
that there exists $0<\epsilon<1$ such that
\begin{equation}
H(p)-h_{\mathrm{RW}}<\Delta-\epsilon,\label{eq:H(p)-h_RW<Delta}
\end{equation}
where $\Delta$ is defined in Section \ref{subsec:Exact-dimensionality-and LY formula}.

Since $\mathcal{G}$ is weakly Diophantine, there exists $c>0$ such
that for infinitely many integers $n\ge1$,
\begin{equation}
d\left(g_{u_{1}},g_{u_{2}}\right)\ge c^{n}\text{ for all }u_{1},u_{2}\in\Lambda^{n}\text{ with }g_{u_{1}}\ne g_{u_{2}}.\label{eq:by Diophantine}
\end{equation}
By (\ref{eq:in second prop of D^X}), there exists $M=M(c)>1$ such
that
\begin{equation}
\mathcal{D}_{Mn}^{\mathrm{G}}\left(g\right)\ne\mathcal{D}_{Mn}^{\mathrm{G}}\left(g'\right)\text{ for all }n\ge1\text{ and }g,g'\in\mathrm{G}\text{ with }d\left(g,g'\right)\ge c^{n}.\label{eq:def prop of M}
\end{equation}

Let $0<\eta<1$ and $n\in\mathbb{Z}_{>0}$ be such that $\epsilon^{-1},M\ll\eta^{-1}\ll n$,
and (\ref{eq:by Diophantine}) holds. Given $\xi\in\mathcal{M}\left(\mathbb{CP}^{1}\right)$,
set
\[
\widehat{H}\left(\xi\right):=\frac{1}{Mn}H\left(\xi,\mathcal{D}_{\left(M+2\chi\right)n}\mid\mathcal{D}_{2\chi n}\right).
\]
By Lemma \ref{lem:from dim to ent} and since $\mu$ is exact dimensional,
we may assume that
\begin{equation}
\dim\mu\ge\widehat{H}(\mu)-\eta.\label{eq:dim mu >=00003D H(mu)-eta}
\end{equation}

Let $\Pi_{n}:\Lambda^{\mathbb{N}}\rightarrow\mathrm{G}$ be defined
by $\Pi_{n}(\omega)=g_{\omega|_{n}}$ for $\omega\in\Lambda^{\mathbb{N}}$,
and recall from Section \ref{subsec:Exact-dimensionality-and LY formula}
that $\left\{ \beta_{\omega}\right\} _{\omega\in\Lambda^{\mathbb{N}}}$
denotes the disintegration of $\beta$ with respect to $L^{-1}\mathcal{B}_{\mathbb{CP}^{1}}$.
From $\mu=\sum_{i\in\Lambda}p_{i}\cdot g_{i}\mu$ and $\beta=\int\beta_{\omega}\:d\beta(\omega)$,
we obtain
\[
\mu=\sum_{u\in\Lambda^{n}}p_{u}\cdot g_{u}\mu=\left(\Pi_{n}\beta\right).\mu=\int\left(\Pi_{n}\beta_{\omega}\right).\mu\:d\beta(\omega).
\]
Hence, by (\ref{eq:dim mu >=00003D H(mu)-eta}) and the concavity
of conditional entropy,
\begin{equation}
\dim\mu\ge\int\widehat{H}\left(\left(\Pi_{n}\beta_{\omega}\right).\mu\right)\:d\beta(\omega)-\eta.\label{eq:ineq to contradict in main pf}
\end{equation}
To prove the theorem, we shall derive a contradiction with (\ref{eq:ineq to contradict in main pf}).

Set
\[
E_{0}:=\left\{ \omega\in\Lambda^{\mathbb{N}}\::\:\left|\frac{1}{n}\log\Vert g_{\omega|_{n}}\Vert_{\mathrm{op}}-\chi\right|<\eta\right\} ,
\]
and let $E_{1}$ be the set of all $\omega\in\Lambda^{\mathbb{N}}$
such that $\beta_{\omega}\left(E_{0}\right)>1-\eta$. By (\ref{eq:chi =00003D lim a.s.}),
$\eta^{-1}\ll n$, and $\beta=\int\beta_{\omega}\:d\beta(\omega)$,
we may assume that $\beta\left(E_{1}\right)>1-\eta$.

Write $\mathcal{E}_{n}:=\left\{ \Pi_{n}^{-1}\{g\}\::\:g\in\mathrm{G}\right\} $
for the partition of $\Lambda^{\mathbb{N}}$ into level sets of $\Pi_{n}$,
and recall that $\mathcal{P}_{n}$ denotes the partition of $\Lambda^{\mathbb{N}}$
into level-$n$ cylinders. By (\ref{eq:def of h_RW}),
\[
h_{\mathrm{RW}}\le\frac{1}{n}H\left(\Pi_{n}\beta\right)=\frac{1}{n}H\left(\beta,\mathcal{E}_{n}\right),
\]
where $H\left(\Pi_{n}\beta\right)$ denotes the Shannon entropy of
the discrete probability measure $\Pi_{n}\beta$. By the last formula,
from (\ref{eq:H(p)-h_RW<Delta}), and since $H(p)=\frac{1}{n}H\left(\beta,\mathcal{P}_{n}\right)$,
\[
\Delta-\epsilon>\frac{1}{n}H\left(\beta,\mathcal{P}_{n}\right)-\frac{1}{n}H\left(\beta,\mathcal{E}_{n}\right)=\frac{1}{n}H\left(\beta,\mathcal{P}_{n}\mid\mathcal{E}_{n}\right).
\]
Thus, by the concavity of conditional entropy,
\begin{equation}
\Delta-\epsilon>\int\frac{1}{n}H\left(\beta_{\omega},\mathcal{P}_{n}\mid\mathcal{E}_{n}\right)\:d\beta(\omega).\label{eq:Delta-epsilon>int of cond ents}
\end{equation}

By Theorem \ref{thm:exact dim of mu and LY} and since $\epsilon^{-1}\ll n$,
\[
\int\frac{1}{n}H\left(\beta_{\omega},\mathcal{P}_{n}\right)\:d\beta(\omega)>\Delta-\epsilon/2.
\]
Hence, by (\ref{eq:Delta-epsilon>int of cond ents}),
\[
\int\frac{1}{n}H\left(\Pi_{n}\beta_{\omega}\right)\:d\beta(\omega)>\epsilon/2.
\]
From this and since
\begin{equation}
\frac{1}{n}H\left(\Pi_{n}\xi\right)\le\log|\Lambda|\text{ for each }\xi\in\mathcal{M}\left(\Lambda^{\mathbb{N}}\right),\label{eq:ub on Shannon ent of Pi_n xi}
\end{equation}
we obtain
\begin{equation}
\beta\left\{ \omega\in\Lambda^{\mathbb{N}}\::\:\frac{1}{n}H\left(\Pi_{n}\beta_{\omega}\right)\ge\epsilon/4\right\} \ge\frac{\epsilon}{4\log|\Lambda|}.\label{eq:Shannon > ep/4}
\end{equation}

Let $E_{2}$ be the set of all $\omega\in E_{1}$ such that $\frac{1}{n}H\left(\Pi_{n}\beta_{\omega}\right)\ge\epsilon/4$
and
\[
\beta_{\omega}\left\{ \omega'\in\Lambda^{\mathbb{N}}\::\:d\left(L\left(\omega\right),L\left(g_{\omega'|_{n}}\right)\right)\le2^{-n\left(2\chi-\eta\right)}\right\} >1-\eta.
\]
Note that, by the definition of $\left\{ \beta_{\omega}\right\} _{\omega\in\Lambda^{\mathbb{N}}}$,
for $\beta$-a.e. $\omega$ we have $L\left(\omega'\right)=L\left(\omega\right)$
for $\beta_{\omega}$-a.e. $\omega'$. From this, by Lemma \ref{lem:dist est bet L(om) and L(g_om|_n)},
since $\beta\left(E_{1}\right)>1-\eta$, from (\ref{eq:Shannon > ep/4}),
and since $\epsilon^{-1}\ll\eta^{-1}\ll n$, it follows that $\beta\left(E_{2}\right)>\frac{\epsilon}{8\log|\Lambda|}$.

Fix $\omega\in E_{2}$, and let $F$ be the set of all $\omega'\in E_{0}$
such that
\[
d\left(L\left(\omega\right),L\left(g_{\omega'|_{n}}\right)\right)\le2^{-n\left(2\chi-\eta\right)}.
\]
Since $\omega\in E_{2}\subset E_{1}$, we have $\beta_{\omega}(F)>1-2\eta$.
Thus, from $\frac{1}{n}H\left(\Pi_{n}\beta_{\omega}\right)\ge\epsilon/4$,
by almost-convexity of entropy, from (\ref{eq:ub on Shannon ent of Pi_n xi}),
and since $\epsilon^{-1}\ll\eta^{-1}\ll n$, we obtain $\frac{1}{n}H\left(\Pi_{n}\left(\beta_{\omega}\right)_{F}\right)\ge\epsilon/8$.

By Lemma \ref{lem:part of bd cardin}, there exists a Borel partition
$\mathcal{Q}$ of $B\left(L\left(\omega\right),2^{-n\left(2\chi-\eta\right)}\right)$
such that $\log\left|\mathcal{Q}\right|=O\left(\eta n\right)$ and
$\mathrm{diam}\left(Q\right)\le2^{-n\left(2\chi+2\eta\right)}$ for
all $Q\in\mathcal{Q}$. Hence, by the definition of $F$, there exist
$m\in\mathbb{Z}_{>0}$ and a Borel partition $\left\{ Z_{1},...,Z_{m}\right\} $
of $F$ such that $\log m=O\left(\eta n\right)$, and for all $1\le j\le m$
and $\omega',\omega''\in Z_{j}$,
\begin{equation}
\frac{1}{2}\le\frac{\Vert g_{\omega'|_{n}}\Vert_{\mathrm{op}}}{\Vert g_{\omega''|_{n}}\Vert_{\mathrm{op}}}\le2\text{ and }d\left(L\left(g_{\omega'|_{n}}\right),L\left(g_{\omega''|_{n}}\right)\right)\le2^{-n\left(2\chi+2\eta\right)}.\label{eq:prop of Z_j}
\end{equation}

Let $1\le j\le m$, and note that from $Z_{j}\subset E_{0}$ and (\ref{eq:prop of Z_j}),
\[
d\left(L\left(g_{\omega'|_{n}}\right),L\left(g_{\omega''|_{n}}\right)\right)\le\Vert g_{\omega'|_{n}}\Vert_{\mathrm{op}}^{-2}\text{ for all }\omega',\omega''\in Z_{j}.
\]
Hence, by Lemma \ref{lem:dist est in G},
\begin{equation}
\mathrm{diam}\left(\Pi_{n}\left(Z_{j}\right)\right)\le R\text{ for every }1\le j\le m,\label{eq:diam(Pi_n(F_j)) <=00003D R}
\end{equation}
where $R>1$ is the global constant obtained in Lemma \ref{lem:dist est in G}.

Let $0<r<1$ be the constant obtained in Theorem \ref{thm:ent inc},
and suppose that $R,r^{-1}\ll\eta^{-1}$. By (\ref{eq:diam(Pi_n(F_j)) <=00003D R})
and Lemma \ref{lem:part of bd cardin}, for each $1\le j\le m$ there
exist $l_{j}\in\mathbb{Z}_{>0}$ and a Borel partition $\left\{ Z_{j,1},...,Z_{j,l_{j}}\right\} $
of $Z_{j}$ such that $\log l_{j}=O_{R,r}(1)$ and $\mathrm{diam}\left(\Pi_{n}\left(Z_{j,i}\right)\right)\le r$
for all $1\le i\le l_{j}$. Setting
\[
\mathcal{Z}:=\left\{ Z_{j,i}\::\:1\le j\le m\text{ and }1\le i\le l_{j}\right\} ,
\]
it holds that $\mathcal{Z}$ is a Borel partition of $F$ with $\log\left|\mathcal{Z}\right|=O_{R,r}\left(\eta n\right)$
and
\begin{equation}
\mathrm{diam}\left(\Pi_{n}(Z)\right)\le r\text{ for }Z\in\mathcal{Z}.\label{eq:diam (Pi_n(Z))<=00003Dr}
\end{equation}

From $\frac{1}{n}H\left(\Pi_{n}\left(\beta_{\omega}\right)_{F}\right)\ge\epsilon/8$
and $\log\left|\mathcal{Z}\right|=O_{R,r}\left(\eta n\right)$, by
the almost-convexity of entropy (see Section \ref{subsec:Entropy}),
and since $R,r^{-1},\epsilon^{-1}\ll\eta^{-1}$,
\begin{equation}
\sum_{Z\in\mathcal{Z}}\frac{\beta_{\omega}(Z)}{\beta_{\omega}(F)}\frac{1}{n}H\left(\Pi_{n}\left(\beta_{\omega}\right)_{Z}\right)\ge\frac{\epsilon}{16}.\label{eq:sum over Z in Z of Shannon ent}
\end{equation}
Let $\mathcal{Z}_{1}$ be the set of all $Z\in\mathcal{Z}$ such that
$\beta_{\omega}(Z)>0$ and $\frac{1}{n}H\left(\Pi_{n}\left(\beta_{\omega}\right)_{Z}\right)\ge\frac{\epsilon}{32}$.
From (\ref{eq:ub on Shannon ent of Pi_n xi}) and (\ref{eq:sum over Z in Z of Shannon ent}),
we obtain that $\left(\beta_{\omega}\right)_{F}\left(\bigcup\mathcal{Z}_{1}\right)\ge\frac{\epsilon}{32\log|\Lambda|}$.
Thus, since $\beta_{\omega}(F)>1-2\eta$, we have $\beta_{\omega}\left(\bigcup\mathcal{Z}_{1}\right)\ge\frac{\epsilon(1-2\eta)}{32\log|\Lambda|}$.

Let $Z\in\mathcal{Z}_{1}$ be given, set $\theta:=\Pi_{n}\left(\beta_{\omega}\right)_{Z}$,
and fix some $g\in\mathrm{supp}(\theta)$. From (\ref{eq:diam (Pi_n(Z))<=00003Dr}),
it follows that $\mathrm{supp}\left(g^{-1}\theta\right)\subset B\left(1_{\mathrm{G}},r\right)$.
Moreover, since $g\in\Pi_{n}(Z)\subset\Pi_{n}(F)\subset\Pi_{n}(E_{0})$,
we have $\left|\frac{1}{n}\log\Vert g\Vert_{\mathrm{op}}-\chi\right|<\eta$.
Hence, by Lemma \ref{lem:scaling of ent in CP^1} and since $\theta.\mu=g\left(\left(g^{-1}\theta\right).\mu\right)$,
\begin{equation}
\widehat{H}\left(\theta.\mu\right)\ge\frac{1}{Mn}H\left(\left(g^{-1}\theta\right).\mu,\mathcal{D}_{Mn}\mid\mathcal{D}_{0}\right)-O\left(\eta\right).\label{eq:lb by scaling of ent}
\end{equation}

Note that by (\ref{eq:ub on card of dyd part of CP^1}) and since
$\eta^{-1}\ll n$,
\begin{equation}
\frac{1}{Mn}H\left(\xi,\mathcal{D}_{0}\right)\le\eta\text{ for all }\xi\in\mathcal{M}\left(\mathbb{CP}^{1}\right).\label{eq:ent wrt D_0 is small in main}
\end{equation}

Let $\delta=\delta\left(\frac{\epsilon}{32M}\right)\in(0,1)$ be as
obtained in Theorem \ref{thm:ent inc}. Since $\epsilon^{-1},M\ll\eta^{-1}$,
we may assume that $\delta^{-1}\ll\eta^{-1}$. From (\ref{eq:by Diophantine})
and since $d_{\mathrm{G}}$ is left invariant, it follows that $d\left(g_{1},g_{2}\right)\ge c^{n}$
for all distinct $g_{1},g_{2}\in\mathrm{supp}\left(g^{-1}\theta\right)$.
Thus, from (\ref{eq:def prop of M}) and since $\frac{1}{n}H\left(g^{-1}\theta\right)=\frac{1}{n}H\left(\theta\right)\ge\frac{\epsilon}{32}$,
we obtain $\frac{1}{Mn}H\left(g^{-1}\theta,\mathcal{D}_{Mn}\right)\ge\frac{\epsilon}{32M}$.
From this, from $\mathrm{supp}\left(g^{-1}\theta\right)\subset B\left(1_{\mathrm{G}},r\right)$,
since $\dim\mu<2$, by Theorem \ref{thm:ent inc}, and since $\epsilon^{-1},M\ll n$,
\[
\frac{1}{Mn}H\left(\left(g^{-1}\theta\right).\mu,\mathcal{D}_{Mn}\right)\ge\dim\mu+\delta.
\]
Combining this with (\ref{eq:lb by scaling of ent}) and (\ref{eq:ent wrt D_0 is small in main}),
and using $\delta^{-1}\ll\eta^{-1}$, we have thus shown that
\begin{equation}
\widehat{H}\left(\Pi_{n}\left(\beta_{\omega}\right)_{Z}.\mu\right)\ge\dim\mu+\delta/2\text{ for all }Z\in\mathcal{Z}_{1}.\label{eq:lb of cond ent for all Z in Z_1}
\end{equation}

Next, we derive a lower bound for the left-hand side of (\ref{eq:lb of cond ent for all Z in Z_1}),
which is valid for all $Z\in\mathcal{Z}$. Let $g\in\Pi_{n}\left(E_{0}\right)$
be given. By applying Lemma \ref{lem:scaling of ent in CP^1} with
$\theta=\delta_{1_{\mathrm{G}}}$,
\[
\widehat{H}\left(g\mu\right)\ge\frac{1}{Mn}H\left(\mu,\mathcal{D}_{Mn}\mid\mathcal{D}_{0}\right)-O\left(\eta\right).
\]
Hence, by Lemma \ref{lem:from dim to ent}, from (\ref{eq:ent wrt D_0 is small in main}),
and since $\eta^{-1}\ll n$,
\begin{equation}
\widehat{H}\left(g\mu\right)\ge\dim\mu-O\left(\eta\right)\text{ for all }g\in\Pi_{n}\left(E_{0}\right).\label{eq:lb of cond ent all g in Pi_n(E_0)}
\end{equation}
Consequently, by the concavity of conditional entropy and since $F\subset E_{0}$,
\begin{equation}
\widehat{H}\left(\Pi_{n}\left(\beta_{\omega}\right)_{Z}.\mu\right)\ge\dim\mu-O\left(\eta\right)\text{ for all }Z\in\mathcal{Z}\text{ with }\beta_{\omega}(Z)>0.\label{eq:lb of cond ent for all Z in Z}
\end{equation}

From (\ref{eq:lb of cond ent for all Z in Z_1}) and (\ref{eq:lb of cond ent for all Z in Z}),
from $\beta_{\omega}(F)>1-2\eta$ and $\beta_{\omega}\left(\bigcup\mathcal{Z}_{1}\right)\ge\frac{\epsilon(1-2\eta)}{32\log|\Lambda|}$,
and by concavity,
\begin{eqnarray*}
\widehat{H}\left(\Pi_{n}\beta_{\omega}.\mu\right) & \ge & \sum_{Z\in\mathcal{Z}}\beta_{\omega}(Z)\cdot\widehat{H}\left(\Pi_{n}\left(\beta_{\omega}\right)_{Z}.\mu\right)\\
 & \ge & \beta_{\omega}\left(\bigcup\mathcal{Z}_{1}\right)\left(\dim\mu+\delta/2\right)+\beta_{\omega}\left(F\setminus\bigcup\mathcal{Z}_{1}\right)\left(\dim\mu-O\left(\eta\right)\right)\\
 & \ge & \dim\mu+\frac{\epsilon\delta}{64\log|\Lambda|}-O\left(\eta\right),
\end{eqnarray*}
which holds for all $\omega\in E_{2}$. Additionally, from (\ref{eq:lb of cond ent all g in Pi_n(E_0)})
and by concavity, for each $\omega\in E_{1}$
\[
\widehat{H}\left(\Pi_{n}\beta_{\omega}.\mu\right)\ge\beta_{\omega}\left(E_{0}\right)\widehat{H}\left(\Pi_{n}\left(\beta_{\omega}\right)_{E_{0}}.\mu\right)\ge\dim\mu-O\left(\eta\right).
\]

From the last two formulas, by (\ref{eq:ineq to contradict in main pf}),
and since $\beta\left(E_{1}\right)>1-\eta$ and $\beta\left(E_{2}\right)>\frac{\epsilon}{8\log|\Lambda|}$,
\begin{eqnarray*}
\dim\mu & \ge & \int_{E_{1}\setminus E_{2}}\widehat{H}\left(\left(\Pi_{n}\beta_{\omega}\right).\mu\right)\:d\beta(\omega)+\int_{E_{2}}\widehat{H}\left(\left(\Pi_{n}\beta_{\omega}\right).\mu\right)\:d\beta(\omega)-\eta\\
 & \ge & \beta\left(E_{1}\setminus E_{2}\right)\left(\dim\mu-O\left(\eta\right)\right)+\beta\left(E_{2}\right)\left(\dim\mu+\frac{\epsilon\delta}{64\log|\Lambda|}-O\left(\eta\right)\right)-\eta\\
 & \ge & \dim\mu+\frac{\epsilon}{8\log|\Lambda|}\cdot\frac{\epsilon\delta}{64\log|\Lambda|}-O\left(\eta\right).
\end{eqnarray*}
Since $\epsilon^{-1},\delta^{-1}\ll\eta^{-1}$, the last formula leads
to a contradiction, which completes the proof of the theorem.
\end{proof}
\appendix

\section{\label{sec:appendix}Exact dimensionality and Ledrappier--Young
formula}

The purpose of this appendix is to derive Theorem \ref{thm:exact dim of mu and LY}
from the results of \cite{rapaport2020exact}. Recall that $\mathcal{B}_{\mathbb{CP}^{1}}$
denotes the Borel $\sigma$-algebra of $\mathbb{CP}^{1}$, that we
set
\[
\Delta:=H\left(\beta,\mathcal{P}_{1}\mid L^{-1}\mathcal{B}_{\mathbb{CP}^{1}}\right),
\]
and that $\left\{ \beta_{\omega}\right\} _{\omega\in\Lambda^{\mathbb{N}}}$
denotes the disintegration of $\beta$ with respect to $L^{-1}\mathcal{B}_{\mathbb{CP}^{1}}$.
For the reader's convenience, we repeat the statement of Theorem \ref{thm:exact dim of mu and LY}.
\begin{thm*}
The measure $\mu$ is exact dimensional with $\dim\mu=\frac{H(p)-\Delta}{2\chi}$.
Moreover,
\begin{equation}
\underset{n\to\infty}{\lim}\frac{1}{n}H\left(\beta_{\omega},\mathcal{P}_{n}\right)=\Delta\text{ for }\beta\text{-a.e. }\omega.\label{eq:local dim of slices of beta}
\end{equation}
\end{thm*}
\begin{rem*}
The derivation of Theorem \ref{thm:exact dim of mu and LY} from \cite{rapaport2020exact}
is somewhat technical. An explanation of why this is necessary is
given in the paragraph at the end of Section \ref{subsec:Exact-dimensionality-and LY formula}.
\end{rem*}
\begin{proof}
Let $T:\mathbb{C}^{2}\rightarrow\mathbb{R}^{4}$ denote the natural
identification between $\mathbb{C}^{2}$ and $\mathbb{R}^{4}$; that
is,
\[
T(x_{1}+x_{2}i,x_{3}+x_{4}i)=(x_{1},x_{2},x_{3},x_{4})\:\text{ for }x_{1},x_{2},x_{3},x_{4}\in\mathbb{R}.
\]
Let $\wedge^{2}\mathbb{R}^{4}$ denote the real vector space of alternating
$2$-forms on the dual of $\mathbb{R}^{4}$, and let $\rho:\mathrm{G}\rightarrow\mathrm{GL}\left(\wedge^{2}\mathbb{R}^{4}\right)$
be such that
\[
\rho(g)(x\wedge y)=\left(TgT^{-1}x\right)\wedge\left(TgT^{-1}y\right)\text{ for all }g\in\mathrm{G}\text{ and }x,y\in\mathbb{R}^{4}.
\]
It is easy to verify that the Lie group representation $\rho$ descends
to an embedding of $\mathrm{PSL}\left(2,\mathbb{C}\right):=\mathrm{G}/\{\pm1_{\mathrm{G}}\}$
into $\mathrm{GL}\left(\wedge^{2}\mathbb{R}^{4}\right)$.

Let $X$ denote the set of vectors in $\wedge^{2}\mathbb{R}^{4}$
of the form $x\wedge T\left(iT^{-1}x\right)$ for some $0\ne x\in\mathbb{R}^{4}$,
and write $\mathbb{V}$ for the subspace of $\wedge^{2}\mathbb{R}^{4}$
spanned by $X$. It is easy to verify that $X$, and hence also $\mathbb{V}$,
is $\rho(\mathrm{G})$-invariant.

Let $\{f_{j}\}_{j=1}^{4}$ denote the standard basis of $\mathbb{R}^{4}$,
and set
\[
\zeta_{1}:=f_{1}\wedge f_{2},\:\zeta_{2}:=f_{3}\wedge f_{4},\:\zeta_{3}:=f_{1}\wedge f_{4}-f_{2}\wedge f_{3}\text{ and }\zeta_{4}:=f_{1}\wedge f_{3}+f_{2}\wedge f_{4}.
\]
It is easy to verify that $\{\zeta_{j}\}_{j=1}^{4}$ forms a basis
of $\mathbb{V}$. Using this, it is not difficult to show that $\rho(\mathrm{G})$
acts proximally and irreducibly on $\mathbb{V}$. Since $\mathrm{G}$
is connected, it follows that $\rho(\mathrm{G})$ acts strongly irreducibly
on $\mathbb{V}$.

Write $\mathrm{P}(\mathbb{V})$ for the projective space of $\mathbb{V}$.
Since $\rho(\mathrm{G})$ acts strongly irreducibly and proximally
on $\mathbb{V}$, it follows from Lemma \ref{lem:zariski density}
and \cite[Lemma 6.23]{BQ} that $\rho(\mathrm{S}_{\mathcal{G}})$
also acts strongly irreducibly and proximally on $\mathbb{V}$. Hence,
setting
\[
\theta:=\sum_{i\in\Lambda}p_{i}\delta_{\rho(g_{i})}\in\mathcal{M}\left(\mathrm{GL}\left(\wedge^{2}\mathbb{R}^{4}\right)\right),
\]
there exists a unique $\mu'\in\mathcal{M}\left(\mathrm{P}(\mathbb{V})\right)$
which is $\theta$-stationary. By \cite[Theorem 1.1]{rapaport2020exact},
the measure $\mu'$ is exact dimensional. From the $\rho(\mathrm{G})$-invariance
of $X$, it follows that the compact set $\mathrm{P}(X):=\left\{ \phi\mathbb{R}\::\:\phi\in X\right\} $
is also $\rho(\mathrm{G})$-invariant. Thus, by the uniqueness of
$\mu'$, it follows that $\mu'$ is supported on $\mathrm{P}(X)$.

Let $F:\mathbb{CP}^{1}\rightarrow\mathrm{P}(X)$ be such that $F\left(z\mathbb{C}\right)=T(z)\wedge T(iz)\mathbb{R}$
for $z\mathbb{C}\in\mathbb{CP}^{1}$. It is easy to verify that $F$
is well defined, and that it is a diffeomorphism of $\mathbb{CP}^{1}$
onto $\mathrm{P}(X)$. Moreover, $F\left(gz\mathbb{C}\right)=\rho(g)\left(F\left(z\mathbb{C}\right)\right)$
for each $g\in\mathrm{G}$ and $z\mathbb{C}\in\mathbb{CP}^{1}$. Thus,
$F$ is an isomorphism between the action of $\mathrm{PSL}\left(2,\mathbb{C}\right)$
on $\mathbb{CP}^{1}$ and the action of $\rho(\mathrm{G})$ on $\mathrm{P}(X)$.
In particular, $\mu'=F\mu$, and so, since $\mu'$ is exact dimensional,
we obtain that $\mu$ is also exact dimensional with $\dim\mu=\dim\mu'$.

The standard Euclidean inner product on $\mathbb{R}^{4}$ induces
an inner product on $\wedge^{2}\mathbb{R}^{4}$ in a natural way (see
\cite[Section III.5]{BL}), which restricts to an inner product on
$\mathbb{V}$. Given a line $\ell\in\mathrm{P}(\mathbb{V})$, write
$\ell^{\perp}$ for the orthogonal complement of $\ell$ in $\mathbb{V}$,
and let $\pi_{\ell^{\perp}}:\mathbb{V}\rightarrow\mathbb{V}$ denote
the orthogonal projection onto $\ell^{\perp}$.

Since $\rho(\mathrm{S}_{\mathcal{G}})$ acts strongly irreducibly
and proximally on $\mathbb{V}$, there exists a unique $\lambda\in\mathcal{M}\left(\mathrm{P}(\mathbb{V})\right)$
which is stationary with respect to $\sum_{i\in\Lambda}p_{i}\delta_{\rho(g_{i})^{-1}}$.
Additionally, let $L':\Lambda^{\mathbb{N}}\rightarrow\mathrm{P}(\mathbb{V})$
denote the Furstenberg boundary map associated to $\theta$ (see \cite[Proposition 4.7]{BQ}),
write $\mathcal{B}_{\mathrm{P}(\mathbb{V})}$ for the Borel $\sigma$-algebra
of $\mathrm{P}(\mathbb{V})$, and set
\[
\mathrm{H}_{1}:=\int H\left(\beta,\mathcal{P}_{1}\mid L'^{-1}\pi_{\ell^{\perp}}^{-1}\mathcal{B}_{\mathrm{P}(\mathbb{V})}\right)\:d\lambda(\ell)\text{ and }\mathrm{H}_{2}:=H\left(\beta,\mathcal{P}_{1}\mid L'^{-1}\mathcal{B}_{\mathrm{P}(\mathbb{V})}\right).
\]
Given $\ell\in\mathrm{P}(\mathbb{V})$, note that $\pi_{\ell^{\perp}}\circ L'$
defines a Borel map on $\Lambda^{\mathbb{N}}$ outside a set of zero
$\beta$-measure, and so $\mathrm{H}_{1}$ is well defined.

Given an orthonormal basis $\left\{ z,w\right\} $ of $\mathbb{C}^{2}$,
it is easy to verify that
\[
\left\{ \begin{array}{c}
T(z)\wedge T(iz),\:\frac{1}{\sqrt{2}}\left(T(z)\wedge T(iw)-T(iz)\wedge T(w)\right),\\
T(w)\wedge T(iw),\:\frac{1}{\sqrt{2}}\left(T(z)\wedge T(w)+T(iz)\wedge T(iw)\right)
\end{array}\right\} 
\]
forms an orthonormal basis of $\mathbb{V}$. Using this, and since
the Lyapunov exponents corresponding to $\sum_{i\in\Lambda}p_{i}\delta_{g_{i}}$
are $\chi$ and $-\chi$, it is not difficult to show that the Lyapunov
exponents corresponding to $\theta$ are $2\chi,0,0,-2\chi$. Hence,
by \cite[Theorem 1.3]{rapaport2020exact},
\begin{equation}
\dim\mu'=\frac{H(p)-\mathrm{H}_{1}}{2\chi}+\frac{\mathrm{H}_{1}-\mathrm{H}_{2}}{4\chi}.\label{eq:LY for mu'}
\end{equation}

Let us next show that in fact $\mathrm{H}_{1}=\mathrm{H}_{2}$. Given
$\ell\in\mathrm{P}(\mathbb{V})$, write $\{\mu'_{\ell,Z}\}_{Z\in\mathrm{P}(\mathbb{V})}$
for the disintegration of $\mu'$ with respect to $\pi_{\ell^{\perp}}^{-1}\mathcal{B}_{\mathrm{P}(\mathbb{V})}$.
By \cite[Theorem 1.3]{rapaport2020exact}, it follows that for $\lambda$-a.e.
$\ell$ and $\mu'$-a.e. $Z$ the measure $\mu'_{\ell,Z}$ is exact
dimensional with dimension $\frac{1}{4\chi}\left(\mathrm{H}_{1}-\mathrm{H}_{2}\right)$.
Thus, in order to show that $\mathrm{H}_{1}=\mathrm{H}_{2}$, it suffices
to prove that $\dim\mu'_{\ell,Z}=0$ for $\lambda\times\mu'$-a.e.
$\left(\ell,Z\right)$.

Recall the basis $\left\{ \zeta_{j}\right\} _{j=1}^{4}$ defined above.
Fix $\ell\in\mathrm{P}(\mathbb{V})$, set
\[
W:=\zeta_{2}+\mathrm{span}\{\zeta_{1},\zeta_{3},\zeta_{4}\},
\]
and let,
\[
S:=\left\{ (x^{2}+y^{2})\zeta_{1}+\zeta_{2}+x\zeta_{3}+y\zeta_{4}\::\:x,y\in\mathbb{R}\right\} .
\]
For $x,y\in\mathbb{R}$,
\[
F\left((x+yi,1)\mathbb{C}\right)=\left((x^{2}+y^{2})\zeta_{1}+\zeta_{2}+x\zeta_{3}+y\zeta_{4}\right)\mathbb{R}.
\]
Thus, setting $N:=F\left((1,0)\mathbb{C}\right)$, each line $Z\in\mathrm{P}(X)\setminus\left\{ N\right\} $
intersects $S$ at precisely one point.

Given $Q\in\mathrm{P}\left(\ell^{\perp}\right):=\left\{ \ell'\in\mathrm{P}(\mathbb{V})\::\:\ell'\subset\ell^{\perp}\right\} $,
the set $\pi_{\ell^{\perp}}^{-1}(Q)$ is a $2$-dimensional linear
subspace of $\mathbb{V}$. Since $0\notin W$, it follows that $\pi_{\ell^{\perp}}^{-1}(Q)\cap W$
is either an affine line or the empty set. Moreover, it is easy to
see that an affine line can intersect the translated paraboloid $S\subset W$
in at most $2$ points. We have thus shown that,
\[
\#\left\{ Z\in\mathrm{P}(X)\setminus\left\{ N\right\} \::\:\pi_{\ell^{\perp}}(Z)=Q\right\} \le2\text{ for all }Q\in\mathrm{P}\left(\ell^{\perp}\right).
\]
Since $\mu'$ is supported on $\mathrm{P}(X)$, this clearly implies
that $\dim\mu'_{\ell,Z}=0$ for $\mu'$-a.e. $Z$, which gives $\mathrm{H}_{1}=\mathrm{H}_{2}$.

Since $F$ is an isomorphism between actions,
\begin{equation}
L'(\omega)=F\circ L(\omega)\text{ for }\beta\text{-a.e. }\omega,\label{eq:L' factors via L}
\end{equation}
which implies $\mathrm{H}_{2}=\Delta$. From this, $\mathrm{H}_{1}=\mathrm{H}_{2}$,
and (\ref{eq:LY for mu'}), we get
\[
\dim\mu=\dim\mu'=\frac{H(p)-\Delta}{2\chi}.
\]
Moreover, from (\ref{eq:L' factors via L}) it also follows that the
disintegration of $\beta$ with respect to $L^{-1}\mathcal{B}_{\mathbb{CP}^{1}}$,
which we have denoted by $\left\{ \beta_{\omega}\right\} _{\omega\in\Lambda^{\mathbb{N}}}$,
equals almost surely the disintegration of $\beta$ with respect to
$L'^{-1}\mathcal{B}_{\mathrm{P}(\mathbb{V})}$. From this, $\mathrm{H}_{2}=\Delta$,
and \cite[Lemma 4.4]{rapaport2020exact}, we obtain (\ref{eq:local dim of slices of beta}),
which completes the proof of the theorem.
\end{proof}

\subsubsection*{\textbf{\emph{Acknowledgment}}}

We thank Boris Solomyak and Adam \'{S}piewak for helpful discussions
in the early stages of this work. We also thank François Ledrappier
and Ilya Gekhtman for helpful remarks on an earlier version of the
paper.

This research was supported by the Israel Science Foundation (grant
No. 619/22). AR received support from the Horev Fellowship at the
Technion -- Israel Institute of Technology.

\bibliographystyle{plain}
\bibliography{../../bibfile.bib}

$\newline$\textsc{Ariel Rapaport, Department of Mathematics, Technion, Haifa, Israel}$\newline$\textit{E-mail: }
\texttt{arapaport@technion.ac.il}

$\newline$\textsc{Haojie Ren, Department of Mathematics, Technion, Haifa, Israel}$\newline$\textit{E-mail: }
\texttt{hjren@campus.technion.ac.il}
\end{document}